\documentclass[a4paper,reqno,oneside]{amsart}
\pdfoutput=1 %on one occasion arXiv did not produce a PDF without this

\usepackage[T1]{fontenc}
\usepackage[utf8]{inputenc}
\usepackage[main=english,german]{babel}
\usepackage{quiver}
\usepackage{dsfont}
\usepackage{trimclip}
\usepackage{amscd,amssymb,amsthm,amsfonts}
\usepackage{mathtools}
\usepackage{stmaryrd} % needed for \sslash
\usepackage{graphicx}
\usepackage[inkscapepath=svgsubdir]{svg}
\usepackage{tikz}
\usepackage{tikz-cd}
\usetikzlibrary{calc,matrix,arrows,arrows.meta,decorations.pathmorphing,positioning,braids}
\usepackage[a4paper]{geometry}
\usetikzlibrary{babel}
\usepackage{hyperref}
\hypersetup{colorlinks=true,pageanchor=false,linkcolor=blue, citecolor=brown, urlcolor=red, breaklinks=true} %arXiv needs breaklinks=true explicitly
\usepackage{extarrows}

\allowdisplaybreaks

\tikzset{
  strand/.style={line width=1.2pt, line cap=round},
  unit/.style={circle, draw, line width=0.9pt, inner sep=1.2pt, fill=white},
  eq/.style={font=\Large},
  open dot/.style={circle,draw,fill=white,inner sep=1pt},
  closed dot/.style={circle,fill=black,inner sep=1pt},
coupon/.style={draw,fill=white,rounded corners=1pt,minimum width=6pt,minimum height=4pt,inner sep=1pt}
}

\newtheorem{theorem}{Theorem}[section]
\newtheorem{corollary}[theorem]{Corollary}
\newtheorem{proposition}[theorem]{Proposition}
\newtheorem{lemma}[theorem]{Lemma}

\theoremstyle{definition}
\newtheorem{definition}[theorem]{Definition}
\newtheorem{example}[theorem]{Example}

\theoremstyle{remark}
\newtheorem{remark}[theorem]{Remark}

\newcommand{\Z}{\mathbb{Z}}

\newcommand{\C}{\mathbb{C}}

\renewcommand{\epsilon}{\varepsilon}
\renewcommand{\theta}{\vartheta}
\renewcommand{\phi}{\varphi}
\renewcommand{\Gamma}{\varGamma}
\renewcommand{\Sigma}{\varSigma}
\newcommand{\mathbbm}{\mathds}
\newcommand{\Aut}{\mathrm{Aut}}
\newcommand{\Hom}{\mathrm{Hom}}
\newcommand{\End}{\mathrm{End}}
\newcommand{\id}{\mathrm{id}}

\newcommand\arxiv[2]{\href{https://arXiv.org/abs/#1}{\texttt{arXiv:\allowbreak #1} #2}}
\newcommand\doi[2]{\href{https://doi.org/#1}{#2}}

\numberwithin{equation}{section}

\begin{document}

\vspace*{-2em}

\title{Crossed-module crossed braided categories}

\author[A.M.~Gainutdinov]{Azat M.~Gainutdinov$^{1}$}
\author[I.~Runkel]{Ingo Runkel$^{2}$}
\author[B.~Wang]{Bangxin Wang$^{2}$}

\vspace*{-1em}
\thanks{$^{1}$Institut Denis Poisson, CNRS, Universit\'e de Tours, Parc de Grandmont,
37200 Tours, France. \texttt{azat.gainutdinov@cnrs.fr}}
\thanks{$^{2}$Fachbereich Mathematik, Universit\"at Hamburg, Bundesstr.\ 55,
20146 Hamburg, Germany. \texttt{ingo.runkel@uni-hamburg.de}, \;  \texttt{bangxin.wang@uni-hamburg.de}}

\begin{abstract}
For a crossed module $\chi\colon G \to H$, we introduce the notion of $\chi$-crossed braided (resp.\ ribbon) categories, where the categories are graded by group $G$ and carry an $H$-action. 
Our definition unifies and generalises several familiar notions: taking $\chi = \mathrm{id}\colon G \to G$ with the conjugation action recovers $G$-crossed braided categories; taking $\chi\colon G \to \{*\}$ for abelian $G$ yields $G$-graded braided categories; taking $\chi\colon \{*\} \to G$ leads to braided categories equipped with a $G$-action. 
The equivalence relation between $\chi$-crossed braided categories is typically finer than that between $G$-crossed braided ones.
We classify $\chi$-crossed braided structures on the category of $G$-graded vector spaces in terms of cohomological data, and give explicit examples for cyclic groups.
Given a doubly central algebra with $G$- and $H$-actions in a braided monoidal category, we define a notion of twisted-local modules and show how they give rise to $\chi$-crossed braided categories.
We furthermore give sufficient conditions so that these categories are additionally $\chi$-crossed ribbon or admit an orthogonal $G$-decomposition.
\end{abstract}

\maketitle

\vspace*{-1.5em}

\setcounter{tocdepth}{2}
{\small
\tableofcontents
}

\thispagestyle{empty}

\section{Introduction}

\subsection{Background and motivation}

Braided monoidal categories and their specialisations, particularly modular tensor categories, provide a unifying formalism to describe various structures arising in quantum topology, conformal field theory, and condensed matter theory~\cite{MS89, RT91, T94, FRS02, Hu05, Ki05, KZ22}. In various extensions of these settings, it becomes natural to consider $G$-crossed braided categories. They were first introduced in the context of homotopy quantum field theory as a generalisation of modular tensor categories adapted to manifolds with a principal $G$-bundle~\cite{Tu00, T10}, and were later studied in conformal field theory, where they describe $G$-twisted modules over
vertex operator algebras and conformal nets~\cite{Ki01a, Mu02, Ki04, Mu04, Mc19}. 
In condensed matter physics, they model symmetry-enriched topological phases, where the group $G$ encodes global symmetries that act non-trivially on anyon types~\cite{BBCW14,CGPW15}. These applications have motivated extensive study of $G$-crossed braided fusion categories~\cite{ENOM09,DGNO10,EGNO15}.

A $G$-crossed braided category is a $G$-graded monoidal category
$\mathcal C=\bigoplus_{g\in G}\mathcal C_g$ equipped with a monoidal action of $G$
by autoequivalences ${}^g(-)$, for $g\in G$,
permuting the graded components by conjugation, together with crossed
braiding isomorphisms $X\otimes Y \xrightarrow{\sim} {}^{g}Y\otimes X$ for
$X\in\mathcal C_g$ satisfying certain coherence conditions.
However, in certain applications discussed below it is important to distinguish the grading and the action groups.
Given a group $H$ acting on $G$,
one can define an $H$-crossed action on a $G$-graded monoidal category,
and one can
ask whether there is a corresponding notion of crossed braiding in this setting.
This question is the main motivation for the present paper, and it leads us to introduce
the notion of $\chi$-crossed braided (resp.\ ribbon) categories for a crossed module $\chi \colon G\to H$. Interestingly, these new structures result in finer equivalence relations compared to the $G$-crossed case, e.g.\ equivalent $G$-crossed braidings may become inequivalent if regarded as $\chi$-crossed.

\subsection{Main results}
For a fixed group action of $H$ on $G$, we introduce the notion of an $H$-crossed $G$-graded monoidal category.
Given in addition a crossed module $\chi\colon G \to H$ with the prescribed $H$-action on $G$, we further define a $\chi$-crossed braided (resp.\ ribbon) structure.
Our definition unifies and generalises several familiar notions: taking $\chi = \mathrm{id}\colon G \to G$ with the conjugation action recovers $G$-crossed braided categories; taking $\chi\colon G \to \{*\}$ for abelian $G$ yields $G$-graded braided categories; taking $\chi\colon \{*\} \to G$ leads to braided categories equipped with a $G$-action.

For a fixed group action of $H$ on $G$ (resp.\ a crossed module $\chi$), we classify all $H$-crossed $G$-graded monoidal (resp.\ $\chi$-crossed braided) structures on the category of $G$-graded vector spaces in terms of cohomological data.

Given an algebra $A$ with two central structures in a monoidal category,
we define a notion of twisted-local $A$-modules, extending the framework in~\cite{Ki01a,Ki01b}. If $A$ is separable, and simple as an $A$-module, we show that this yields a graded decomposition of the category of $A$-modules.
In the braided setting,
we show that categories of these twisted-local modules give rise to $\chi$-crossed braided categories and present explicit examples. Furthermore, if the ambient category is ribbon and $A$ is $\Delta$-separable Frobenius, the twisted-local modules form a $\chi$-crossed ribbon category.

\medskip

In the following we introduce the main players and describe the above results in more detail.

\subsubsection{$H$-Crossed $G$-graded monoidal categories}
Let $G$ and $H$ be groups. Fix an action of $H$
on $G$ by group automorphisms,
and denote the action of $h \in H$ on $g \in G$ by $h\rhd g$.  
An \textit{$H$-crossed $G$-graded monoidal category $\mathcal{C}$} is an additive $G$-graded monoidal category
\begin{equation}
    \mathcal{C} \; = \; \bigoplus_{g \in G}\mathcal{C}_g
\end{equation}
equipped with a monoidal action of the discrete category $\underline{H}$ associated to $H$ via additive monoidal autoequivalences
\begin{equation}
    \phi\colon \underline{H} \longrightarrow \mathrm{Aut}_\otimes(\mathcal C)
\end{equation}
such that
\begin{equation}
    ^h\mathcal{C}_g \coloneqq \phi(h)\big(\mathcal{C}_g\big) \subset \mathcal{C}_{h \rhd g} \qquad \textup{ for any } (h,g) \in H \times G \ ;
\end{equation}
see Definition~\ref{H-crossed action def} and the unpacked formulation below it. 
If $H=G$ acts by conjugation, this recovers the notion of $G$-crossed categories as defined e.g.\ in~\cite[Ch.\,8.24]{EGNO15} or~\cite{Ga16}.

We define functors and natural transformations between $H$-crossed $G$-graded monoidal categories in Definitions~\ref{crossed functor} and~\ref{chi-crossed natural transformation}, respectively,
and give a strictification result in Remark~\ref{rem: coherence}.
Fix a ground field $\mathbb{K}$. In Section~\ref{H-action on VecG}, for a fixed $H$-action on $G$,
we classify equivalence classes of
$H$-crossed $G$-graded monoidal structures on $\mathrm{vec}_G$, the category of finite-dimensional $G$-graded $\mathbb{K}$-vector spaces, 
in terms of group cohomology. Denote by
\begin{equation}
    \mathbf{C}^{p,q} \coloneqq C^p\big(H;C^q(G;\mathbb{K}^\times)\big)
\end{equation}
the Lyndon–Hochschild–Serre double complex associated to the prescribed $H$-action on $G$~\cite[Ch.\,7]{Br82}. Delete the $q=0$ row of $\mathbf{C}^{p,q}$. We denote the total complex of the truncated bicomplex and its cohomology by $C^n_{\mathrm{tot};1}(G \rtimes_\rhd H;\mathbb{K}^\times)$ and $H^n_{\mathrm{tot};1}(G \rtimes_\rhd H;\mathbb{K}^\times)$.

\begin{theorem}[Thm.~\ref{thm:VecG-Hcrossed-classification-H3tot}]\label{thm: introduction}
For a fixed action $\rhd$ of $H$ on $G$,
$H$-crossed $G$-graded monoidal structures on $\mathrm{vec}_G$ are classified up to equivalence by
$H^3_{\mathrm{tot};1}(G \rtimes_\rhd H;\mathbb{K}^\times)$.
\end{theorem}
We further compute several concrete examples with $G = \Z/4\Z$ in Section~\ref{Z/4 example}.

\medskip

In the case that the categories are linear over a field, we define fibre products of $H$-crossed $G$-graded monoidal categories $\mathcal C$ and $\mathcal D$ in Definition~\ref{fibre product def}, with the underlying $G$-graded category structure given by
\begin{equation}
    \mathcal C \boxtimes_G \mathcal D
    \coloneqq
    \bigoplus_{g\in G}
    \bigl(\mathcal C_g \boxtimes \mathcal D_g\bigr) \ .
\end{equation}
We show that $\mathcal C \boxtimes_G \mathcal D$ is again $H$-crossed $G$-graded monoidal. 
The point of introducing this product is that any $H$-crossed $G$-graded monoidal category
$\mathcal C$ can now be twisted by a class in $H^3_{\mathrm{tot};1}(G \rtimes_\rhd H;\mathbb{K}^\times)$ by forming its fibre product with
$\mathcal D=\mathrm{vec}_G$, equipped with the
$H$-crossed $G$-graded monoidal structure from 
Theorem~\ref{thm: introduction}.

\subsubsection{$\chi$-Crossed ribbon categories}

Let $\chi\colon G \to H$
be a crossed module with the prescribed $H$-action $\rhd$ on $G$, that is, $\chi$ is a group homomorphism and for all $h \in H$, $g,g' \in G$, 
\begin{equation}
    h \chi(g) h^{-1} = \chi({h \rhd g}) \ , \qquad gg'g^{-1} = \chi(g) \rhd g' \ ,
\end{equation}
see Definition~\ref{crossed module def}.
A \textit{$\chi$-crossed braided category} is an $H$-crossed $G$-graded monoidal category $\mathcal{C}$ endowed with a family of $\chi$-crossed braidings
\begin{equation}
    \{C_{U,V}\colon U \otimes V \to {^{\chi(|U|)}V} \otimes U\}_{U, V \in \mathcal{C}_{\mathrm{hom}}}
\end{equation}
satisfying appropriate compatibility conditions (Definition~\ref{crossed braiding def}). Here, $\mathcal{C}_{\mathrm{hom}}$ denotes the full subcategory consisting of homogeneous objects, and $|U|$ denotes the $G$-degree of $U$.
For $\mathcal{C}$ rigid, a $\chi$-crossed ribbon structure is given by a family of balancing isomorphisms
\begin{equation}
    \{\vartheta_U\colon U \to {^{\chi(|U|)}U}\}_{U \in \mathcal{C}_{\mathrm{hom}}}
\end{equation}
subject to the conditions in Definition~\ref{crossed balanced}.
We introduce functors between $\chi$-crossed braided (resp.\ ribbon) categories in Definition~\ref{crossed braided functor}.
Moreover, we state in Proposition~\ref{prop:chi-crossed-br-strict} that every $\chi$-crossed braided (resp.\ ribbon) category is equivalent to a strict one.
As a result, we prove in Corollary~\ref{cor:pivotal} that every $\chi$-crossed ribbon category is pivotal.

Fix a crossed module $\chi$ and a ground field $\mathbb{K}$. In Section~\ref{VecG chi crossed classification} we classify $\chi$-crossed braided structures on $\mathrm{vec}_G$.
Similar results are obtained in~\cite{Na11} in the context of classifying $H$-crossed braided structures on pointed fusion categories.
We compute several concrete examples with $G = \Z/4\Z$ in Section~\ref{different equiv relation}. These examples show that the essential distinction between
$\chi$-crossed braided categories and the usual notion of $G$-crossed braided categories
does not lie in the form of the crossed braiding isomorphisms, but rather in
the corresponding notion of equivalence.
More precisely, let $\chi \colon G \to H$ be a surjective crossed module, and let $\mathcal C$ be a $G$-crossed category whose $G$-action factors through $H$.
By Proposition~\ref{chi crrossed implies G crossed}, every $G$-crossed braiding on $\mathcal C$ is also a $\chi$-crossed braiding.
However, it can then happen that equivalent $G$-crossed braidings on $\mathcal C$ become inequivalent when regarded as $\chi$-crossed braidings, discussed more in Remarks~\ref{equivalence remark} and~\ref{rmk: different equiv for chi}.

Finally, in Section~\ref{sec:equivar} we study $H$-equivariantisation $\mathcal{C}^H$ of a $\chi$-crossed braided/ribbon category $\mathcal{C}$ and Proposition~\ref{prop:equiv} shows that $\mathcal{C}^H$ is an ordinary braided/ribbon category.

\subsubsection{Twisted-local modules}

Let $\mathcal{C}$ be an abelian monoidal category with right exact tensor product, and let $A \in \mathcal{C}$ be an algebra object. Denote by $_A\mathcal C$ the category of $A$-modules in $\mathcal{C}$. Suppose that $A$ lifts to two commutative algebras $(\overline{A}, \overline{\sigma}), \, (\widetilde{A},\widetilde{\sigma})$ in the Drinfeld centre $Z(\mathcal{C})$.
A natural source of such algebras is provided by the full centre construction, as explained in Examples~\ref{ex:full-centre} and \ref{ex:sweedler}. Furthermore, let $\phi\colon G \to \mathrm{Aut}_\mathcal{C}(A)$ be a suitable $G$-action on $A$.

For each $g \in G$, an $A$-module $M$
with action morphism $\rho_M \colon A \otimes M \to M$
is called \textit{$g$-local} if
\begin{equation}\label{eq:g-local condition intro}
    \rho_M \circ \overline{\sigma}_M^{-1} = \rho_M \circ \widetilde{\sigma}_M^{-1} \circ (\mathrm{id}_M \otimes \phi_g) 
\end{equation}
as morphisms $M \otimes A \to M$,
see Definition~\ref{twisted local modules}. 
 If $\mathcal{C}$ is braided and we specialise $\overline{\sigma}$ and $\widetilde{\sigma}$ to be
\begin{equation}\label{eq:intro-halfbraid-special}
    \overline{\sigma}_U \coloneqq C_{A,U} \ ,\qquad
\widetilde{\sigma}_U \coloneqq C_{U,A}^{-1} \ ,
\end{equation}
where $C$ denotes the braiding of $\mathcal C$,  then~\eqref{eq:g-local condition intro} recovers the $g$-twisted condition for an $A$-module in~\cite{Ki01a,Ki01b}.

For each $g \in G$, denote by $_{A}\mathcal{C}_g$ the full subcategory of $g$-local $A$-modules in $\mathcal C$. 
We stress that an $A$-module $M$ may be both $g$-local and $k$-local for distinct $g, k \in G$, and it can happen that $_{A}\mathcal{C}_g$ is zero for some $g$. However, we can still define the external direct sum
\begin{equation}
    {}_{A}\mathcal{C}^\mathrm{gr} \coloneqq \bigoplus_{g \in G}{}_{A}\mathcal{C}_g \ ,
\end{equation}
where by definition there are no non-zero morphisms between different components.
We show in Lemma~\ref{twisted-local modules G-graded monoidal} that ${}_{A}\mathcal{C}^\mathrm{gr}$ is a $G$-graded monoidal category,
with the tensor product $\otimes_A$ defined via the half-braiding $\overline{\sigma}$; see Proposition~\ref{monoidal structure on A-mod}.

Since the subcategories $_{A}\mathcal{C}_g$ may intersect non-trivially in $_{A}\mathcal{C}$,  
the forgetful functor ${}_{A}\mathcal{C}^\mathrm{gr} \to {}_{A}\mathcal{C}$ is in general not full. This seems odd at first sight, but it is actually a feature: in Section~\ref{superalgebra example} we consider examples based on the superalgebra $\mathbb{K}[x]/(x^2)$ in the category of super vector spaces, where allowing different choices of half-braidings, as well as allowing modules to be $g$-local for different $g$ at the same time, leads to interesting results such as categories with different non-semisimple fibres.

In Theorem~\ref{thm:inner-direct-sum-for-algebra-modules} we show that if $A$ is a separable algebra and simple as an $A$-module, then, subject to some additional technical assumptions,
every object of $_A\mathcal C$ decomposes uniquely as a direct sum of $g$-local modules.
This generalises the constructions of~\cite{Ki01a,Ki01b,Mc19}.
In this case we have ${}_{A}\mathcal{C} = {}_{A}\mathcal{C}^\mathrm{gr}$.

\medskip

Let $\chi\colon G \to H$ be a crossed module and let $H$ 
act
on $A$ by $\upsilon\colon H \to \mathrm{Aut}_\mathcal{C}(A)$. 
In Proposition~\ref{structure morphism for h-action in AMod}, assuming some compatibility conditions, we define a strictly monoidal $H$-action on $_A\mathcal C$
\begin{equation}
    \psi\colon \underline{H} \to \mathrm{Aut}_\otimes({_{A}\mathcal{C}}) \ , \qquad h \mapsto {}^h(-)
\end{equation} 
by twisting the $A$-action by $\upsilon$.
Moreover, for all $g \in G$, $h \in H$, and $M \in {}_{A}\mathcal{C}_g$, we have $^hM \in {}_{A}\mathcal{C}_{h \rhd g}$.

\begin{theorem}[{Thm.\,\ref{twisted local module crossed braiding},\; Prop.\,\ref{twisted local module crossed balanced},\; Thm.\,\ref{twisted local module crossed ribbon}}]
If $\mathcal{C}$ is braided (resp.\ balanced), and if certain compatibility conditions are satisfied, then $_{A}\mathcal{C}^\mathrm{gr}$ can be endowed with a $\chi$-crossed braided (resp.\ balanced) structure.
Moreover, if $\mathcal{C}$ is ribbon and $A$ is $\Delta$-separable Frobenius, then $_{A}\mathcal{C}^\mathrm{gr}$ is $\chi$-crossed ribbon.
\end{theorem}

This theorem generalises results in \cite{Ki01a,Ki01b,Mc19} on the $G$-crossed braided structure on ${}_A\mathcal{C}$ for $A$
in a braided category $\mathcal C$ with half-braidings as in \eqref{eq:intro-halfbraid-special}.
In Section~\ref{twisted local module examples}, we consider the example $\mathcal{C}=\mathrm{vec}_K$ for an abelian group $K$ and $A$ the algebra object associated with a finite subgroup $B \subset K$.

\subsection{Potential applications to quantum topology}
In~\cite{SV22}, the authors introduced \emph{$\chi$-graded monoidal categories} in the context of state-sum construction of 3d homotopy quantum field theories with target a homotopy 2-type. In their setting, objects are graded by $H$, and there is an extra $G$-grading for morphisms. This has recently been extended to the 4d setting with target a homotopy 3-type, which in particular takes a \emph{$\chi$-graded fusion 2-category} as input and produces an invariant of homotopy classes of maps from 4-manifolds to the classifying space of the crossed module 
$\chi$~\cite{Ce26}. This construction generalises the state-sum approach of 4-manifold invariants based on fusion 2-categories in~\cite{DR18}.

A notable feature that we anticipate for $\chi$-crossed braided categories is
that every $(\chi\colon G\to H)$-crossed braided category $\mathcal C$ gives
rise to a $\chi$-graded monoidal 2-category $\mathbf C$ in the sense
of~\cite{Ce26}. For simplicity, assume here that $\mathcal C$ is strictly
monoidal, so that $\mathbf C$ is semistrict. The 2-category $\mathbf C$ is
defined as follows:
\begin{itemize}
    \item $\mathrm{Ob}(\mathbf{C}) = H$,
    \item $\displaystyle \mathrm{Hom}_\mathbf{C}(h, h')
    \,= \hspace{-.5em} \bigoplus_{\,g \in \chi^{-1}(h'h^{-1})} \hspace{-1em} \mathcal{C}_g$.
\end{itemize}
The composition functors on $\mathrm{Hom}$-categories are induced by the
$G$-graded tensor product on~$\mathcal C$. The monoidal structure of
$\mathbf C$ is induced on objects by multiplication in $H$, and on
$\mathrm{Hom}$-categories by the $H$-action on $\mathcal C$, with
interchanger induced by the $\chi$-crossed braiding on $\mathcal C$. This
construction is the crossed-module analogue of the monoidal 2-category
associated to a $G$-crossed braided category, as considered
in~\cite[Sec.\,6.1]{Cu16}.
If $\mathcal{C}$ is a $\chi$-crossed braided fusion category for a finite crossed module $\chi$
(such as the examples based on $G$-graded vector spaces in Section~\ref{VecG chi crossed classification}), the resulting $\chi$-graded 2-category $\mathbf{C}$ is prefusion. This allows one to apply the 4-dimensional state-sum construction of~\cite{Ce26},
and it would be interesting to compute the homotopy-type 4-manifold invariants arising from the explicit examples of Section~\ref{different equiv relation}.
In the special case $\chi = \mathrm{id}\colon G \to G$, we expect this to reproduce the 4-manifold invariants obtained from $G$-crossed braided fusion categories in~\cite{Cu16}.

Since the approaches above rely heavily on the semisimplicity constraint, it is also natural to ask whether one can construct non-semisimple homotopy-type 4-manifold invariants directly via Kirby calculus from a $\chi$-crossed ribbon category based on the methods in~\cite{BD21}, without passing to its corresponding monoidal 2-category.

\smallskip

Our own motivation comes from two constructions of non-semisimple quantum invariants of 3-manifolds~\cite{CGP12, DGGPR19}, whose algebraic input categories exhibit non-semisimplicity of different nature: in \cite{CGP12}, the input is in particular a $G$-graded ribbon category
with generically semisimple components for an infinite abelian group $G$, whereas in \cite{DGGPR19}, the input is a modular tensor category, so in particular finite abelian and not necessarily graded.
Our goal is to bridge and generalise
these two constructions via a homotopy quantum field theory
based on a non-semisimple $G$-crossed ribbon category, where $G$ is not necessarily abelian and can be finite or infinite.
In~\cite{GRW} we will introduce the notion of $\chi$-crossed ribbon quasi-Hopf algebras and show that their representation categories produce $\chi$-crossed ribbon categories as defined in this paper. As a non-semisimple 
example, we construct a $\C/2\Z$-crossed quantum group based on $U_q(\mathfrak{sl}_2)$ at even orders of $q$. 
In a sequel~\cite{GGPRW}, we will use this example as algebraic data to construct a 3-manifold invariant
which recovers both constructions of~\cite{CGP12,DGGPR19} in special cases.

\subsection{Structure of the paper}
In Section~\ref{sec:crossed-actions}, we introduce and unpack the definitions of $H$-crossed $G$-graded monoidal categories, as well as their corresponding functors and natural transformations. We also introduce their fibre products.
In Section~\ref{H-action on VecG}, we classify $H$-crossed $G$-graded monoidal structures on $\mathrm{vec}_G$ in terms of group cohomological data, and compute concrete examples.
In Section~\ref{section crossed ribbon cat}, we introduce the notion of $\chi$-crossed braided (resp.\ ribbon) categories and functors between them for a given crossed module $\chi\colon G \to H$. We classify $\chi$-crossed braided structures on $\mathrm{vec}_G$ and compute concrete examples.
In Section~\ref{section twisted-local mod}, we study twisted-local modules over an algebra object in an abelian braided monoidal category.

\subsubsection*{\textbf{Acknowledgements}} The authors thank Matthew Cellot and Alexis Virelizier for helpful discussions.
AMG acknowledges support by C.N.R.S.\ and l'Agence nationale de la recherche
(ANR grant ``New algebraic structures in quantum integrability: towards 3D" NASQI3D ANR-24-CE40-7252). 
I.R. and B.W. acknowledge support by the Deutsche Forschungsgemeinschaft (DFG, German Research Foundation) under Germany's Excellence Strategy - EXC 2121 ``Quantum Universe'' - 390833306 and the Collaborative Research Centre - SFB 1624 ``Higher structures, moduli spaces and integrability'' - 506632645.

\section{Crossed actions on graded monoidal categories}\label{sec:crossed-actions}

Given an action of a group $H$ on a group $G$, we introduce the notion of an \textit{$H$-crossed $G$-graded monoidal category} (Definition~\ref{H-crossed action def}). This generalises and unifies the familiar notions of $G$-graded monoidal categories, monoidal categories equipped with a $G$-action, and $G$-crossed monoidal categories. We further introduce their fibre products in Definition~\ref{fibre product def}. For background literature on $G$-crossed categories, we refer the reader to~\cite{EGNO15, Ga16}.

\subsection{$H$-Crossed $G$-graded monoidal categories}

Let $G$ be a group, and let $\mathcal{C}$ be an additive monoidal category. $\mathcal{C}$ is called \textit{$G$-graded monoidal} if there exists a family of full subcategories $\{\mathcal{C}_g\}_{g \in G}$ in $\mathcal{C}$ such that

\begin{itemize}
    \item $\mathcal{C}=\bigoplus_{g \in G}\mathcal{C}_g$ (we allow $\mathcal{C}_g=0$);

    \item the tensor unit $\mathbbm{1}$ lies in $\mathcal{C}_{1_G}$;
    
    \item for any $a, b\in G$ and any $U \in\mathcal{C}_a$, $V\in\mathcal{C}_b$, one has $U\otimes V\in\mathcal{C}_{ab}$.
\end{itemize}
Here, by $\mathcal{C}=\bigoplus_{g \in G}\mathcal{C}_g$ we mean that any object in $\mathcal{C}$ is isomorphic to a finite direct sum of homogeneous objects, and that for any $U \in \mathcal{C}_a$ and $V \in \mathcal{C}_b$ with $a \neq b$, $\Hom_\mathcal{C}(U,V)=0$.
The category $\mathcal C$ is called \textit{faithfully} $G$-graded if $\mathcal{C}_g \neq 0$ for all $g \in G$.

Denote by $\mathcal{C}_{\mathrm{hom}}$ the full subcategory in $\mathcal{C}$ consisting of homogeneous objects. The degree of $U \in \mathcal{C}_{\mathrm{hom}}$ is denoted by $|U| \in G$.

\begin{remark}
    If $\mathcal{C}$ is moreover rigid, then for every $a\in G$ and every object $U\in \mathcal{C}_a$, its left and right duals satisfy $U^*,\,{}^*U \in \mathcal{C}_{a^{-1}}$.
\end{remark}

Let $H$ be a group acting on $G$ by group automorphisms, and we denote the action of an element $h \in H$ by $h\rhd (-)$. $H$ can be considered as a discrete strictly monoidal category $\underline{H}$ with the monoidal structure induced by group multiplication.

Let $\Aut_\otimes(\mathcal C)$ denote the category whose objects are monoidal additive autoequivalences of $\mathcal C$ and whose morphisms are monoidal natural isomorphisms. $\Aut_\otimes(\mathcal C)$ is strictly monoidal with the tensor product on objects given by composition of functors, and on morphisms given by horizontal composition of natural transformations. We denote vertical composition
of natural transformations by $\Box$ and horizontal composition by $*$.

\begin{definition}\label{H-crossed action def}
    Given an action of $H$ on $G$, an \textit{$H$-crossed action} on a $G$-graded monoidal category $\mathcal{C}$ is a monoidal functor
    \begin{equation}
        \phi\colon \underline{H} \longrightarrow \mathrm{Aut}_\otimes(\mathcal{C})
    \end{equation}
    such that for any $(h,g)\in H \times G$, the functor $\phi(h)$ maps $\mathcal{C}_g$ into $\mathcal{C}_{h\rhd g}$.
    
The category $\mathcal C$ with such an action is called an \textit{$H$-crossed $G$-graded monoidal category}.
\end{definition}

\begin{remark}
    If $G = \{*\}$, this recovers the notion of a monoidal category with an $H$-action; if $H = G$ and the action is given by conjugation, this recovers $G$-crossed monoidal categories.
\end{remark}

Let us unpack the definition above. We write $\lambda$, $\rho$, $\alpha$ for the left unitor, right unitor, and associator in $\mathcal{C}$, respectively.
\begin{itemize}
    \item For each $h \in H$, \textit{$\varphi(h) = \bigl( {}^h(-), h_{-,-}, h_{\mathbbm{1}}\bigr)$ is a monoidal autoequivalence of $\mathcal{C}$:} ${}^h(-)$ is an autoequivalence
    \begin{equation}\label{eq:H-crossed-phi-functor}
        ^h(-)\colon \mathcal{C} \to \mathcal{C} \ , \qquad ^h\mathcal{C}_g \subset \mathcal{C}_{h \rhd g}\ .
    \end{equation}
For all $U,V \in \mathcal{C}$, $h_{U,V}$ is a natural isomorphism
\begin{gather}\label{monoidal structure h-action}
    h_{U,V}\colon {^hU \otimes {^hV}} \xrightarrow[]{\ \sim \ } {^h(U \otimes V)}\ ,
\end{gather}
and $h_{\mathbbm{1}}$ is an isomorphism
    \begin{equation}\label{h-action unitor}
        h_{\mathbbm{1}}\colon \mathbbm{1} \xrightarrow{\ \sim\ }{}^h\mathbbm{1}\ ,
    \end{equation}
such that for any $U,V,W \in \mathcal{C}$, the following diagrams commute:
\begin{subequations}\label{action C'}
\begin{equation}\label{action C3}
\begin{tikzcd}
	{^hU \otimes \mathbbm{1}} && { {}^hU \otimes {}^h\mathbbm{1}} && {^h(U \otimes \mathbbm{1})} \\
	&& {^hU} \\
	{\mathbbm{1} \otimes {}^hU} && { {}^h\mathbbm{1} \otimes {}^h U} && {^h(\mathbbm{1} \otimes U)}
	\arrow["{\mathrm{id} \otimes h_{\mathbbm{1}}}", from=1-1, to=1-3]
	\arrow["{\rho_{^hU}}"{pos=0.6}, from=1-1, to=2-3]
	\arrow["{h_{U,\mathbbm{1}}}", from=1-3, to=1-5]
	\arrow["{^h\rho_U}"'{pos=0.6}, from=1-5, to=2-3]
	\arrow["{\lambda_{^hU}}"'{pos=0.6}, from=3-1, to=2-3]
	\arrow["{h_{\mathbbm{1}} \otimes \mathrm{id}}"', from=3-1, to=3-3]
	\arrow["{h_{\mathbbm{1},U}}"', from=3-3, to=3-5]
	\arrow["{^h\lambda_U}"{pos=0.6}, from=3-5, to=2-3]
\end{tikzcd}
\end{equation}

\begin{equation}\label{action C4}
\begin{tikzcd}
	{^hU \otimes ({^hV}\otimes {^hW})} &&& {(^hU \otimes {^hV})\otimes {^hW}} \\
	{^hU \otimes {}^h(V\otimes W)} &&& {^h(U \otimes V)\otimes {^hW}} \\
	{^h(U \otimes (V\otimes W))} &&& {^h((U \otimes V)\otimes W)}
	\arrow["{\alpha_{{}^hU, {}^hV, {}^hW }}", from=1-1, to=1-4]
	\arrow["{\mathrm{id}\otimes h_{V,W}}", from=1-1, to=2-1]
	\arrow["{h_{U,V}\otimes\mathrm{id}}"', from=1-4, to=2-4]
	\arrow["{h_{U, V \otimes W}}", from=2-1, to=3-1]
	\arrow["{h_{U\otimes V,W}}"', from=2-4, to=3-4]
	\arrow["{^h\alpha_{U,V,W}}", from=3-1, to=3-4]
\end{tikzcd}
\end{equation}
\end{subequations}

\item
\textit{$\varphi = (\varphi,T,\eta)$ is a monoidal functor:} For any two group elements $h_1, h_2 \in H$, $T^{h_1, h_2}$ is a natural isomorphism
    \begin{equation}\label{eq:phi-mon-T-def}
        T^{h_1, h_2}\colon {}^{h_1}({}^{h_2}(-))\xRightarrow{\ \sim \ } {}^{h_1h_2}(-) \ ,
    \end{equation}
    such that for any $h_1, h_2, h_3 \in H$, the following diagrams commute:
\begin{subequations}\label{action C}
\begin{equation}\label{action C2}
\begin{tikzcd}
	{^{h_1}(^{h_2}U)\otimes{}^{h_1}(^{h_2}V)} && {^{h_1}({}^{h_2}U \otimes {^{h_2}V})} && {^{h_1}({}^{h_2}(U \otimes V))} \\
	{^{h_1h_2}U \otimes {}^{h_1h_2}V} &&&& {^{h_1h_2}(U \otimes V)}
	\arrow["{{h_1}_{{}^{h_2}U, {^{h_2}V}}}", from=1-1, to=1-3]
	\arrow["{T^{h_1,h_2}_U\otimes T^{h_1,h_2}_V}", from=1-1, to=2-1]
	\arrow["{^{h_1}{h_2}_{U,V}}", from=1-3, to=1-5]
	\arrow["{T^{h_1,h_2}_{U \otimes V}}"', from=1-5, to=2-5]
	\arrow["{{h_1h_2}_{U,V}}", from=2-1, to=2-5]
\end{tikzcd}
\end{equation}
\begin{equation}\label{action C5}
\begin{tikzcd}[column sep = 6em, row sep = 1.5em]
	{\mathbbm{1}} & {^{h_1}\mathbbm{1}} \\
	{^{h_1h_2}\mathbbm{1}} & {^{h_1}(^{h_2}\mathbbm{1})}
	\arrow["{({h_1})_\mathbbm 1}", from=1-1, to=1-2]
	\arrow["{({h_1h_2})_\mathbbm{1}}", from=1-1, to=2-1]
	\arrow["{^{h_1}({h_2})_\mathbbm 1}"', from=1-2, to=2-2]
	\arrow["{T^{h_1,h_2}_\mathbbm 1}"', from=2-2, to=2-1]
\end{tikzcd}
\end{equation}
\begin{equation}\label{action C1}
\begin{tikzcd}
	{{}^{h_1}({}^{h_2}({}^{h_3}(-)))} && {{}^{h_1}({}^{h_2h_3}(-))} \\
	{{}^{h_1h_2}({}^{h_3}(-))} && {{}^{h_1h_2h_3}(-)}
	\arrow["{\mathrm{id}*T^{h_2,h_3}}", Rightarrow, from=1-1, to=1-3]
	\arrow["{T^{h_1,h_2}*\mathrm{id}}", Rightarrow, from=1-1, to=2-1]
	\arrow["{T^{h_1,h_2h_3}}"', Rightarrow, from=1-3, to=2-3]
	\arrow["{T^{h_1h_2,h_3}}", Rightarrow, from=2-1, to=2-3]
\end{tikzcd}
\end{equation}
\end{subequations}
where the first two diagrams express that each $T^{h_1,h_2}$ is a monoidal natural isomorphism, and the last diagram is the associativity coherence for the monoidal functor $\phi$.
Furthermore, $\eta$ is a natural isomorphism
\begin{equation}\label{eq:phi-monoidal-unit-iso}
    \eta\colon \mathrm{id}_\mathcal{C} \xRightarrow{\ \sim \ } {}^{1_H}(-)\ ,
\end{equation}
such that for any $h \in H$, $U,V \in \mathcal C$, the following conditions hold:
\begin{subequations}\label{action C''}
\begin{gather}
    T^{h,1}\Box(\mathrm{id}_{^h(-)}*\eta) = \mathrm{id}_{^h(-)} = T^{1,h}\Box(\eta*\mathrm{id}_{^h(-)})\ ,\\
    \eta_{U\otimes V} = (1_H)_{U,V} \circ (\eta_U \otimes \eta_V)\ ,\\
    \eta_\mathbbm 1 = (1_H)_\mathbbm 1\ . 
\end{gather}
\end{subequations}
\end{itemize}
\medskip

\begin{remark}\label{rem:canoncial-iso-duals}
If $\mathcal{C}$ is rigid, then for any $h \in H$, $U \in \mathcal{C}$, the objects $(^hU)^*$ and ${}^h(U^*)$ are both dual to ${}^hU$. Hence there is a unique isomorphism $\nu_{h,U}$ between them, compatible with the (co)evaluation maps. Explicitly, $\nu_{h,U}$ is the composition
\begin{align}
& {}^h(U^*) 
\xrightarrow{\sim} 
{}^h(U^*) \otimes \mathbbm{1}
\xrightarrow{\id \otimes \mathrm{coev}_{^hU}}
{}^h(U^*) 
\otimes
\big({}^hU \otimes ({}^hU)^* \big)
\xrightarrow{\sim}
\big({}^h(U^*)
\otimes
{}^hU \big) \otimes ({}^hU)^*
\nonumber \\
& \xrightarrow{h_{U^*,U} \otimes \id}
{}^h(U^* \otimes U ) \otimes ({}^hU)^*
\xrightarrow{{}^h\mathrm{ev}_{U} \otimes \id}
{}^h\mathbbm{1} \otimes ({}^hU)^*
\xrightarrow{h_{\mathbbm{1}}^{-1} \otimes \id}
\mathbbm{1} \otimes ({}^hU)^*
\xrightarrow{\sim}
({}^hU)^* 
\ .
\end{align}
There is a corresponding isomorphism for right duals, but we will not need it.
\end{remark}

\begin{remark}\label{free choice up to monoidal natural equiv}
Let $(\phi, T, \eta)\colon \underline{H}\to \mathrm{Aut}_\otimes(\mathcal C)$ be a monoidal functor and
suppose that we are given a family of monoidal autoequivalences and monoidal natural isomorphisms
\begin{equation}
    \{\phi'_h\colon \mathcal C\to \mathcal C\}_{h\in H}\ , \qquad \{\zeta^h\colon \phi(h)\xRightarrow{\sim}\phi'_h\}_{h\in H}\ .
\end{equation}
Then the monoidal structure on $\phi$ can be transported along the family
$\zeta=\{\zeta^h\}_{h\in H}$. Namely, we define 
\begin{subequations}
\begin{gather}
    \phi'(h)\coloneqq \phi'_h \ , \\
    T'^{h_1,h_2}
    \coloneqq
    \zeta^{h_1h_2}
    \Box T^{h_1,h_2}
    \Box
    \big((\zeta^{h_1})^{-1}*(\zeta^{h_2})^{-1}\big) \ ,\\
    \eta'
    \coloneqq
    \zeta^{1_H}\Box \eta
    \colon
    \id_{\mathcal C}
    \xRightarrow{\ \sim\ }
    \phi'(1_H) \ .
\end{gather}
\end{subequations}
With these definitions, $(\phi',T',\eta')$ is a monoidal functor, and $\zeta = \{\zeta^h\}_{h \in H}\colon \phi\xRightarrow{\sim}\phi'$ becomes a monoidal natural isomorphism.
As a consequence, 
up to monoidal natural isomorphism, we may assume, for all $h \in H$:
\begin{equation}\label{eq: strictness assumption}
    h_\mathbbm 1 = \mathrm{id}_\mathbbm 1 \ , \qquad \phi(1_H) = \mathrm{id}_\mathcal{C}\ , \qquad \eta = \mathrm{id}_{\mathrm{id}_\mathcal{C}}\ .
\end{equation}
\end{remark}

\subsection{Equivariant functors and natural transformations}

We introduce the notions of functors and natural transformations for $H$-crossed $G$-graded monoidal categories, which specialise to those given in~\cite[Ch.\,6.3.5]{T10} and~\cite{Ga16} in the case $H=G$ with the conjugation action.

\medskip

Note that an $H$-crossed action 
$(\phi,T,\eta)\colon \underline H \to \Aut_\otimes(\mathcal C)$
on a $G$-graded monoidal category~$\mathcal C$ naturally equips $\mathcal C$ with an $\underline{H}$-module category structure, with structural isomorphisms
(mixed associators and module unitors) given by~\eqref{eq:phi-mon-T-def} and~\eqref{eq:phi-monoidal-unit-iso}:
\begin{equation}
    T^{h_1,h_2}_U\colon
    {}^{h_1}({}^{h_2}U)
    \xrightarrow{\ \sim\ }
    {}^{h_1h_2}U \ , \qquad \eta_U\colon U\xrightarrow{\ \sim\ }{}^{1_H}U \ .
\end{equation}
These isomorphisms satisfy standard coherence conditions (see, e.g. \cite[Def.\,7.1.2]{EGNO15}); in particular, diagram \eqref{action C1} corresponds to the pentagon equation for the mixed associator, taking into account the strictness of
$\underline H$.

\begin{definition} \label{crossed functor}
Let $\mathcal{C}$ and $\mathcal{C}'$ be two $H$-crossed $G$-graded monoidal categories with the same underlying $H$-action on $G$. An \textit{$H$-equivariant $G$-graded monoidal functor} $\mathcal{C} \to \mathcal{C}'$ is a pair $(F,\omega)$, where
\begin{itemize}
    \item $F = (F, F_{-,-}, F_{\mathbbm 1})$ is a degree preserving 
    additive monoidal functor:
    \begin{equation}
        F\colon \mathcal{C} \to \mathcal{C}'\ , \qquad
        F_{-,-}: F(-) \otimes F(-) \Rightarrow F(-\otimes -) \ ,
        \qquad F_{\mathbbm 1}\colon \mathbbm{1}' \to F(\mathbbm{1}) \ ;
    \end{equation}
    \item $\omega = (\omega^h)_{h \in H}$ is a family of monoidal natural isomorphisms
    \begin{equation}\label{eq: omega^h}
        \omega^h \coloneqq \omega(h, -) \colon {}^hF(-) \xLongrightarrow{\sim} F({}^h(-))
    \end{equation}
which equips $F$ with the structure of an $\underline{H}$-module functor.
\end{itemize}    
If $F$ is an equivalence of categories, then $\mathcal{C}$ and $\mathcal{C}'$ are said to be \textit{$H$-crossed $G$-graded monoidally equivalent}.
\end{definition}

Let us unpack the conditions on the family of natural isomorphisms $\omega = (\omega^h)_{h \in H}$. Explicitly, for all $h,h_1,h_2 \in H$ and all $U,V \in \mathcal C$, the following diagrams must commute:
\begin{subequations}\label{crossed functor definition}
\begin{equation}\label{crossed functor definition a}
    \begin{tikzcd}[column sep = 4em, row sep = 1.5em]
	{{}^hF(U)\otimes {}^hF(V)} & {{}^h(F(U)\otimes F(V))} & {{}^hF(U \otimes V)} \\
	{F({}^hU)\otimes F({}^hV)} & {F({}^hU \otimes {}^hV)} & {F({}^h(U \otimes V))}
	\arrow["{h_{F(U),F(V)}}", from=1-1, to=1-2]
	\arrow["{\omega^h_U \otimes \omega^h_V}", from=1-1, to=2-1]
	\arrow["{{}^hF_{U,V}}", from=1-2, to=1-3]
	\arrow["{\omega^h_{U \otimes V}}"', from=1-3, to=2-3]
	\arrow["{F_{{}^hU,{}^hV}}", from=2-1, to=2-2]
	\arrow["{F(h_{U,V})}", from=2-2, to=2-3]
\end{tikzcd}
\end{equation}

\begin{equation}\label{crossed functor definition b}
\begin{tikzcd}[column sep = 4em, row sep = 1.5em]
	{\mathbbm 1'} & {F(\mathbbm 1)} & {F({}^h\mathbbm 1)} \\
	{{}^h\mathbbm 1'} && {{}^hF(\mathbbm 1)}
	\arrow["{F_{\mathbbm 1}}", from=1-1, to=1-2]
	\arrow["{h_{\mathbbm 1'}}", from=1-1, to=2-1]
	\arrow["{F(h_{\mathbbm 1})}", from=1-2, to=1-3]
	\arrow["{{}^hF_{\mathbbm 1}}", from=2-1, to=2-3]
	\arrow["{\omega^h_{\mathbbm 1}}", from=2-3, to=1-3]
\end{tikzcd}
\end{equation}

\begin{equation}\label{crossed functor definition c}
\begin{tikzcd}
	{{}^{h_2}({}^{h_1}F(-))} && {{}^{h_2}F({}^{h_1}(-))} && {F({}^{h_2}({}^{h_1}(-)))} \\
	{{}^{h_2h_1}(F(-))} &&&& {F({}^{h_2h_1}(-))}
	\arrow["{\mathrm{id}_{^{h_2}(-)} * \omega^{h_1}}", Rightarrow, from=1-1, to=1-3]
	\arrow["{T'^{h_2,h_1} * \mathrm{id}_F}", Rightarrow, from=1-1, to=2-1]
	\arrow["{\omega^{h_2} * \mathrm{id}_{^{h_1}(-)}}", Rightarrow, from=1-3, to=1-5]
	\arrow["{\mathrm{id}_F * T^{h_2,h_1}}"', Rightarrow, from=1-5, to=2-5]
	\arrow["{\omega^{h_2h_1}}", Rightarrow, from=2-1, to=2-5]
\end{tikzcd}
\end{equation}

\begin{equation}\label{crossed functor definition d}
\begin{tikzcd}[column sep = 4em, row sep = 1.5em]
	{{}^{1_H}F(-)} & {F({}^{1_H}(-))} \\
	{F(-)}
	\arrow["{\omega^{1_H}}", Rightarrow, from=1-1, to=1-2]
	\arrow["{\eta' * \mathrm{id}_F}", Rightarrow, from=2-1, to=1-1]
	\arrow["{\mathrm{id}_F * \eta}"'{pos=0.5}, Rightarrow, from=2-1, to=1-2]
\end{tikzcd}
\end{equation}
\end{subequations}
The first two diagrams express the condition that each $\omega^h$ is a monoidal natural isomorphism, while the last two diagrams are the coherence conditions for $F$ to be an $\underline{H}$-module functor (cf.~\cite[Def.\,7.2.1]{EGNO15}).

\begin{remark}\label{rem: coherence}
Analogously to Mac Lane's coherence theorem for monoidal categories, any $H$-crossed $G$-graded monoidal category is equivalent to a strict one,
i.e.\ one in which all coherence isomorphisms in Definition~\ref{H-crossed action def} are identity morphisms (namely -- in addition to \eqref{eq: strictness assumption} -- the associator and unitors of $\mathcal{C}$ and
$h_{U,V}$, $T^{h,h'}_U$ for all $h,h'\in H$, $U,V\in\mathcal{C}$). This is proven for $G=\{*\}$ in~\cite[Thm.\,4.3]{Ga16},
and for $H=G$ with the conjugation action in~\cite[Cor.\,5.3]{Ga16}.
The general case follows by essentially the same arguments as~\cite[Cor.\,5.3]{Ga16}. Moreover, in the spirit of the strictification results for monoidal functors~\cite{JS93}, we expect a similar coherence theorem to hold for $H$-equivariant $G$-graded monoidal functors, ensuring that any diagram built from the underlying structural isomorphisms commutes.
However, we do not pursue this further in the present paper.
\end{remark}

\begin{remark}\label{equivalent action are indeed equivalent}
As in Remark~\ref{free choice up to monoidal natural equiv},
suppose two $H$-crossed actions $\phi_1, \phi_2\colon \underline{H} \to \mathrm{Aut}_\otimes(\mathcal{C})$ are
monoidally natural isomorphic via $\zeta\colon
\phi_1 \xRightarrow{\sim} \phi_2$. Then their induced $H$-crossed $G$-graded monoidal structures on $\mathcal{C}$ are equivalent via $(\mathrm{id}, \omega = \zeta^{-1})$.
\end{remark}

\begin{definition}\label{chi-crossed natural transformation}
    Let $(F, \omega), (F', \omega')\colon \mathcal{C} \to \mathcal{C}'$ 
    be two $H$-equivariant $G$-graded monoidal functors. An \textit{$H$-equivariant monoidal natural transformation} $\xi\colon (F, \omega) \Rightarrow (F', \omega')$ is a monoidal natural transformation 
    that is also an $\underline{H}$-module natural transformation. If $\xi_U$ is an isomorphism for every $U \in \mathcal{C}$, then $\xi$ is called an \textit{$H$-equivariant monoidal natural isomorphism}.
\end{definition}

In particular, being an $\underline{H}$-module natural transformation
means that for any $h \in H$ and $U \in \mathcal{C}$, the following diagram commutes:
    \begin{equation}
        \begin{tikzcd}[column sep = 4em, row sep = 1.5em]
	{^hF(U)} & {F({}^hU)} \\
	{^hF'(U)} & {F'({}^hU)}
	\arrow["{\omega^h_U}", from=1-1, to=1-2]
	\arrow["{^h\xi_U}", from=1-1, to=2-1]
	\arrow["{\xi_{{}^hU}}"', from=1-2, to=2-2]
	\arrow["{\omega'^h_U}", from=2-1, to=2-2]
\end{tikzcd}
    \end{equation}
Recall the natural isomorphisms $\eta$ and $\eta'$ from \eqref{eq:phi-monoidal-unit-iso}. For any $U \in \mathcal{C}$, the naturality of $\xi$ with respect to $\eta_U$, together with~\eqref{crossed functor definition d}, ensures that the following diagram commutes:
\begin{equation}
\begin{tikzcd}
	& {^{1_H}F(U)} & {^{1_H}F'(U)} \\
	{F(U)} &&& {F'(U)} \\
	& {F(^{1_H}U)} & {F'(^{1_H}U)}
	\arrow["{^{1}\xi_U}", from=1-2, to=1-3]
	\arrow["{{\eta'_{F'(U)}}^{-1}}", from=1-3, to=2-4]
	\arrow["{\eta'_{F(U)}}", from=2-1, to=1-2]
	\arrow["{\xi_U}"{description}, from=2-1, to=2-4]
	\arrow["{F(\eta_U)}"', from=2-1, to=3-2]
	\arrow["{\xi_{^{1}U}}", from=3-2, to=3-3]
	\arrow["{F'(\eta_U^{-1})}"', from=3-3, to=2-4]
\end{tikzcd}
\end{equation}

\begin{remark}\label{unital assumption on functors}
Up to an $H$-equivariant monoidal natural isomorphism, we may always assume an $H$-equivariant $G$-graded monoidal functor $(F,\omega)\colon \mathcal{C} \to \mathcal{C}'$ to be strictly unital, i.e.\ $\mathbbm{1}' = F(\mathbbm{1})$. If $1_H$ acts by identity on both $\mathcal{C}$ and $\mathcal{C}'$, we may also assume $\omega^{1_H} = \mathrm{id}_F$.
\end{remark}

\begin{remark}\label{rmk: adjunction}
It is well-known that the adjoint equivalence of a monoidal functor is again monoidal, and the unit and counit are monoidal natural isomorphisms. Likewise,
by~\cite[Lem.\,2.11]{DSPS14} 
(the lemma assumes linearity, but it is not required in the proof), 
the adjoint equivalence of a module functor is again a module functor, with unit and counit being module natural isomorphisms. Therefore, for an $H$-equivariant $G$-graded monoidal functor $F\colon \mathcal C \to \mathcal C'$, the following are equivalent:
\begin{itemize}
\item $F$ is an equivalence of categories;
\item there exists an $H$-equivariant $G$-graded monoidal functor $F'\colon \mathcal C' \to \mathcal C$
together with $H$-equivariant monoidal natural isomorphisms
\begin{equation}
    \eta\colon \mathrm{id}_{\mathcal C} \xRightarrow{\sim} F'F \ ,
\qquad
\varepsilon\colon FF' \xRightarrow{\sim} \mathrm{id}_{\mathcal C'} \ .
\end{equation}
\end{itemize}
\end{remark}

\begin{definition}\label{def: Cr_H,G}
For fixed groups $H,\ G$ and a specified action $\rhd$ of $H$ on $G$, we define a strict 2-category $\mathsf{Cr}_{H,G,\rhd}$ as follows:
    \begin{itemize}
        \item objects are $H$-crossed $G$-graded monoidal categories whose underlying $H$-action on $G$ is a fixed action $\rhd$;
        \item 1-morphisms are $H$-equivariant $G$-graded monoidal functors,
        with composition
        \begin{equation}
            (F', \omega') \circ (F,\omega) \coloneqq \Big(F'\circ F, \,  (\id_{F'} * \omega)\Box
        \bigl(\omega' * \id_{\id_{\underline H}\times F}\bigr)\Big) \ ;
        \end{equation}
        \item 2-morphisms are $H$-equivariant monoidal natural transformations.
    \end{itemize}
\end{definition}

It is straightforward to verify that the above composition of $1$-morphisms is strictly unital and strictly associative, by the standard associativity of pasting in the strict $2$-category of categories. Hence, $\mathsf{Cr}_{H,G,\rhd}$ is indeed a strict $2$-category.

\subsection{Fibre products}\label{sec: fibre product}

Fix a ground field $\mathbb{K}$. Let \(\mathcal C\) and \(\mathcal D\) be $\mathbb{K}$-linear \(H\)-crossed \(G\)-graded monoidal categories with the same underlying action of \(H\) on \(G\), and with the respective monoidal products $\otimes_{\mathcal C}$ and $\otimes_{\mathcal D}$.
One can endow the 
fibre product of $G$-graded categories \cite{Ni18} with an $H$-crossed structure:

\begin{definition}\label{fibre product def}
The \emph{fibre product} $\mathcal C\boxtimes_G \mathcal D$ of \(\mathcal C\) and \(\mathcal D\) over \(G\) is the $H$-crossed \(G\)-graded monoidal category defined as follows:
\begin{itemize}
\item
The underlying $G$-graded additive category is
\begin{equation}
    \mathcal C\boxtimes_G \mathcal D = \bigoplus_{g \in G} (\mathcal C\boxtimes_G \mathcal D)_g \;\coloneqq\; \bigoplus_{g\in G}\big(\mathcal C_g\boxtimes \mathcal D_g\big)\ ,
\end{equation}
where $\boxtimes$ denotes the additive closure of the $\mathbb{K}$-linear cartesian product.
Explicitly, an object in $\mathcal C\boxtimes_G\mathcal D$ is a finite direct sum
of objects of the form $U\boxtimes V$, where $U\in\mathcal C_g$ and $V\in\mathcal D_g$ for some $g\in G$; and
morphisms for pure tensors are given by
\begin{align}
    \begin{split}
        \Hom_{\mathcal C\boxtimes_G\mathcal D}(U\boxtimes V,\,U'\boxtimes V') \; &\coloneqq  \;
\Hom_{\mathcal C}(U,U')\otimes_\mathbb{K} \Hom_{\mathcal D}(V,V')\ , \\
f \boxtimes f' \; &\coloneqq \;  f\otimes_\mathbb{K} f' \ .
    \end{split}
\end{align}

\item The tensor product on pure tensors is given by
\begin{equation}
    (U_1\boxtimes V_1)\otimes (U_2\boxtimes V_2)
\;\coloneqq\;
(U_1\otimes_\mathcal C U_2)\boxtimes (V_1\otimes_\mathcal D V_2)
\end{equation}
with the unit being
\begin{equation}
    \mathbbm 1_{\mathcal C\boxtimes_G\mathcal D}
\;\coloneqq\;
\mathbbm 1_{\mathcal C}\boxtimes \mathbbm 1_{\mathcal D} \ .
\end{equation}
The associator on pure tensors is given component-wise by
\begin{equation}
    \alpha_{(U_1\boxtimes V_1),(U_2\boxtimes V_2),(U_3\boxtimes V_3)}
\;\coloneqq\;
\alpha^{\mathcal C}_{U_1,U_2,U_3}\boxtimes \alpha^{\mathcal D}_{V_1,V_2,V_3} \ ,
\end{equation}
and the left and right unitors are given by
\begin{equation}
    \lambda_{U\boxtimes V}\coloneqq\lambda^{\mathcal C}_{U}\boxtimes \lambda^{\mathcal D}_{V} \ ,
\qquad
\rho_{U\boxtimes V}\coloneqq\rho^{\mathcal C}_{U}\boxtimes \rho^{\mathcal D}_{V} \ .
\end{equation}

\item The $H$-action on pure tensors is given by
\begin{equation}\label{eq:H-action-on-Kelly-product}
    {}^{h}(U\boxtimes V)\;\coloneqq\;{}^{h}U\boxtimes {}^{h}V \ ,
\end{equation}
with the structure isomorphisms componentwise given by
\begin{gather}
    h_{(U_1\boxtimes V_1),(U_2\boxtimes V_2)}
\; \coloneqq\; h^{\mathcal C}_{U_1,U_2}\boxtimes h^{\mathcal D}_{V_1,V_2} \ , \qquad
T^{h_1,h_2}_{U\boxtimes V} \; \coloneqq \; T^{\mathcal C; \, h_1, h_2}_U \boxtimes T^{\mathcal D; \, h_1, h_2}_V \ .
\end{gather}
\end{itemize}
\end{definition}

It is straightforward to check that this indeed defines an $H$-crossed $G$-graded monoidal structure.
The terminology is justified by the following observation. A $G$-graded category, or more precisely the full subcategory of homogeneous objects in the additive setting, can be regarded as a category fibred over the discrete category $\underline G$. The above construction first takes the fibrewise product over $\underline G$, and then completes with respect to direct sums.

\begin{remark}
If $\mathcal C$ and $\mathcal D$ are abelian with right exact tensor product,
one may instead take $\boxtimes$ to be 
the Deligne tensor product
as in~\cite{Ni18}. Therefore, the precise meaning of $\boxtimes_G$ may be adapted to the ambient setting of interest. We expect that in a suitable context, $\boxtimes_G$ equips the strict 2-category $\mathsf{Cr}_{H,G,\rhd}$ from Definition~\ref{def: Cr_H,G}
with the structure of a symmetric monoidal 2-category, but we do not pursue this in the present paper. 
\end{remark}

\section{Crossed actions on graded vector spaces}\label{H-action on VecG}

For given groups $G,H$ and a fixed action $\rhd$ of $H$ on $G$, in this section, we classify
$H$-crossed $G$-graded monoidal structures 
on the category of $G$-graded vector spaces in terms of group cohomology (Theorem~\ref{thm:VecG-Hcrossed-classification-H3tot}). Section~\ref{Z/4 example} demonstrates examples of this classification for $G = \Z/4\Z$.

\subsection{Group cohomology}\label{sec: group cohomology}

Let $G$ be a group and let $M$ be an abelian group
equipped with a right $G$-action; we write this action as $m^g$ for $m\in M$ and $g\in G$.
Consider the standard cochain complex $(C^{n}(G;M),\mathrm{d}^n)_{n\geq 0}$, where:
\begin{itemize}
\item $C^{n}(G;M)\coloneqq\mathrm{Map}(G^{n};M)$ with the convention $G^{0}\coloneqq\{*\}$;

\item the differential map $\mathrm{d}^{n}\colon C^{n}(G;M) \to  C^{n+1}(G;M)$ is
given by
\begin{equation}\label{eq:group-cohomology-differential}
    \begin{split}
        (\mathrm{d}^{n}f)(g_{1},\dots,g_{n+1})
\,=\,& f(g_{2},\dots,g_{n+1}) \\
& + \sum_{i=1}^{n}(-1)^{i}\,f(g_{1},\dots,g_{i}g_{i+1},\dots,g_{n+1}) \\
& + (-1)^{n+1}f(g_{1},\dots,g_{n})^{g_{n+1}} \ .
    \end{split}
\end{equation}
\end{itemize}
Each $C^{n}(G;M)$ is an abelian group under pointwise addition. For each $n \geq 0$, the
\textit{$n\text{-}\mathrm{th}$ cohomology group of $G$ with values in $M$} is
\begin{gather}
    H^n(G;M)\coloneqq \begin{cases}
\mathrm{Ker}(\mathrm{d}^n)/\mathrm{Im}(\mathrm{d}^{n-1}) \ , & n \geq 1  \ ,\\
\mathrm{Ker}(\mathrm{d}^0) \ , & n = 0 \ .
\end{cases}
\end{gather}

\medskip

Let $H$ be a group acting
on $G$ from the left by $\rhd$, and we write
\begin{equation}
    G \rtimes_\rhd H
\end{equation}
for the semidirect product
with multiplication $(g,h) \cdot (g',h') = (g (h\rhd g'),hh')$.
Let $A$ be an abelian group
with trivial $G \rtimes_{\rhd} H$-action. One can consider general $G \rtimes_{\rhd} H$-actions on $A$ but we will not need this.

For any integer $q \geq 0$, \(C^{q}(G;A)\) is a right $H$-module with the action of $h \in H$ on a cochain $f\in C^{q}(G;A)$ given by:
\begin{equation}\label{h-action notation}
    \begin{aligned}
(f^h)(g_1,\dots,g_q) &\coloneqq
f(h\rhd g_1,\dots,h\rhd g_q)
\ , \qquad &q>0\ ,\\
f^h &\coloneqq
f
\ , \qquad &q=0\ .
\end{aligned}
\end{equation}

Consider the Lyndon–Hochschild–Serre double complex $(\mathbf{C}^{p,q},\mathrm{d}^{p,q}_H,\mathrm{d}^{p,q}_G)_{p,q \geq 0}$, where
\begin{itemize}
    \item $\mathbf{C}^{p,q}\coloneqq C^{p}\!\bigl(H;C^{q}(G;A)\bigr)
\cong \mathrm{Map}(H^{p}\times G^{q}, A)$;
\item under this identification, the differentials \(\mathrm{d}_H^{p,q}:\mathbf{C}^{p,q}\to \mathbf{C}^{p+1,q}\) and
\(\mathrm{d}_G^{p,q}:\mathbf{C}^{p,q}\to \mathbf{C}^{p,q+1}\) are given by
\begin{subequations}\label{eq:double-complex-differentials}
\begin{align}
(\mathrm{d}_H^{p,q} f)(h_1,\dots,h_{p+1};\,g_1,\dots,g_q)
&= f(h_2,\dots,h_{p+1};\,g_1,\dots,g_q) \nonumber\\
&\quad +\sum_{i=1}^{p}\,(-1)^i\,
f(h_1,\dots,h_i h_{i+1},\dots,h_{p+1};\,g_1,\dots,g_q) \nonumber\\
&\quad +
(-1)^{p+1}\, f(h_1,\dots,h_p;\,h_{p+1}\rhd g_1,\dots,h_{p+1}\rhd g_q)
\ , \label{eq:dH}\\[0.5em]
(\mathrm{d}_G^{p,q} f)(h_1,\dots,h_p;\,g_1,\dots,g_{q+1})
&= f(h_1,\dots,h_p;\,g_2,\dots,g_{q+1}) \nonumber\\
&\quad +\sum_{i=1}^{q}\,
(-1)^{i}\,f(h_1,\dots,h_p;\,g_1,\dots,g_i g_{i+1},\dots,g_{q+1}) \nonumber\\
&\quad +
(-1)^{q+1}\,f(h_1,\dots,h_p;\,g_1,\dots,g_q)
\ . \label{eq:dG}
\end{align}
\end{subequations}
\end{itemize}
With these conventions, we have
\begin{equation}
    \mathrm{d}_H^{p,q+1} \mathrm{d}_G^{p,q} = \mathrm{d}_G^{p+1,q} \mathrm{d}_H^{p,q} \ .
\end{equation}
The associated total complex $(C^n_{\mathrm{tot}}(G \rtimes_\rhd H;A),\mathrm{d}^n_{\mathrm{tot}})_{n\ge 0}$ is defined by
\begin{gather}\label{total complex}
    C^n_{\mathrm{tot}}(G \rtimes_\rhd H;A)\coloneqq\bigoplus_{p+q=n}\mathbf C^{p,q}\ ,\qquad
\mathrm{d}^n_{\mathrm{tot}}\big|_{\mathbf C^{p,q}}(f)\coloneqq
\mathrm{d}_G^{p,q}f
-(-1)^{q}\,\bigl(\mathrm{d}_H^{p,q}f
\bigr)\ ,
\end{gather}
and we denote the cocycles and cohomology of the total complex by $Z^n_{\mathrm{tot}}(G \rtimes_\rhd H;A)$ and $H^n_{\mathrm{tot}}(G \rtimes_\rhd H;A)$, respectively. 

As explained in \cite{GGU13}, the total complex of the Lyndon-Hochschild-Serre double complex arises from a free 
$\mathbb{Z}[G \rtimes_\rhd H]$-resolution 
of $\mathbb{Z}$.
As it is quasi-isomorphic to the standard bar resolution, 
this results in an isomorphism of cohomology groups \cite[Lem.\,1.5]{GGU13}
\begin{equation}
  H^n_{\mathrm{tot}}(G \rtimes_\rhd H;A) \cong H^n(G \rtimes_\rhd H;A) \ .  
\end{equation}
The left hand side is computed via the total complex~\eqref{total complex}, and the right hand side via the usual group cohomology
as in~\eqref{eq:group-cohomology-differential}.

Deleting the $q=0$ row of the double complex $\mathbf{C}^{p,q}$, we denote the total complex of the truncated bicomplex and its cohomology by
$\bigl(C^n_{\mathrm{tot};1}(G \rtimes_\rhd H;A),\, \mathrm{d}^n_{\mathrm{tot};1}\bigr)$ and $H^n_{\mathrm{tot};1}(G \rtimes_\rhd H;A)$. The following proposition is shown in \cite[Lem.\,1.5,\,1.9]{GGU13}:

\begin{proposition}\label{group cohomology classification H-crossed action}
There is a decomposition
\begin{equation}
    H^n(G \rtimes_\rhd H;A) \; \cong \; H^n_{\mathrm{tot}}(G \rtimes_\rhd H;A) \; \cong \; H^n(H; A) \oplus H^n_{\mathrm{tot};1}(G \rtimes_\rhd H;A) \  .
\end{equation}
\end{proposition}

\subsection{$H$-crossed actions on $\mathrm{vec}_G$}\label{sec:H-crossed-vec_G}

Fix a ground field $\mathbb{K}$ and denote by $\mathrm{vec}_G$ the category of finite-dimensional $G$-graded $\mathbb{K}$-vector spaces.
Throughout this subsection, we work in the $\mathbb K$-linear setting, where relevant functors are required to be $\mathbb K$-linear rather than merely additive.
For $V \in \mathrm{vec}_G$, denote by $V_g \subset V$ its homogeneous component of degree $g$, such that $V = \bigoplus_{g \in G} V_g$. By $v_g \in V$ we will denote an arbitrary element of $V_g$. Morphisms in $\mathrm{vec}_G$ are degree preserving linear maps, i.e.\ $f \colon V \to W$ such that $f(V_g) \subset W_g$, and we write $f_g$ for the restriction of $f$ on $V_g$. As representatives of the isomorphism classes of simple objects in $\mathrm{vec}_G$ we take $\{\mathbb{K}_g\}_{g\in G}$, where $\mathbb{K}_g$ is
the one-dimensional vector space concentrated in degree $g$.  We use the unit of $\mathbb{K}$ as a preferred basis $\mathbf{1}_g \in \mathbb{K}_g$ in each $\mathbb{K}_g$.

Every normalised $3$-cocycle $\omega:G^3 \to \mathbb{K}^\times$ gives rise to a $G$-graded monoidal structure on $\mathrm{vec}_G$ as follows:
\begin{itemize}
    \item the $\mathbb{K}$-bilinear
    functor $\otimes\colon \mathrm{vec}_G \times \mathrm{vec}_G \to \mathrm{vec}_G$ is the tensor product of the underlying vector spaces, with the grading given by $(V \otimes W)_g = \bigoplus_{k \in G} V_k \otimes W_{k^{-1}g}$;

    \item the tensor unit is $\mathbb{K}_1$;
    
    \item the associator $\alpha\colon \otimes \circ (\mathrm{id}\times\otimes) \Rightarrow \otimes \circ (\otimes \times \mathrm{id})$ is given by
    \begin{equation}
    \begin{split}
    \alpha_{U,V,W} \colon
    U \otimes (V \otimes W)
    &\longrightarrow
    (U \otimes V) \otimes W  \\
    u_g \otimes (v_l \otimes w_k)
    &\longmapsto \omega(g,l,k) \, (u_g \otimes v_l) \otimes w_k \ ;
    \end{split}
    \end{equation}   

    \item the left and right unitors are those of vector spaces.
\end{itemize}
The resulting $G$-graded monoidal category is denoted by $\mathrm{vec}_G^\omega$.

Up to $G$-graded monoidal equivalence, all $G$-graded monoidal structures on $\mathrm{vec}_G$ are of this form, and  $\mathrm{vec}_G^\omega$ is $G$-graded monoidally equivalent to $\mathrm{vec}_G^{\omega'}$ if and only if 
$[\omega] = [\omega']$ in $H^3(G;\mathbb{K}^\times)$, see~\cite[Prop.\,2.6.1]{EGNO15}.

\begin{remark}
Note that monoidal structures (with no grading considered) on $\mathrm{vec}_G$ are classified by $H^3(G;\mathbb{K}^\times) \slash \mathrm{Out} (G)$~\cite[Prop.\,2.6.1]{EGNO15}. In the ungraded setting, a monoidal equivalence may permute homogeneous components via an automorphism of $G$, whereas a $G$-graded equivalence is necessarily degree-preserving. 
\end{remark}

The category $\mathrm{vec}_G^\omega$ is also rigid. For $V\in \mathrm{vec}_G$, we set $^{*}V =V^{*}$ to be the dual vector space
with graded components
\begin{equation}
    ({}^{*}V)_{g} = (V^*)_{g}     
    \coloneqq (V_{g^{-1}})^* \ .
\end{equation}
For each $g\in G$, fix a basis $\{v_{g,i}\}_i$ of $V_g$ and the corresponding dual basis
$\{v^{g,i}\}_i$ of $(V_g)^{*}=({V}^{*})_{g^{-1}}$. Define evaluation and coevaluation by
\begin{equation}
\begin{aligned}
\mathrm{ev}_V\colon {V}^{*} \otimes V &\to \mathbb{K}_1\ ,\\
f_{g^{-1}} \otimes v_g &\mapsto \omega(g,g^{-1},g)\, f_{g^{-1}}(v_g) \ ,
\end{aligned}
\qquad\qquad
\begin{aligned}
\mathrm{coev}_V\colon \mathbb{K}_1 &\to V \otimes {V}^{*} \ ,\\
\mathbf{1}_{1} &\mapsto \sum_{g \in G}\sum_i v_{g,i} \otimes v^{g,i} \ ,
\end{aligned}
\end{equation}
and similarly for the right dual, see~\cite[Ex.\,2.10.14]{EGNO15}.

\medskip

Fix a left action of $H$ on $G$ by group automorphisms. We now classify $H$-crossed $G$-graded monoidal structures on $\mathrm{vec}_G$ up to $H$-equivariant $G$-graded monoidal equivalence. 
This is also considered in~\cite[Sec.\,5]{GJ21}, but without the cohomological interpretation we give here.

The first piece of data is a $G$-graded monoidal structure on $\mathrm{vec}_G$, which, as we saw above, is given by a normalised 3-cocycle
\begin{equation}\label{eq:vecG-total-cochain-cond-1}
\omega \colon
G^3 \to \mathbb{K}^\times \ .
\end{equation}

Next we give a monoidal functor 
$\phi(h) \coloneqq {}^h(-)\colon \mathrm{vec}_G^\omega \to \mathrm{vec}_G^\omega$ for a fixed $h \in H$ as in \eqref{eq:H-crossed-phi-functor}. Without loss of generality we may take the underlying linear functor to be
\begin{equation}\label{eq: degree shifting notation}
        ({}^hV)_g \coloneqq V_{h^{-1} \rhd g}\ , \qquad
        ({}^hf)_g \coloneqq f_{h^{-1} \rhd g} \ .
\end{equation}
For example, for $g,k \in G$, $h \in H$, this gives
$({}^h(\mathbb{K}_g))_k = (\mathbb{K}_g)_{h^{-1} \rhd k}$ which is nonzero only for $g = h^{-1} \rhd k$. So on simple objects we have ${}^h(\mathbb{K}_g) \cong \mathbb{K}_{h \rhd g}$.

For the monoidal structure, the coherence isomorphisms $h_{U,V}$ in \eqref{monoidal structure h-action} are necessarily of the form
\begin{equation}
   h_{U,V}\colon {^hU \otimes {^hV}} \xrightarrow[]{\sim} {^h(U \otimes V)} \ ,
    \qquad
    u_{g}\otimes v_{k} \mapsto \Omega(h,g,k)\;  u_{g}\otimes v_{k} \ .
\end{equation}
Here, the degree subscripts for $u$, $v$ refer to the degrees in the underlying vector spaces $U$, $V$, e.g.\ $u_{g} \in U_{g} = ({}^hU)_{h \rhd g}$, and $\Omega$ is
a function
    \begin{equation}
        \Omega\colon H \times G^2 \to \mathbb{K}^\times \ .
    \end{equation}
Since ${}^h\mathbbm{1} = \mathbbm{1}$, by Remark~\ref{free choice up to monoidal natural equiv} we can choose $h_{\mathbbm{1}}$ in \eqref{h-action unitor} to be $h_{\mathbbm{1}} = \mathrm{id}_\mathbbm 1$.
The conditions in \eqref{action C'} become, for all $g,k,l \in G$, $h \in H$,
\begin{subequations}\label{eq-system-ad}
\begin{align}
&\Omega(h,1_G,g) = \Omega(h,g,1_G) = 1 \ , \label{eq-system-d}\\
&\frac{\omega(g,k,l)}{\omega(h\rhd g, h\rhd k, h\rhd l)}
  = \frac{\Omega(h,gk,l)\;\Omega(h,g,k)}{\Omega(h,k,l)\;\Omega(h,g,kl)} \label{eq-system-a} \ .
\end{align}
\end{subequations}

Finally we need the monoidal structure of the functor
\begin{equation}
    \varphi\colon \underline{H} \to \mathrm{Aut}_\otimes(\mathrm{vec}_G^\omega) \ , \qquad h \mapsto \phi(h)\ .
\end{equation}
For any $h_1, h_2 \in H$, the isomorphism $T^{h_1, h_2}$ in \eqref{eq:phi-mon-T-def} is of the form
    \begin{equation}
        (T^{h_1, h_2})_U \colon {}^{h_1}({}^{h_2}U) \xrightarrow{\sim} {}^{h_1h_2}U \ ,
\qquad
        u_{g} \mapsto t(h_1,h_2,g)\; 
        u_{g}
    \end{equation}
for a function
    \begin{equation}
        t\colon H^2 \times G \to \mathbb{K}^\times \ .
    \end{equation}
We have ${}^{1_H}(-) = \mathrm{id}$, and    
by Remark~\ref{free choice up to monoidal natural equiv} we can choose $\eta_U = \mathrm{id}_U$ in \eqref{eq:phi-monoidal-unit-iso}. Conditions \eqref{action C} and \eqref{action C''} become, for all $g,k \in G$, $h_1, h_2, h_3 \in H$,
\begin{subequations}\label{eq-system-bce}
\begin{align}
& \frac{\Omega(h_2,g,k)\;\Omega(h_1,h_2\rhd g, h_2\rhd k)}{\Omega(h_1h_2,g,k)}
  = \frac{t(h_1,h_2,k)\;t(h_1,h_2,g)}{t(h_1,h_2,gk)} \ , \label{eq-system-b}\\
&\frac{t(h_2,h_3,g)\;t(h_1,h_2h_3, g)}{t(h_1h_2,h_3,g)\;t(h_1,h_2, h_3 \rhd g)}
  = 1 \ , \label{eq-system-c}\\
&t(1_H,h_2,g) = t(h_1,1_H,g) = t(h_1,h_2,1_G) = 1 \ ,
\qquad \Omega(1_H,g,k) = 1 \ . 
\label{eq-system-e}
\end{align}
\end{subequations}

\medskip

We can now make the connection to the total complex
introduced in \eqref{total complex}. We take $A = \mathbb{K}^\times$, which -- somewhat confusingly -- means that we now write the group operation multiplicatively,
rather than additively as in \eqref{total complex}.
The above data $\omega$, $\Omega$, $t$ can be assembled into a 3-cochain as
\begin{equation}
\kappa \coloneqq (\omega,\Omega,t)\in \mathbf C^{0,3}\oplus \mathbf C^{1,2}\oplus \mathbf C^{2,1}
= C^3_{\mathrm{tot};1}(G \rtimes_\rhd H;\mathbb{K}^\times) \ .
\end{equation}
The conditions in \eqref{eq:vecG-total-cochain-cond-1}, \eqref{eq-system-ad}, and \eqref{eq-system-bce} are equivalent to $\kappa$ being normalised and
\begin{equation}
    \mathrm{d}^3_{\mathrm{tot};1} \kappa = 1 \ ,
\end{equation}
i.e.\ to the simple statement that $\kappa$ is a normalised 3-cocycle in $C^3_{\mathrm{tot};1}(G \rtimes_\rhd H;\mathbb{K}^\times)$;
recall the truncated complex above Proposition~\ref{group cohomology classification H-crossed action}. For example, the component of $\mathrm{d}^3_{\mathrm{tot};1} \kappa$ in $\mathbf{C}^{2,2}$ is 
$(\mathrm{d}_G^{2,1}t)(h_1,h_2,g,k)\,(\mathrm{d}_H^{1,2}\Omega)(h_1,h_2,g,k)^{-1}$, which results precisely in \eqref{eq-system-b}.

In summary, a normalised 3-cocycle
\begin{equation}
    \kappa=(\omega,\Omega,t)
    \in Z^3_{\mathrm{tot};1}(G\rtimes_\rhd H;\mathbb K^\times)
\end{equation}
determines an $H$-crossed $G$-graded monoidal structure on
$\mathrm{vec}_G$. We denote the resulting category by
\begin{equation}\label{eq:vecG-kappa-def}
    \mathrm{vec}_G^\kappa \ .
\end{equation}

\medskip

Next we ask when two $H$-crossed $G$-graded monoidal structures on $\mathrm{vec}_G$, defined by normalised 3-cocycles
\begin{equation}
(\omega,\Omega,t),\ (\widetilde\omega,\widetilde\Omega,\widetilde t)\in Z^3_{\mathrm{tot};1}(G \rtimes_\rhd H;\mathbb{K}^\times) \ ,
\end{equation}
are equivalent in the sense of Definition~\ref{crossed functor}. We construct an $H$-equivariant $G$-graded monoidal functor
\begin{equation}
    (F,\xi):\ \mathrm{vec}_G^{(\omega,\Omega,t)}\longrightarrow \mathrm{vec}_G^{(\widetilde\omega,\widetilde\Omega,\widetilde t)} \ .
\end{equation}
Since $\omega$ is already used to denote the monoidal structure, here we use $\xi$ for the monoidal natural isomorphism denoted by $\omega$ in Definition~\ref{crossed functor}.
Up to natural isomorphism,
we may assume the underlying $\mathbb{K}$-linear functor $F=\mathrm{id}$. The monoidal structure on $F$ is then determined by a function
\begin{equation}
    \alpha:G^2\to \mathbb{K}^\times
\end{equation}
via
\begin{equation}
    F_{U,V}:U\otimes V\to U\otimes V \ ,\qquad
u_g\otimes v_k\longmapsto \alpha(g,k)\,u_g\otimes v_k \ ,
\qquad
F_{\mathbbm 1}=\mathrm{id} \ ,
\end{equation}
where $F_{\mathbbm 1}=\mathrm{id}$ uses Remark~\ref{unital assumption on functors}. The condition that $F$ is monoidal is equivalent to, for all $g,k,l\in G$,
\begin{equation}\label{eq:vecG-functor-monoidal-cond}
\frac{\widetilde{\omega}(g,k,l)}{\omega(g,k,l)}
=
\frac{\alpha(k,l)\,\alpha(g,kl)}{\alpha(gk,l)\,\alpha(g,k)} \ ,
\qquad
\alpha(1_G,g)=\alpha(g,1_G)=1 \ .
\end{equation}
In addition, we specify a family $(\xi^h)_{h \in H}$ of natural isomorphisms of the form
\begin{equation}
\xi^h \colon {}^hU \to {}^hU 
~,\qquad
u_{g} \mapsto \beta(h,g)\;u_{g}
\end{equation}
for a function
\begin{equation}
    \beta \colon H \times G \to \mathbb{K}^\times \ .
\end{equation}
Conditions~\eqref{crossed functor definition} become, for all  $g,k \in G$, $h,h_1,h_2 \in H$,
\begin{subequations}\label{coboundary condition}
    \begin{align}
         &
         \frac{\widetilde{\Omega}(h,g,k)}{\Omega(h,g,k)}=\frac{\beta(h,k)\; \beta(h,g)}{\beta(h,gk)}\frac{\alpha(h\rhd g, h \rhd k)}{\alpha(g, k)} \ , \label{eq:vecG-omegah-monoidal}\\
        &
        \frac{\widetilde{t}(h_1,h_2,g)}{t(h_1,h_2,g)} = \frac{\beta(h_2,g)\; \beta(h_1,h_2 \rhd g)}{\beta(h_1h_2,g)} \ ,\\
        &
        \beta(1_H, g) = \beta(h, 1_G) = 1 \ .
    \end{align}
\end{subequations}

If we collect $\alpha$ and $\beta$ into a 2-cochain
\begin{equation}
(\alpha,\beta)\in \mathbf C^{0,2}\oplus \mathbf C^{1,1}
= C^2_{\mathrm{tot};1}(G \rtimes_\rhd H;\mathbb{K}^\times) \ ,
\end{equation}
then conditions~\eqref{eq:vecG-functor-monoidal-cond} and~\eqref{coboundary condition} are equivalent to $(\alpha,\beta)$ being normalised and
\begin{equation}\label{d2tot-via-alpha-beta}
\mathrm{d}^2_{\mathrm{tot};1}(\alpha, \beta) = 
(\widetilde\omega,\widetilde\Omega,\widetilde t) \cdot (\omega,\Omega,t)^{-1} \ .  
\end{equation}

In summary, we have shown,
recall Proposition~\ref{group cohomology classification H-crossed action}:
\begin{theorem}\label{thm:VecG-Hcrossed-classification-H3tot}
For a fixed action $\rhd$ of $H$ on $G$,
$H$-crossed $G$-graded monoidal structures on $\mathrm{vec}_G$ are, up to equivalence, 
classified by $H^3_{\mathrm{tot};1}(G \rtimes_\rhd H;\mathbb{K}^\times)$.
\end{theorem}

The next lemma helps with the computation of $H^3_{\mathrm{tot};1}$. First note that there is a canonical map
\begin{equation}\label{res map}
    \mathrm{res}_G\colon H^3_{\mathrm{tot};1}(G \rtimes_\rhd H;\mathbb{K}^\times) \to H^3(G;\mathbb{K}^\times) \ , \qquad [\omega,\Omega,t] \mapsto [\omega] \ .
\end{equation}

\begin{lemma}\label{prop:direct prod simplify}
    If the $H$-action on $G$ is trivial, then $H^3_{\mathrm{tot};1}(G \times H;\mathbb{K}^\times)$ admits a canonical decomposition
\begin{equation}
    H^3_{\mathrm{tot};1}(G \times H;\mathbb{K}^\times)\cong H^3(G;\mathbb{K}^\times)\oplus \mathrm{Ker}(\mathrm{res}_G) \ ,
\end{equation}
and $\mathrm{Ker}(\mathrm{res}_G)$ fits into the short exact sequence
\begin{equation}\label{eq:HS-filtration}
    0\longrightarrow H^2(H;\widehat{G})\longrightarrow \mathrm{Ker}(\mathrm{res}_G)\longrightarrow
H^1\!\bigl(H;H^2(G;\mathbb{K}^\times)\bigr)\longrightarrow 0 \ ,
\end{equation}
where $\widehat{G}\coloneqq\Hom(G,\mathbb{K}^\times)$.
\end{lemma}

\begin{proof}
    If the $H$-action on $G$ is trivial, then 
\begin{equation}
    s_G\colon H^3(G;\mathbb{K}^\times) \to H^3_{\mathrm{tot};1}(G \times H;\mathbb{K}^\times) \ , \qquad [\omega] \mapsto [\omega,1,1]
\end{equation}
is a well-defined section of $\mathrm{res}_G$, so there is a canonical decomposition
\begin{equation}
    H^3_{\mathrm{tot};1}(G \times H;\mathbb{K}^\times)\cong H^3(G;\mathbb{K}^\times)\oplus \mathrm{Ker}(\mathrm{res}_G) \ .
\end{equation}
Define
\begin{equation}
    i\colon H^2(H;\widehat{G})\longrightarrow \mathrm{Ker}(\mathrm{res}_G) \ , \qquad [t] \mapsto [1,1,t] \ ,
\end{equation}
and
\begin{equation}
    p\colon \mathrm{Ker}(\mathrm{res}_G) \longrightarrow H^1\!\bigl(H;H^2(G;\mathbb{K}^\times)\bigr) \ , \qquad [1, \Omega, t] \mapsto [\Omega] \ .
\end{equation}
It is straightforward to check that $i$ and $p$ are well-defined and make~\eqref{eq:HS-filtration} an exact sequence.
\end{proof}

\subsection{Fibre products for $\mathrm{vec}_G^\kappa$}\label{sec: fibre product 2}

Recall the fibre product of $H$-crossed $G$-graded categories from Section~\ref{sec: fibre product}. For the categories $\mathrm{vec}_G^{\kappa}$ as in \eqref{eq:vecG-kappa-def} we have:

\begin{proposition}\label{prop:Vec-addition}
For any normalised
$\kappa_1, \kappa_2 \in Z^3_{\mathrm{tot};1}(G \rtimes_\rhd H;\mathbb{K}^\times)$, there is a canonical equivalence of
$H$-crossed $G$-graded monoidal categories
\begin{equation}
    \mathrm{vec}_G^{\kappa_1} \boxtimes_G \mathrm{vec}_G^{\kappa_2} \;\cong\; \mathrm{vec}_G^{\kappa_1 \kappa_2} \ ,
\end{equation}
where $\kappa_1\kappa_2$ is given by pointwise multiplication.
\end{proposition}

\begin{proof}
Recall that an object of \(\mathrm{vec}_G^{\kappa_1}\boxtimes_G \mathrm{vec}_G^{\kappa_2}\) is a finite direct sum of objects
\(U_g\boxtimes V_g\) with \(g\in G\) and \(U_g,V_g\) vector spaces concentrated in degree \(g\).
We denote by \(\otimes\) the graded tensor product in \(\mathrm{vec}_G\) and by \(\bar\otimes\) the usual tensor product of
underlying vector spaces.

We want to define a strictly unital \(H\)-equivariant \(G\)-graded monoidal equivalence
\begin{equation}
(F,\xi):\ \mathrm{vec}_G^{\kappa_1}\boxtimes_G \mathrm{vec}_G^{\kappa_2}\longrightarrow \mathrm{vec}_G^{\kappa_1\kappa_2} \ .
\end{equation}

We start with a $\mathbb{K}$-linear functor $F\colon \mathrm{vec}_G\boxtimes_G \mathrm{vec}_G\to \mathrm{vec}_G$ given on pure tensors by
    \begin{equation}
        U_g \boxtimes V_g \mapsto U_g\bar\otimes V_g \ , \qquad f\boxtimes f' \coloneqq f\bar\otimes f' \ ,
    \end{equation}
    where $U_g\bar\otimes V_g$ is viewed as a vector space concentrated in degree $g$. It is clear that $F$ is a $\mathbb K$-linear equivalence.
    
For any homogeneous objects \(U_g\boxtimes V_g\) and \(U_l\boxtimes V_l\), define the monoidal structure
\begin{equation}
    F_{U_g\boxtimes V_g,\;U_l\boxtimes V_l}:\ F(U_g\boxtimes V_g)\otimes
    F(U_l\boxtimes V_l)\to
F\bigl((U_g\otimes U_l)\boxtimes (V_g\otimes V_l)\bigr)
\end{equation}
to be the canonical linear isomorphism
\begin{equation}
    (U_g\bar\otimes V_g)\bar\otimes (U_l\bar\otimes V_l)\ \xrightarrow{\cong}\ (U_g\bar\otimes U_l)\bar\otimes (V_g\bar\otimes V_l) \ ,
\qquad
(u_g\bar\otimes v_g)\bar\otimes(u'_l\bar\otimes v'_l)\mapsto (u_g\bar\otimes u'_l)\bar\otimes (v_g\bar\otimes v'_l) \ .
\end{equation}
It is straightforward to check that the coherence conditions are satisfied.

Finally, for any $h\in H$ and any homogeneous object $U_g\boxtimes V_g$, we define $\xi^h_{U_g\boxtimes V_g}$ to be the identity.
The $H$-equivariance conditions~\eqref{crossed functor definition} are immediate.
\end{proof}

 \begin{corollary}\label{cor:Vec-invertible}
 For every such $\kappa$, $\mathrm{vec}_G^\kappa$ is $\boxtimes_G$-invertible with the inverse
 $\mathrm{vec}_G^{\kappa^{-1}}$.
 \end{corollary}

Recall that for fixed groups $G$ and $H$ together with a specified action
$\rhd$ of $H$ on $G$, we introduced in Definition~\ref{def: Cr_H,G} a
$2$-category $\mathsf{Cr}_{H,G,\rhd}$. Given $\kappa \in Z^3_{\mathrm{tot};1}(G \rtimes_\rhd H;\mathbb{K}^\times)$, we
can twist any object
$\mathcal C \in \mathsf{Cr}_{H,G,\rhd}$ by $\kappa$ by setting
\begin{equation}\label{eq:twisting-as-fibred-product}
    \mathcal C^\kappa
    \coloneqq
    \mathrm{vec}_G^\kappa \boxtimes_G \mathcal C \ .
\end{equation}
In this way we obtain an action of $H^3_{\mathrm{tot};1}(G \rtimes_\rhd H;\mathbb{K}^\times)$ on the equivalence classes in $\mathsf{Cr}_{H,G,\rhd}$.

\subsection{Example: $G = \mathbb{Z}/4\mathbb{Z}$}\label{Z/4 example}

We study concrete examples with fixed $G = \Z/4\Z = \{0,1,2,3\}$. Note that $\mathrm{Aut}(\Z/4\Z) \cong \Z/2\Z$, and an automorphism is uniquely determined by whether it sends $1$ to $1$ or $3$.
Consider the following 4 cases:
\begin{enumerate}
    \item $H_1 = \{*\}$ acts on $G$ trivially;
    \item $H_2 = \Z/2\Z$ acts on $G$ trivially;
    \item $H_3 = \Z/4\Z$ acts on $G$ trivially;
    \item $H_4 = \Z/4\Z$ acts on $G$ via the quotient map $q\colon \Z/4\Z \to \Z/2\Z \cong \mathrm{Aut}(\Z/4\Z)$;
\end{enumerate}
Fix an algebraically closed field $\mathbb{K}$ with $\mathrm{char}(\mathbb{K}) \neq 2$ and a primitive 4th root of unity
\begin{equation}\label{eq:example-i}
     i \in \mathbb{K}^\times \ .
\end{equation}
We classify $H_j$-crossed
$G$-graded monoidal structures
on $\mathrm{vec}_G$ for $j = 1, 2, 3, 4$.

\medskip

Note that for any finite cyclic group $K = \mathbb Z/n\mathbb Z$ with a chosen generator $t$, and any $\mathbb Z[K]$-module~$A$, the standard periodicity theorem for cyclic group cohomology yields~\cite[Thm.\,6.2.2]{We94}
\begin{equation}\label{eq:cyclic-group-coh-period}
    H^m(K;A) \cong
    \begin{cases}
        A^K \ , & m = 0 \ , \\[0.5em]
        \ker(N)/(t-1)A \ , & m \geq 1 \text{ odd} \ , \\[0.5em]
        A^K/N(A) \ , & m \geq 2 \text{ even} \ ,
    \end{cases}
\end{equation}
where $t$ is viewed as a linear map $t\colon A \to A$, and
\begin{equation}
    A^K \coloneqq \{a \in A \mid t\cdot a = a\} \ , 
    \qquad
    N \coloneqq 1+t+\cdots+t^{n-1} \ .
\end{equation}
In particular, for $K = G = \Z/4\Z$ with untwisted coefficient $\mathbb{K}^\times$, we have $H^2(\Z/4\Z;\mathbb{K}^\times) = 0$. Following Proposition~\ref{prop:direct prod simplify}, for the first three cases $j = 1,2,3$, the corresponding computation for $H^3_{\mathrm{tot};1}(G \times H_j;\mathbb{K}^\times)$ simplifies to
\begin{equation}
    H^3_{\mathrm{tot};1}(G \times H_j;\mathbb{K}^\times) \cong H^3(G; \mathbb{K}^\times) \oplus H^2(H_j;\widehat{G}) \ .
\end{equation}
We present below the classification results for the four cases described above. 

\medskip

\noindent
$\bullet$ \textit{Case (\romannumeral 1):} We have
\begin{equation}
    H^3_{\mathrm{tot};1}(G \times H_1;\mathbb{K}^\times)\;\cong\;H^3(\Z/4\Z;\mathbb{K}^\times)\;\cong\;\Z/4\Z \ .
\end{equation}
A generator $[\omega]$ for $H^3(\Z/4\Z; \mathbb{K}^\times)$ is given by the 3-cocycle (see e.g.~\cite[Prop.\,2.3]{HLY12})
\begin{equation}\label{generator Z/4Z-cocycle}
\omega \colon (\Z/4\Z)^3 \longrightarrow \mathbb{K}^\times,
\qquad
\omega(a,b,c)\;\coloneqq\;
\begin{cases}
i^c \ ,& a+b \geq 4 \ ,\\
1 \ ,& a+b <4\ .
\end{cases}
\end{equation}
A generator of $H^3_{\mathrm{tot};1}(G \times H_1;\mathbb{K}^\times)$ is then given by $[\omega,1,1]$.

\medskip

\noindent
$\bullet$ \textit{Case (\romannumeral 2):} We have
\begin{equation}
H^3_{\mathrm{tot};1}(G \times H_2;\mathbb{K}^\times)\;\cong\;H^3(\Z/4\Z;\mathbb{K}^\times)\;\oplus\;H^2(\Z/2\Z;\widehat{\Z/4\Z})
\;\cong\;\Z/4\Z\;\oplus\;\Z/2\Z \ .
\end{equation}
A representative of the generator for $H^2(\Z/2\Z;\widehat{\Z/4\Z})$ is given by
\begin{equation}\label{eq:ti}
    t\colon (\Z/2\Z)^2 \times \Z/4\Z \longrightarrow \mathbb{K}^\times \ ,
\qquad
t(h_1,h_2,c)\;\coloneqq\;
i^{h_1 h_2 c} \ .
\end{equation}
Therefore, $[\omega,1,1]$ and $[1,1,t]$ are two generators of
$H^3_{\mathrm{tot};1}(G \times H_2;\mathbb{K}^\times)$.

\medskip

\noindent
$\bullet$ \textit{Case (\romannumeral 3):}
We have
\begin{equation}
H^3_{\mathrm{tot};1}(G \times H_3;\mathbb{K}^\times)\;\cong\;H^3(\Z/4\Z;\mathbb{K}^\times)\;\oplus\;H^2(\Z/4\Z;\widehat{\Z/4\Z})
\;\cong\;\Z/4\Z\;\oplus\;\Z/4\Z \ .
\end{equation}
A generator for $H^2(\Z/4\Z;\widehat{\Z/4\Z})$ is represented by the following cocycle
\begin{equation}\label{generator ti Z/4Z-cocycle}
     t' \colon (\Z/4\Z)^2 \times (\Z/4\Z)  \longrightarrow \mathbb{K}^\times \ ,
 \qquad
 t'(h_1,h_2,c)\;\coloneqq\;
 \begin{cases}
 i^c \ ,& h_1+h_2 \geq 4 \ ,\\
 1 \ ,& h_1+h_2 <4\ .
 \end{cases}
\end{equation}
Therefore, $[\omega,1,1]$ and $[1,1,t']$ are two generators of $H^3_{\mathrm{tot};1}(G \times H_3;\mathbb{K}^\times)$.

\medskip

\noindent
$\bullet$ \textit{Case (\romannumeral 4):}
This case is more involved. Note that the $\Z/4\Z$-action on $\Z/4\Z$ is given by
\begin{equation}\label{non-trivial Z/4Z-action}
    h \rhd g =
\begin{cases}
g \ ,& h \equiv 0 \; \mathrm{mod} \; 2 \ ,\\
-g \ ,& h \equiv 1 \; \mathrm{mod} \; 2 \ .
\end{cases}
\end{equation}
Functoriality of group cohomology gives $\Z/4\Z$-actions on $H^1(\Z/4\Z;\mathbb K^\times)$ and $H^3(\Z/4\Z;\mathbb K^\times)$.
Recall that the $H_4$-action on cochains is given by~\eqref{h-action notation}. On the level of cohomology, the induced $\Z/4\Z$-action on $H^1(\Z/4\Z;\mathbb K^\times)$ is given by the same formula as~\eqref{non-trivial Z/4Z-action}, and the $\Z/4\Z$-action on $H^3(\Z/4\Z;\mathbb K^\times)$ is trivial.

Consider the truncated double complex $\mathbf C^{p,q>0}$ and its associated Lyndon–Hochschild–Serre spectral sequence~\cite[Ch.\,7]{Br82}
\begin{equation}
    E_2^{p,q} = H^p\big(\Z/4\Z; H^q(\Z/4\Z; \mathbb{K}^\times)\big) \Longrightarrow H^{p+q}_{\mathrm{tot};1}(G\rtimes_q H_4;\mathbb{K}^\times) \ .
\end{equation}
We first compute the $E_2$-page in total degree 3:
\begin{itemize}
    \item $E_2^{0,3} = H^0\big(\Z/4\Z; H^3(\Z/4\Z; \mathbb{K}^\times)\big) \cong  H^0(\Z/4\Z; \Z/4\Z) \cong \Z/4\Z$;
    \item $E_2^{1,2} = H^1\big(\Z/4\Z; H^2(\Z/4\Z; \mathbb{K}^\times)\big) \cong  0$;
    \item $E_2^{2,1} = H^2\big(\Z/4\Z; H^1(\Z/4\Z; \mathbb{K}^\times)\big) \cong  H^2(\Z/4\Z; \Z/4\Z) \cong \Z/2\Z$.
\end{itemize}
Since $H^2(\Z/4\Z;\mathbb{K}^\times) = 0$, the differential $\mathrm{d}_2^{0,2} = 0$, and since we are in the truncated bicomplex, $\mathrm{d}_r^{2,1} = 0$ for all $r \geq 2$. Therefore, we have $E_\infty^{2,1} = E_2^{2,1}$. Consequently, the order of $H^3_{\mathrm{tot};1}(G\rtimes_q H_4;\mathbb{K}^\times)$ can only be 2, 4, or 8. We show
%COMMENT: See https://groupprops.subwiki.org/wiki/Group_cohomology_of_groups_of_order_16 : this gives homology H_3(Z4 x-semi Z4) = Z4 + Z4 + Z2. Cohomology (maps into K^*) give the same abelian group. Furthermore, H3(Z4,K*) = Z4, and so Prop 3.1 confirms Z4 + Z2.
\begin{equation}
    H^3_{\mathrm{tot};1}(G\rtimes_q H_4;\mathbb{K}^\times) \; \cong \; \Z/4\Z \;\oplus\; \Z/2\Z
\end{equation}
by exhibiting two independent generators of order 4 and 2:
\begin{itemize}
    \item Define $\omega \colon (\Z/4\Z)^3 \to \mathbb{K}^\times$ as~\eqref{generator Z/4Z-cocycle}, and define
\begin{equation}\label{eq: Omega_i}
    \Omega \colon \Z/4\Z \times (\Z/4\Z)^2 \to \mathbb{K}^\times \ , \qquad \Omega(h,a,b) \; = \; \begin{cases}
        i^b \ ,& h \equiv 1 \; \mathrm{mod} \; 2 \ , \; a \neq 0 \ ,\\
 1 \ ,& \textup{otherwise}.
    \end{cases}
\end{equation}
One can check that, for any $h \in \Z/4\Z$,
\begin{equation}
    \frac{\omega}{\omega^h} = \mathrm{d}_G \big(\Omega(h,-,-)\big) \ , \qquad \mathrm{d}_{H_4}\Omega = 1 \ .
\end{equation}
Therefore, $[\omega,\Omega,1]$ is a class in $H^3_{\mathrm{tot};1}$ and clearly $[\omega^4,\Omega^4,1^4] = [1,1,1]$. Moreover, the image of $[\omega,\Omega,1]$ under the restriction map $\mathrm{res}_G$ introduced in~\eqref{res map} is $[\omega] \in H^3(\Z/4\Z;\mathbb{K}^\times)$, which is of order $4$. Therefore, $[\omega,\Omega,1]$ can not have order 1 or 2, so it has order 4.

\item Define a normalised cochain
\begin{equation}\label{eq:t_-}
    t_-\colon (\Z/4\Z)^2\times \Z/4\Z \longrightarrow \mathbb{K}^\times \ , \qquad t_-(h_1,h_2,c)\;\coloneqq\;
\begin{cases}
1 \ ,& h_1+h_2 <4 \ ,\\
(-1)^c \ ,& h_1+h_2 \geq 4 \ .
\end{cases}
\end{equation}
One can check that $[1,1,t_-]$ is indeed a class in $H^3_{\mathrm{tot};1}$ of order 2, 
which is independent of $[\omega, \Omega, 1]$ (by applying $\mathrm{res}_G$ from~\eqref{res map}).
\end{itemize}

\section{$\chi$-Crossed braided and ribbon structures}\label{section crossed ribbon cat}

For a crossed module $\chi\colon G \to H$, we introduce the notion of a $\chi$-crossed braiding and ribbon structure on an $H$-crossed $G$-graded monoidal category (Definitions~\ref{crossed braiding def} and~\ref{crossed balanced}).
We give a cohomological classification of $\chi$-crossed braided structures on the category $\mathrm{vec}_G$ of $G$-graded vector spaces (Proposition~\ref{prop:classification chi-crossed on vec_G}). This recovers two well-known classification results: when $\chi = \mathrm{id}:G \to G$ with the conjugation action, it recovers the classification of $G$-crossed braided structures on $\mathrm{vec}_G$ (Proposition~\ref{H_qa for id chi}); when $\chi = * \colon G \to \{*\}$ for an abelian $G$, it reduces to the classification of $G$-graded braided structures on $\mathrm{vec}_G$ (Proposition~\ref{compare to ab cohomology}). In Section~\ref{different equiv relation}, we revisit the examples from Section~\ref{Z/4 example} to further consider their crossed braidings,
which gives an explicit example of 
$\chi$-crossed braidings leading to finer equivalence relations as compared to $G$-crossed braidings (Remark~\ref{rmk: different equiv for chi}).
We also show that every $\chi$-crossed braided category canonically carries both a $G$-crossed braided structure and an $H$-crossed braided structure (Proposition~\ref{chi crrossed implies G crossed}).

\subsection{Crossed modules}
Crossed modules were introduced by Whitehead while studying homotopy 2-types~\cite{Wh41,Wh49}. It is also well-known that crossed modules are models for strict 2-groups~\cite{BL04,No05}. We start by recalling the definition.

\begin{definition}\label{crossed module def}
    A \textit{crossed module} $\chi$ is a tuple $(G, H, \chi, \gamma)$, where
\begin{itemize}
    \item $G$ and $H$ are groups,
    \item $\chi\colon G \to H$ and $\gamma\colon H \to \mathrm{Aut}(G)$ are group homomorphisms,
\end{itemize}
such that for any $g, g' \in G$, $h \in H$, the following two identities hold:
\begin{gather}\label{crossed module long}
    h\chi(g)h^{-1} = \chi(\gamma(h)(g)) \ , \qquad gg'g^{-1} = \gamma(\chi(g))(g') \ .
\end{gather}
\end{definition}

The second identity in~\eqref{crossed module long} is also called the \textit{Peiffer identity}. To simplify our notation, in the rest of this paper a crossed module will just be denoted by $\chi\colon G \to H$. The $H$-action on $G$ is left implicit in this notation, and we will denote the action of $h \in H$ on $g \in G$ by $h\rhd g$ as in Section~\ref{sec:crossed-actions}. Below we will furthermore write $\chi_g$ instead of $\chi(g)$, so that the identities in \eqref{crossed module long} read
\begin{equation}\label{crossed module}
    h \chi_g h^{-1} = \chi_{h \rhd g} \ , \qquad gg'g^{-1} = \chi_g \rhd g' \ .
\end{equation}

\begin{remark}
    For any group $G$, we have the following canonical crossed modules: $\mathrm{id}:G \to G$ with the conjugation action; $*\colon G \to \{*\}$ with trivial action ($G$ needs to be abelian in this case); $*\colon \{*\} \to G$ with trivial action.
\end{remark}

A morphism between two crossed modules $\chi\colon G \to H$ and $\chi'\colon G' \to H'$ is a pair of group homomorphisms $(\alpha\colon G \to G', \beta\colon H \to H')$ such that the following two diagrams commute:
\begin{equation}\label{usual crossed module morphism}
\begin{tikzcd}[column sep = 3em, row sep = 1.5em]
G \ar[r, "\chi"]\ar[d, "\alpha"]
&
H \ar[d, "\beta"']
&
H\times G \ar[r, "\beta \times \alpha"]\ar[d,"\rhd"]
&
H' \times G' \ar[d, "\rhd'"']\\
G' \ar[r, "\chi'"]
&
H'
&
G \ar[r, swap,"\alpha"']
&
G'
\end{tikzcd}
\end{equation}

\begin{lemma}\label{lem: K central I normal}
Let $\chi\colon G \to H$ be a crossed module, then we have
\begin{enumerate}
    \item $\mathrm{Ker}(\chi)$ is a central subgroup in $G$;
    \item $\mathrm{Im}(\chi)$ is a normal subgroup in $H$.
\end{enumerate}
\end{lemma}

\begin{proof}
    If $k \in \mathrm{Ker}(\chi)$, then for any $g \in G$ it holds that $kgk^{-1} = \chi_k \rhd g = 1_H \rhd g = g$.
    If $h = \chi_g$ for some $g \in G$, then for any $h' \in H$ we have $h'hh'^{-1} = \chi_{h'\rhd g} \in \mathrm{Im(\chi)}$.
\end{proof}

Let $\chi\colon G \to H$ be a crossed module and denote
\begin{equation}
    K \coloneqq \mathrm{Ker}(\chi) \ , \qquad I \coloneqq \mathrm{Im}(\chi) \ , \qquad Q\coloneqq \mathrm{Coker}(\chi) \ .
\end{equation}
Then, $\chi$ gives rise to an exact sequence:
\begin{equation}
\begin{tikzcd}[column sep = 2.5em]
1 \ar[r, ""]
&
K \ar[r,"i"]
&
G \ar[r,"\chi"]
&
H \ar[r,"p"]
&
Q \ar[r, ""]
&
1 \ .
\end{tikzcd}
\end{equation}
We can decompose $\chi$ into $G \xrightarrow{\pi}G/K\xrightarrow{\cong}I\xrightarrow{\iota}H$ and then bend this exact sequence to obtain the following commutative diagram, where the top horizontal map is trivial:
\begin{equation}\label{crossed module extension}
\begin{tikzcd}[column sep = 4em, row sep = 1.5em]
K\ar[d,"i"]\ar[r, "p \circ \chi \circ i"]
&
Q\\
G \ar[r, "\chi"]\ar[d, "\pi"]
&
H\ar[u, "p"]\\
G/K \ar[r, "\mathrm{id}"]
&
G/K \cong I \ar[u, "\iota"]
\end{tikzcd}
\end{equation}
Note that each column in the diagram above is a group extension, and each row is a crossed module. Indeed, in the bottom row, $G/K$ acts on itself by conjugation, while in the top row, the action of $Q$ on $K$ is induced by the action of $H$ on $G$. Explicitly, this action is given by $[h]\rhd g \coloneqq h \rhd g$, which is well-defined due to
the fact that $I$ acts on $K$ trivially by Lemma~\ref{lem: K central I normal}.

\subsection{$\chi$-Crossed braided and ribbon
categories}\label{crossed ribbon cat}
Let $\chi\colon G \to H$ be a crossed module. Throughout, when we speak of an $H$-crossed $G$-graded monoidal category, we implicitly mean that the underlying action of $H$ on $G$ is the one in $\chi$.

\begin{definition}\label{crossed braiding def}
    Let $\mathcal{C}$ be an $H$-crossed $G$-graded monoidal category. A \textit{$\chi$-crossed braiding} on $\mathcal{C}$ is a 
family of isomorphisms
\begin{equation}\label{crossed braiding}
    \big\{\,C_{U,V}\colon U \otimes V \to 
    ({}^{\chi_{|U|}}V) \otimes U
    \,\big\}_{U,V \in \mathcal{C}_{\mathrm{hom}}}
\end{equation}
such that
\begin{itemize}
    \item for any morphisms $f\colon U \to U'$, $g\colon V \to V'$ between homogeneous objects, the following diagram commutes:
    \begin{equation}
    % https://q.uiver.app/#q=WzAsNCxbMCwwLCJVIFxcb3RpbWVzIFYiXSxbMywwLCJee1xcY2hpX3t8VXx9fVYgXFxvdGltZXMgVSJdLFswLDIsIlUnIFxcb3RpbWVzIFYnIl0sWzMsMiwiXntcXGNoaV97fFV8fX1WJyBcXG90aW1lcyBVJyJdLFswLDEsIkNfe1UsVn0iXSxbMCwyLCJmIFxcb3RpbWVzIGciLDJdLFsyLDMsIkNfe1UnLFYnfSIsMl0sWzEsMywiXntcXGNoaV97fFV8fX1nIFxcb3RpbWVzIGYiXV0=
\begin{tikzcd}\tag{B1} \label{B1}
	{U \otimes V} && {^{\chi_{|U|}}V \otimes U} \\
	{U' \otimes V'} && {^{\chi_{|U|}}V' \otimes U'}
	\arrow["{C_{U,V}}", from=1-1, to=1-3]
	\arrow["{f \otimes g}", from=1-1, to=2-1]
    \arrow["\big({}^{\chi_{|U|}}g \big) \otimes f"', from=1-3, to=2-3]
	\arrow["{C_{U',V'}}", from=2-1, to=2-3]
\end{tikzcd}
    \end{equation}

    \item for any $h\in H$, $a,b,c \in G$, and any homogeneous objects $U,V,W\in\mathcal{C}$ of degrees $a,b,c$,
respectively, the following diagrams commute:
    \begin{equation}\tag{B2} \label{B2}
    % https://q.uiver.app/#q=WzAsNyxbMCwwLCIoVSBcXG90aW1lcyBWKVxcb3RpbWVzIFciXSxbMiwwLCJ7fV57XFxjaGlfe2FifX1XIFxcb3RpbWVzIChVXFxvdGltZXMgVikiXSxbMCwxLCJVIFxcb3RpbWVzIChWXFxvdGltZXMgVykiXSxbMCwyLCJVIFxcb3RpbWVzICh7fV57XFxjaGlfYn1XXFxvdGltZXMgVikiXSxbMiwyLCIoVSBcXG90aW1lcyB7fV57XFxjaGlfYn1XKVxcb3RpbWVzIFYiXSxbNCwyLCIoe31ee1xcY2hpX2F9KF57XFxjaGlfe2J9fVcpIFxcb3RpbWVzIFUpXFxvdGltZXMgViJdLFs0LDAsIih7fV57XFxjaGlfe2FifX1XIFxcb3RpbWVzIFUpXFxvdGltZXMgViJdLFswLDEsIkNfe1VcXG90aW1lcyBWLCBXfSIsMl0sWzMsNCwiXFxhbHBoYV97VSx7fV57XFxjaGlfe2J9fVcsIFZ9Il0sWzIsMCwiXFxhbHBoYV97VSxWLFd9IiwyXSxbMiwzLCJcXG1hdGhybXtpZH1cXG90aW1lcyBDX3tWLFd9Il0sWzQsNSwiQ197VSx7fV57XFxjaGlfYn1XfVxcb3RpbWVzIFxcbWF0aHJte2lkfSJdLFsxLDYsIlxcYWxwaGFfe3t9XntcXGNoaV97YWJ9fVcsIFUsIFZ9IiwyXSxbNSw2LCIoVF57XFxjaGlfYSxcXGNoaV9ifV9XXFxvdGltZXNcXG1hdGhybXtpZH0pXFxvdGltZXNcXG1hdGhybXtpZH0iXV0=
\begin{tikzcd}[column sep = 2em, row sep = 1.5em]
	{(U \otimes V)\otimes W} && {{}^{\chi_{ab}}W \otimes (U\otimes V)} && {({}^{\chi_{ab}}W \otimes U)\otimes V} \\
	{U \otimes (V\otimes W)} \\
	{U \otimes ({}^{\chi_b}W\otimes V)} && {(U \otimes {}^{\chi_b}W)\otimes V} && {({}^{\chi_a}(^{\chi_{b}}W) \otimes U)\otimes V}
	\arrow["{C_{U\otimes V, W}}", from=1-1, to=1-3]
	\arrow["{\alpha_{{}^{\chi_{ab}}W, U, V}}", from=1-3, to=1-5]
	\arrow["{\alpha_{U,V,W}}"', from=2-1, to=1-1]
	\arrow["{\mathrm{id}\otimes C_{V,W}}", from=2-1, to=3-1]
	\arrow["{\alpha_{U,{}^{\chi_{b}}W, V}}", from=3-1, to=3-3]
	\arrow["{C_{U,{}^{\chi_b}W}\otimes \mathrm{id}}", from=3-3, to=3-5]
	\arrow["{(T^{\chi_a,\chi_b}_W\otimes\mathrm{id})\otimes\mathrm{id}}", from=3-5, to=1-5]
\end{tikzcd}
\end{equation}
\begin{equation}\tag{B3} \label{B3}
    % https://q.uiver.app/#q=WzAsNyxbMCwwLCJVIFxcb3RpbWVzIChWXFxvdGltZXMgVykiXSxbMCwxLCIoVSBcXG90aW1lcyBWKVxcb3RpbWVzIFciXSxbMCwyLCIoe31ee1xcY2hpX3thfX1WIFxcb3RpbWVzIFUpXFxvdGltZXMgVyJdLFsxLDAsInt9XntcXGNoaV97YX19KFYgXFxvdGltZXMgVykgXFxvdGltZXMgVSJdLFsyLDAsIih7fV57XFxjaGlfe2F9fVYgXFxvdGltZXMge31ee1xcY2hpX3thfX1XKSBcXG90aW1lcyBVIl0sWzIsMSwie31ee1xcY2hpX3thfX1WIFxcb3RpbWVzICh7fV57XFxjaGlfe2F9fVcgXFxvdGltZXMgVSkiXSxbMiwyLCJ7fV57XFxjaGlfe2F9fVYgXFxvdGltZXMgKFVcXG90aW1lcyBXKSJdLFswLDMsIkNfe1UsVlxcb3RpbWVzIFd9Il0sWzQsMywie1xcY2hpX1V9X3tWLFd9XFxvdGltZXMgXFxtYXRocm17aWR9IiwyXSxbMCwxLCJcXGFscGhhX3tVLFYsV30iXSxbMSwyLCJDX3tVLFZ9IFxcb3RpbWVzIFxcbWF0aHJte2lkfSJdLFs2LDIsIlxcYWxwaGFfe3t9XntcXGNoaV97YX19ViwgVSwgV30iLDJdLFs2LDUsIlxcbWF0aHJte2lkfVxcb3RpbWVzIENfe1UsV30iXSxbNSw0LCJcXGFscGhhX3t7fV57XFxjaGlfe2F9fVYsIHt9XntcXGNoaV97YX19VywgVX0iXV0=
\begin{tikzcd}[column sep = 4em, row sep = 1.5em]
	{U \otimes (V\otimes W)} & {{}^{\chi_{a}}(V \otimes W) \otimes U} & {({}^{\chi_{a}}V \otimes {}^{\chi_{a}}W) \otimes U} \\
	{(U \otimes V)\otimes W} && {{}^{\chi_{a}}V \otimes ({}^{\chi_{a}}W \otimes U)} \\
	{({}^{\chi_{a}}V \otimes U)\otimes W} && {{}^{\chi_{a}}V \otimes (U\otimes W)}
	\arrow["{C_{U,V\otimes W}}", from=1-1, to=1-2]
	\arrow["{\alpha_{U,V,W}}", from=1-1, to=2-1]
	\arrow["{({\chi_a})_{V,W}\otimes \mathrm{id}}"', from=1-3, to=1-2]
	\arrow["{C_{U,V} \otimes \mathrm{id}}", from=2-1, to=3-1]
	\arrow["{\alpha_{{}^{\chi_{a}}V, {}^{\chi_{a}}W, U}}", from=2-3, to=1-3]
	\arrow["{\mathrm{id}\otimes C_{U,W}}", from=3-3, to=2-3]
	\arrow["{\alpha^{-1}_{{}^{\chi_{a}}V, U, W}}", from=3-1, to=3-3]
\end{tikzcd}
\end{equation}
\begin{equation}\tag{B4} \label{B4}
% https://q.uiver.app/#q=WzAsNixbMCwwLCJeaFUgXFxvdGltZXMge31eaFYiXSxbMSwwLCJee1xcY2hpX3toIFxccmhkIGF9fShee2h9VikgXFxvdGltZXMge31eaFUiXSxbMCwxLCJeaChVIFxcb3RpbWVzIFYpIl0sWzEsMSwiXmgoe31ee1xcY2hpX3thfX1WIFxcb3RpbWVzIFUpIl0sWzMsMCwiXntcXGNoaV97aCBcXHJoZCBhIH1ofVYgXFxvdGltZXMge31eaFUiXSxbMywxLCJeaCh7fV57XFxjaGlfe2F9fVYpIFxcb3RpbWVzIHt9XmhVIl0sWzAsMSwiQ197e31eaFUse31eaFZ9Il0sWzAsMiwiaF97VSxWfSIsMl0sWzEsNCwiVF57XFxjaGlfe2ggXFxyaGQgYX0saH1fVlxcb3RpbWVzIFxcbWF0aHJte2lkfSJdLFsyLDMsIl5oQ197VSxWfSJdLFs1LDMsImhfe3t9XntcXGNoaV97YX19ViwgVX0iLDJdLFs1LDQsIlRee2gsIFxcY2hpX2F9X1YgXFxvdGltZXMgXFxtYXRocm17aWR9Il1d
\begin{tikzcd}
	{^hU \otimes {}^hV} & |[xshift=20mm]| {^{\chi_{h \rhd a}}(^{h}V) \otimes {}^hU} && {^{(\chi_{h \rhd a }h)}V \otimes {}^hU}
    \\
	{^h(U \otimes V)} & |[xshift=20mm]| {^h({}^{\chi_{a}}V \otimes U)} && {^h({}^{\chi_{a}}V) \otimes {}^hU}
	\arrow["{C_{{}^hU,{}^hV}}", from=1-1, to=1-2]
	\arrow["{h_{U,V}}", from=1-1, to=2-1]
	\arrow["{T^{\chi_{h \rhd a},h}_V \otimes \mathrm{id}}", from=1-2, to=1-4]
	\arrow["{^hC_{U,V}}", from=2-1, to=2-2]
	\arrow["{T^{h, \chi_a}_V \otimes \mathrm{id}}", from=2-4, to=1-4]
	\arrow["{h^{-1}_{{}^{\chi_{a}}V, U}}", from=2-2, to=2-4]
\end{tikzcd}
\end{equation}
\end{itemize}

Finally, an $H$-crossed $G$-graded monoidal category with a chosen $\chi$-crossed braiding is called \textit{$\chi$-crossed braided}.
\end{definition}

\begin{remark}
Let us highlight some aspects of the above definition.
    \begin{itemize}
        \item In order to have $\mathrm{Hom}_\mathcal{C}(U \otimes V, {^{\chi_a}V} \otimes U) \neq 0$, we need
        \begin{equation}
            ab = (\chi_a \rhd b)a \ ,
        \end{equation} 
which is the second identity in~\eqref{crossed module}.
        \item To make sense of~\eqref{B4}, the target objects of $T^{\chi_{h \rhd a},h}_V$ and $T^{h, \chi_a}_V$ must coincide, i.e.\ we must have
        \begin{equation}
            \chi_{h\rhd a} \,h \,=\, h \, \chi_a \ ,
        \end{equation}
which is equivalent to the first identity in~\eqref{crossed module}.

\item 
Under the assumption~\eqref{eq: strictness assumption}, $C_{\mathbbm 1, -}$ and $C_{-, \mathbbm 1}$ are given by compositions of the unitors, following arguments similar to those for braided categories in~\cite[Prop.\,13.1.2]{K95}.
    \end{itemize}
\end{remark}

\begin{definition}\label{crossed balanced}
    A $\chi$-crossed braided category $\mathcal{C}$ is called
    \begin{enumerate}
        \item \textit{$\chi$-crossed balanced} if there exists a family of $\chi$-crossed balancing isomorphisms
\begin{equation}\label{crossed balancing}
    \{\vartheta_U\colon U \to {^{\chi_{|U|}}U}\}_{U \in \mathcal{C}_{\mathrm{hom}}}
\end{equation}
satisfying the following conditions:
\begin{itemize}
    \item for any morphism $f\colon U \to V$ between homogeneous objects, we have
    \begin{equation}
        \vartheta_V \circ f = ({^{\chi_{|U|}}f}) \circ \theta_U \ ; \tag{BL1}\label{BL1}
    \end{equation}

    \item for any $a,b \in G$ and any $U,V \in \mathcal{C}$ of degrees $a$ and $b$, the following diagram commutes:
    \begin{equation}\tag{BL2}\label{BL2}
        % https://q.uiver.app/#q=WzAsNixbMCwwLCJVIFxcb3RpbWVzIFYiXSxbMCwxLCJee1xcY2hpX2F9VSBcXG90aW1lcyB7fV57XFxjaGlfYn1WIl0sWzAsMiwiXntcXGNoaV9hfShee1xcY2hpX3tifX1WKSBcXG90aW1lcyB7fV57XFxjaGlfYX1VIl0sWzIsMSwiXntcXGNoaV97YWJ9fVUgXFxvdGltZXMge31ee1xcY2hpX3thYn19ViJdLFsyLDAsIl57XFxjaGlfe2FifX0oVSBcXG90aW1lcyBWKSJdLFsyLDIsIl57XFxjaGlfe2FiYV57LTF9fX0oXntcXGNoaV97YX19VSkgXFxvdGltZXMge31ee1xcY2hpX2F9KF57XFxjaGlfYn1WKSJdLFswLDEsIlxcdGhldGFfVSBcXG90aW1lcyBcXHRoZXRhX1YiLDJdLFsxLDIsIkNfe157XFxjaGlfYX1VLHt9XntcXGNoaV9ifVZ9IiwyXSxbMCw0LCJcXHRoZXRhX3tVIFxcb3RpbWVzIFZ9Il0sWzMsNCwie1xcY2hpX3thYn19X3tVLFZ9IiwyXSxbMiw1LCJDX3tee1xcY2hpX2F9KF57XFxjaGlfe2J9fVYpLCB7fV57XFxjaGlfYX1VfSJdLFs1LDMsIlRfVV57XFxjaGlfe2FiYV57LTF9fSxcXGNoaV9hfSBcXG90aW1lcyBUX1Zee1xcY2hpX2EsIFxcY2hpX2J9IiwyXV0=
\begin{tikzcd}[column sep = 4em, row sep = 1.5em]
	{U \otimes V} && {^{\chi_{ab}}(U \otimes V)} \\
	{^{\chi_a}U \otimes {}^{\chi_b}V} && {^{\chi_{ab}}U \otimes {}^{\chi_{ab}}V} \\
	{^{\chi_a}(^{\chi_{b}}V) \otimes {}^{\chi_a}U} && {^{\chi_{aba^{-1}}}(^{\chi_{a}}U) \otimes {}^{\chi_a}(^{\chi_b}V)}
	\arrow["{\theta_{U \otimes V}}", from=1-1, to=1-3]
	\arrow["{\theta_U \otimes \theta_V}", from=1-1, to=2-1]
	\arrow["{C_{^{\chi_a}U,{}^{\chi_b}V}}", from=2-1, to=3-1]
	\arrow["{(\chi_{ab})_{U,V}}", from=2-3, to=1-3]
	\arrow["{C_{^{\chi_a}(^{\chi_{b}}V), {}^{\chi_a}U}}"', from=3-1, to=3-3]
	\arrow["{T_U^{\chi_{aba^{-1}},\chi_a} \otimes T_V^{\chi_a, \chi_b}}", from=3-3, to=2-3]
\end{tikzcd}
    \end{equation}

    \item for any $h \in H$, $a \in G$, and $U \in \mathcal{C}_{a}$, the following diagram commutes:
    \begin{equation}
        % https://q.uiver.app/#q=WzAsNCxbMCwwLCJeaFUiXSxbMCwxLCJeaChee1xcY2hpX2F9VSkiXSxbMiwwLCJee1xcY2hpX3toIFxccmhkIGF9fSheaFUpIl0sWzIsMSwiXntoXFxjaGlfe2F9fVUiXSxbMCwxLCJeaFxcdGhldGFfVSJdLFswLDIsIlxcdGhldGFfe15oVX0iXSxbMiwzLCJUX1Vee1xcY2hpX3toIFxccmhkIGF9LGh9IiwyXSxbMSwzLCJUX1Vee2gsXFxjaGlfYX0iXV0=
\begin{tikzcd}
	{^hU} && {^{\chi_{h \rhd a}}(^hU)} \\
	{^h(^{\chi_a}U)} && {^{h\chi_{a}}U}
	\arrow["{\theta_{^hU}}", from=1-1, to=1-3]
	\arrow["{^h\theta_U}", from=1-1, to=2-1]
	\arrow["{T_U^{\chi_{h \rhd a},h}}"', from=1-3, to=2-3]
	\arrow["{T_U^{h,\chi_a}}", from=2-1, to=2-3]
\end{tikzcd}\tag{BL3}\label{BL3}
    \end{equation}
    \end{itemize}

\item If $\mathcal{C}$ is moreover rigid, and for any $a \in G$, $U \in \mathcal{C}_{a}$, the following diagram commutes (recall $\nu$ from Remark~\ref{rem:canoncial-iso-duals}):
\begin{equation}
\begin{tikzcd}
	{}^{\chi_a}(U^*) 
    &&
    {^{\chi_{a}}(^{\chi_a^{-1}}(U^*))} 
    \\
	({}^{\chi_a}U)^* 
    &
    |[xshift=8mm]| U^* & 
    {^{1_H}(U^*)}
	\arrow["\nu_{\chi_a,U}", from=1-1, to=2-1]
	\arrow["{^{\chi_a}\theta_{U^*}}", from=1-1, to=1-3]
	\arrow["{(\theta_U)^*}", from=2-1, to=2-2]
	\arrow["{T^{\chi_a,\chi_a^{-1}}_{U^*}}"', from=1-3, to=2-3]
	\arrow["{\eta_{U^*}}", from=2-2, to=2-3]
\end{tikzcd}
\tag{RB1} \label{RB1}
\end{equation}
then $\mathcal{C}$ is called \textit{$\chi$-crossed ribbon}.
    \end{enumerate}
\end{definition}

\begin{remark}\label{rem:usual-ribbon-G-usual-Gcrossed}
    The neutral component $\mathcal{C}_{1_G}$ of a $\chi$-crossed braided (ribbon) category $\mathcal{C}$ is a usual braided (ribbon) category. If $\chi$ is $\mathrm{id}:G \to G$ with the conjugation action, our definition recovers the
    notion of $G$-crossed ribbon categories in~\cite{Tu00, GLM24b}.
\end{remark}

\begin{remark}
The notion of a $G$-crossed braided category can be regarded as a partial categorification of a crossed module. 
If we compare a crossed module 
$\big(\chi\colon K \to G,\, \gamma\colon G \to \mathrm{Aut}(K)\big)$ 
with a $G$-crossed braided category~$\mathcal{C}$,
then the group $K$ is
replaced by the monoidal category $\mathcal C$,
and $\Aut(K)$ is replaced by
$\Aut_\otimes(\mathcal C)$, while $G$ remains an ordinary group.
The $G$-grading on $\mathcal{C}$ categorifies $\chi$, 
and the action of $G$ on $\mathcal{C}$ categorifies $\gamma$. 
The condition
\begin{equation}
    {}^{g_1}\mathcal{C}_{g_2} \subset \mathcal{C}_{g_1 g_2 g_1^{-1}}
\end{equation}
categorifies the first crossed-module identity in~\eqref{crossed module}, 
and the relation
\begin{equation}
    M \otimes N \cong {^{|M|}N} \otimes M
\end{equation}
categorifies the Peiffer identity. 
Further details on $G$-crossed braided fusion categories for finite $G$ 
can be found, for example, in~\cite{ENOM09, DGNO10, EGNO15}.
\end{remark}

For $\mathbb{K}$-linear $H$-crossed $G$-graded monoidal categories $\mathcal{C}$ and $\mathcal{D}$, we introduced in Section~\ref{sec: fibre product} their fibre product $\mathcal{C}\boxtimes_G \mathcal{D}$. If $\mathcal{C}$ and $\mathcal{D}$ are further $\chi$-crossed ribbon, then $\mathcal{C}\boxtimes_G \mathcal{D}$ carries a canonical $\chi$-crossed ribbon structure:
\begin{itemize}
    \item The $\chi$-crossed braiding on pure tensors is given by
    \begin{equation}
        C_{U_1,U_2} \boxtimes C'_{V_1,V_2}\colon (U_1 \boxtimes V_1) \otimes (U_2 \boxtimes V_2) \longrightarrow {}^{\chi_{|U_1|}}(U_2 \boxtimes V_2) \otimes (U_1 \boxtimes V_1) \ ,
    \end{equation}
    where $C$ and $C'$ are $\chi$-crossed braidings in $\mathcal{C}$ and $\mathcal{D}$, respectively. Recall that by construction, $|U_1| = |V_1|$, and that the $H$-action on the product is as in \eqref{eq:H-action-on-Kelly-product}.

    \item The $\chi$-crossed balancing isomorphism on a pure tensor is given by
    \begin{equation}
        \theta_U \boxtimes \theta'_V\colon U \boxtimes V \longrightarrow {}^{\chi_{|U|}}(U \boxtimes V) \ ,
    \end{equation}
    where $\theta$ and $\theta'$ are $\chi$-crossed balancing isomorphisms in $\mathcal{C}$ and $\mathcal{D}$, respectively.
\end{itemize}

Let $\chi\colon G \to H$ be a crossed module, and let $\mathcal{C}$ be a $\chi$-crossed ribbon category. For another crossed module $\chi'\colon G' \to H'$, there are two possible ways to define a $\chi'$-crossed ribbon structure on $\mathcal{C}$:
\begin{enumerate}
    \item If there exist two group homomorphisms $\alpha\colon G \to G'$, $\beta:H' \to H$ such that 
    the following two diagrams commute:
\begin{equation}
\begin{tikzcd}[column sep = 4em, row sep = 1.5em]\label{crossed module morphism}
G \ar[r, "\chi"]\ar[d, "\alpha"]
&
H
&
H\times G \ar[d, "\rhd"]
&
H'\times G \ar[l, swap, "\beta \times \mathrm{id}_G"]\ar[r, "\mathrm{id}_{H'} \times \alpha"]
&
H' \times G' \ar[d, swap, "\rhd'"]\\
G' \ar[r, "\chi'"]
&
H' \ar[u, "\beta"]
&
G \ar[rr, "\alpha"]
&
{}
&
G'
\end{tikzcd}
\end{equation}
then $\mathcal{C}$ carries a canonical $\chi'$-crossed ribbon structure:
\begin{itemize}
    \item the $G'$-grading on $\mathcal{C}$ is given by
 \begin{equation}\label{G' grading}
            \mathcal{C} \; = \; \bigoplus_{g' \in G'}\Big(\bigoplus_{g \in \alpha^{-1}(g')}\mathcal{C}_g\Big) \ ;
    \end{equation}
    
    \item the $H'$-crossed action on $\mathcal{C}$ is given by the composition $\underline{H}' \xrightarrow{\underline{\beta}} \underline{H} \xrightarrow{\phi} \mathrm{Aut}_\otimes(\mathcal{C})$;
    \item thanks to~\eqref{crossed module morphism}, the isomorphisms in~\eqref{crossed braiding} and~\eqref{crossed balancing} defining the $\chi$-crossed ribbon structure on~$\mathcal{C}$ also yield a $\chi'$-crossed ribbon structure. For instance, the $\chi$-crossed braiding $\{C_{U,V}\colon U \otimes V \to {}^{\chi_{|U|}}V \otimes U\}$ is also a $\chi'$-crossed braiding, as the $\chi'$-crossed coherence conditions~\eqref{B1}--\eqref{B4} follow from
    the $\chi$-crossed ones by pushforward along $\alpha$ and pullback along $\beta$.
\end{itemize}

\item If there exists a crossed module morphism $(\alpha\colon G \to G', \, \beta\colon H \to H')$ with $\beta$ being surjective, such that the $H$-crossed action $\phi\colon \underline{H} \to \mathrm{Aut}_\otimes(\mathcal{C})$ factors through $\beta$, i.e.\ there exists a monoidal functor $\psi$
and a monoidal natural isomorphism $\delta$
making the following diagram commute:
\begin{equation}\label{eq: action factors}
    \begin{tikzcd}[column sep = 3em, row sep = 1.5em]
	{\underline{H}} & {\mathrm{Aut}_\otimes(\mathcal C)} \\
	{\underline{H}'}
	\arrow[""{name=0, anchor=center, inner sep=0}, "\varphi", from=1-1, to=1-2]
	\arrow["\beta"', from=1-1, to=2-1]
	\arrow["{\psi}"', dashed, from=2-1, to=1-2]
	\arrow[
	    Rightarrow,
	    shorten <= 0.4em,
	    shorten >= 0.8em,
	    "\delta",
	    from=0,
	    to=2-1
	]
\end{tikzcd}
\end{equation}
then $\mathcal C$ has a $\chi'$-crossed ribbon structure defined as follows:
\begin{itemize}
    \item the $G'$-grading is given by~\eqref{G' grading};

    \item the $H'$-crossed action is given by $\psi$;

    \item
    for objects $U\in\mathcal C_g$ and $V\in\mathcal C$,
    the $\chi'$-crossed braiding is given by
    \begin{equation}
        C'_{U,V}
        \coloneqq
        (\delta^{\chi_g}_V\otimes \id_U)\circ C_{U,V} \colon
        U\otimes V
        \longrightarrow {}^{\beta(\chi_g)}V\otimes U =
        {}^{\chi'_{\alpha(g)}}V\otimes U \ ,
    \end{equation}
    where $C_{U,V}$ is the $\chi$-crossed braiding on $\mathcal C$.
    Here we use~\eqref{usual crossed module morphism} that $\chi'_{\alpha(g)}=\beta(\chi_g)$. Similarly, the $\chi'$-crossed balancing is given by
    \begin{equation}
        \vartheta'_U
        \coloneqq
        \delta^{\chi_g}_U\circ \vartheta_U
        \colon
        U\longrightarrow {}^{\chi'_{\alpha(g)}}U \ ,
    \end{equation}
    where $\theta_U$ is the $\chi$-crossed balancing on $\mathcal C$. It is straightforward to check that all coherence conditions are satisfied.
\end{itemize}
\end{enumerate}

\medskip

Consider the following four special cases:
\begin{equation*}
\begin{tikzcd}
	G & H && G & H && G & H && G & G \\
	G & G && {\mathrm{Im}(\chi)} & {\mathrm{Im}(\chi)} && H & H && G & H
	\arrow["\chi", from=1-1, to=1-2]
	\arrow["{\mathrm{id}}", from=1-1, to=2-1]
	\arrow["\chi", from=1-4, to=1-5]
	\arrow["\chi", from=1-4, to=2-4]
	\arrow["\chi", from=1-7, to=1-8]
	\arrow["\chi", from=1-7, to=2-7]
	\arrow["{\mathrm{id}}", from=1-10, to=1-11]
	\arrow["{\mathrm{id}}", from=1-10, to=2-10]
	\arrow["\chi"', two heads, from=1-11, to=2-11]
	\arrow["{\mathrm{id}}", from=2-1, to=2-2]
	\arrow["\chi", from=2-2, to=1-2]
	\arrow["{\mathrm{id}}", from=2-4, to=2-5]
	\arrow[hook, from=2-5, to=1-5]
	\arrow["{\mathrm{id}}", from=2-7, to=2-8]
	\arrow["{\mathrm{id}}"', from=1-8, to=2-8]
	\arrow["\chi", from=2-10, two heads, to=2-11]
\end{tikzcd}
\end{equation*}
where  the first two satisfy~\eqref{crossed module morphism}, and the last two are obvious crossed module morphisms.
We obtain the following result:
\begin{proposition}\label{chi crrossed implies G crossed}
Let $\chi\colon G\to H$ be a crossed module.
\begin{enumerate}
\item Any $\chi$-crossed ribbon category $\mathcal C$ carries canonically
\begin{enumerate}
    \item a $G$-crossed ribbon structure,
    \item an $\mathrm{Im}(\chi)$-crossed ribbon structure, and
    \item an $H$-crossed ribbon structure.
\end{enumerate}

\item Conversely, if $\chi\colon G\to H$ is surjective, and
if $\mathcal{C}$ is a $G$-crossed ribbon category whose $G$-action factors through $\chi$, then $\mathcal{C}$ is also $\chi$-crossed ribbon.
\end{enumerate}
\end{proposition}

\begin{remark}
    If $\mathcal C$ is a $\chi$-crossed ribbon category with faithful $G$-grading, then the induced $\mathrm{Im}(\chi)$-crossed ribbon structure on $\mathcal C$ also has a faithful $\mathrm{Im}(\chi)$-grading.
\end{remark}

\begin{remark}\label{equivalence remark}
    The $\mathrm{Im}(\chi)$-crossed ribbon structure on $\mathcal{C}$ only makes use of partial information from $\chi$, which is the bottom row in~\eqref{crossed module extension}. At first sight, the proposition above may suggest that the notion of a $\chi$-crossed ribbon category does not provide anything beyond the usual notion of a $G$-crossed ribbon category. This is however not the case. The essential difference lies in the notion of the corresponding functors, which we will introduce next, and it leads to substantially different equivalence relations. Namely, equivalent $G$-crossed braidings on $\mathcal{C}$ may become inequivalent when regarded as $\chi$-crossed braidings. We will see this in examples in Section~\ref{different equiv relation}.
\end{remark}

\subsection{$\chi$-Crossed braided and ribbon functors}
Let $\chi\colon G \to H$ be a crossed module, and let $\mathcal{C}_1$ and $\mathcal{C}_2$ be $\chi$-crossed braided categories.

\begin{definition}\label{crossed braided functor}
A \textit{$\chi$-crossed braided functor} $\mathcal{C} \to \mathcal{C}'$ is an $H$-equivariant $G$-graded monoidal functor $(F,\omega)$
in the sense of Definition~\ref{crossed functor}
such that, for any $a, b \in G$ and any $U \in \mathcal{C}_{a}$, $V \in \mathcal{C}_{b}$, the following diagram commutes:
\begin{equation}\label{crossed braided functor diagram}
\begin{tikzcd}
	{F(U)\otimes_{\mathcal C'}F(V)} && {{}^{\chi_a}F(V)\otimes_{\mathcal C'}F(U)} && {F({}^{\chi_a}V)\otimes_{\mathcal C'}F(U)} \\
	{F(U\otimes_{\mathcal C}V)} &&&& {F({}^{\chi_a}V\otimes_{\mathcal C}U)}
	\arrow["{C'_{F(U),F(V)}}", from=1-1, to=1-3]
	\arrow["{F_{U,V}}", from=1-1, to=2-1]
	\arrow["{\omega^{\chi_a}_V\otimes_{\mathcal C'}\id}", from=1-3, to=1-5]
	\arrow["{F_{{}^{\chi_a}V,U}}"', from=1-5, to=2-5]
	\arrow["{F(C_{U,V})}", from=2-1, to=2-5]
\end{tikzcd}
\end{equation}
where $C$ and $C'$ are the $\chi$-crossed braidings in $\mathcal{C}$ and $\mathcal{C}'$, respectively.

If $\mathcal{C}$ and $\mathcal{C}'$ are further $\chi$-crossed balanced/ribbon, then $(F, \omega)$ is said to be \textit{$\chi$-crossed balanced/ribbon} if, for any $a \in G$ and any $U \in \mathcal{C}_{a}$, the following diagram commutes:
\begin{equation}
\begin{tikzcd}[column sep = 4em, row sep = 1.5em]
F(U) \ar[r, "F\theta_U"]\ar[d,swap, "\theta'_{F(U)}"]
&
F({^{\chi_a}U})\\
^{\chi_a}F(U) \ar[ru, swap, "\omega_U^{\chi_a}"]
\end{tikzcd}
\end{equation}
where $\theta$ and $\theta'$ are the $\chi$-crossed balancing isomorphisms in $\mathcal{C}$ and $\mathcal{C}'$, respectively.

If $F$ is an equivalence of categories, then $\mathcal{C}$ and $\mathcal{C}'$ are said to be \textit{$\chi$-crossed braided/balanced/ribbon equivalent}.
\end{definition}

Finally, a \textit{$\chi$-crossed braided/balanced/ribbon natural transformation} is just an $H$-equivariant monoidal natural transformation as given in Definition~\ref{chi-crossed natural transformation}.

\begin{remark}\label{rmk: crossed functor when chi is id}
If $\chi = \id \colon G \to G$ with conjugation action, our definitions specialise to the notions of $G$-crossed braided functor and $G$-crossed braided natural transformation given in~\cite[Def.\,4.3, Def.\,4.4]{JPR20}.
\end{remark}

Recalling Remark~\ref{rem: coherence} on the strictification of $H$-crossed $G$-graded monoidal categories, we have the following generalisation of~\cite[Thm.\,5.6\,\&\,Rem.\,5.7]{Ga16}
(the proof follows the same steps):

\begin{proposition}\label{prop:chi-crossed-br-strict}
Every $\chi$-crossed braided/ribbon category is equivalent to a strict $\chi$-crossed braided/ribbon category, in the sense that the underlying $H$-crossed $G$-graded monoidal category is strict.
\end{proposition}

\begin{corollary}\label{cor:pivotal}
    Every $\chi$-crossed ribbon category admits a natural pivotal structure $\delta \colon \mathrm{id} \Rightarrow (-)^{**}$.
For a homogeneous object $V$, $\delta_V$ is given by
\begin{align}\label{eq: pivotal isomorphism}
    \begin{split}
        & V\xrightarrow{\vartheta_V \otimes \mathrm{coev}_{V^*}}{}^{\chi_{|V|}}V \otimes V^* \otimes V^{**}\xrightarrow{C_{^{\chi_{|V|}}V,
        V^*} \otimes \mathrm{id}_{V^{**}}}{}
        ^{\chi_{|V|}}(V^*) \otimes {^{\chi_{|V|}}V} \otimes V^{**}
        \\
        &
        \xrightarrow{{\chi_{|V|}}_{V^*,V} \otimes \mathrm{id}_{V^{**}}}{}^{\chi_{|V|}}(V^* \otimes V)\otimes V^{**} \xrightarrow{^{\chi_{|V|}}\mathrm{ev}_V \otimes \mathrm{id}_{V^{**}}}
        {}^{\chi_{|V|}}(\mathbbm 1) \otimes V^{**} \xrightarrow{(\chi_{|V|})^{-1}_\mathbbm 1\otimes \mathrm{id}_{V^{**}}} V^{**}\ ,
    \end{split}
    \end{align}
where we omit associators and unitors for readability.
\end{corollary}

\begin{proof}
By Proposition~\ref{chi crrossed implies G crossed}, every $\chi$-crossed ribbon category $\mathcal C$ is in particular $G$-crossed ribbon. It is a known fact, stated without proof in~\cite[Lem.\,2.3]{Ki04}, that any $G$-crossed ribbon category admits a natural pivotal structure. Applying this to the underlying $G$-crossed ribbon structure of~$\mathcal C$, the pivotal structure is given precisely by~\eqref{eq: pivotal isomorphism}. To verify that $\delta$ is indeed a monoidal natural transformation, Proposition~\ref{prop:chi-crossed-br-strict} 
allows us to assume that all 
coherence morphisms
of the underlying $H$-crossed $G$-graded monoidal category
are identities. We can then compute directly via string diagrams
with over/under crossings. 
Since the graphical calculus for such open diagrams in $G$-crossed ribbon categories is identical to that of standard ribbon categories, the problem reduces to the standard ribbon case, which is detailed in~\cite[Ch.\,8.10]{EGNO15}.
\end{proof}

\subsection{$\chi$-Crossed braided structures on $\mathrm{vec}_G$}\label{VecG chi crossed classification}
Let $\chi\colon G \to H$ be a crossed module. 
We saw in 
Theorem~\ref{thm:VecG-Hcrossed-classification-H3tot} that a normalised 3-cocycle $\kappa = (\omega, \Omega, t) \in Z^3_{\mathrm{tot};1}(G\rtimes_\rhd H;\mathbb{K}^\times)$ gives rise to an $H$-crossed $G$-graded monoidal category $\mathrm{vec}_G^\kappa$. 
In this section we consider
\(\chi\)-crossed braidings on \(\mathrm{vec}_G^{\kappa}\).
This is closely related to~\cite{Na11}, where $H$-crossed braided structures on $\mathrm{vec}_G$ are considered (and a coarser notion of equivalence is used).

Recall our notation in~\eqref{eq: degree shifting notation}.
A $\chi$-crossed braiding on $\mathrm{vec}_G^{\kappa}$ for homogeneous objects is of the form
\begin{equation}
    C_{U_g,V_k}\colon U_g \otimes V_k \to {}^{\chi_g}V_k \otimes U_g\ , \qquad u_g \otimes v_k \mapsto R(g,k)\, v_k \otimes u_g
\end{equation}
for a normalised function
\begin{equation}
    R\colon G^2 \to \mathbb{K}^\times \ .
\end{equation}
Here, $u_g, v_k$ are vectors in the underlying vector spaces $U_g$ and $V_k$. The conditions~\eqref{B2}--\eqref{B4} correspond to, for all $g,k,l \in G, \ h\in H$,
\begin{subequations}\label{eq:qa-cocycle-system}
\begin{align}
& \omega(g,k,l)\,R(gk,l)\,\omega(gklk^{-1}g^{-1},g,k)
= R(k,l)\,\omega(g,klk^{-1},k)\,R(g,klk^{-1})\,t(\chi_g,\chi_k,l)\ ,
\label{eq:qa-cocycle-system-a}\\
& R(g,kl)\,\omega(gkg^{-1},g,l)
= \omega(g,k,l)\,R(g,k)\,R(g,l)\,\omega(gkg^{-1},glg^{-1},g)\,\Omega(\chi_g,k,l)\ ,
\label{eq:qa-cocycle-system-b}\\
& \Omega(h,gkg^{-1},g)\,R(h\rhd g, h\rhd k)\,t(\chi_{h\rhd g},h,k)
= \Omega(h,g,k)\,R(g,k)\,t(h,\chi_g,k) \ .
\label{eq:qa-cocycle-system-c}
\end{align}
\end{subequations}
We denote the resulting $\chi$-crossed braided category by $\mathrm{vec}_G^{(\kappa;R)}$.

Specialising Definition~\ref{crossed braided functor} to this case, $(\kappa; R)$ and $(\widetilde{\kappa}; \widetilde{R})$ define equivalent $\chi$-crossed braided structures on $\mathrm{vec}_G$ if and only if there exist normalised functions
\begin{equation}
    \alpha\colon G^2 \to \mathbb{K}^\times \ , \qquad \beta\colon H \times G \to \mathbb{K}^\times
\end{equation}
such that, for all $g,k,l \in G, \ h, h_1, h_2\in H$,
\begin{itemize}
    \item identities~\eqref{eq:vecG-functor-monoidal-cond} and~\eqref{coboundary condition}
    hold, establishing that $\kappa, \widetilde\kappa \in Z^3_{\mathrm{tot};1}(G\rtimes_\rhd H;\mathbb{K}^\times)$ are cohomologous by Theorem~\ref{thm:VecG-Hcrossed-classification-H3tot}, and
    \item condition \eqref{crossed braided functor diagram} holds, which in the present case reads
\begin{equation}\label{braiding coboundary condition}
        \alpha(gkg^{-1},g) \, \widetilde{R}(g,k) \, \beta(\chi_g,k) = R(g,k) \, \alpha(g,k) \ .
\end{equation}
\end{itemize}

We denote the group of equivalence classes of $(\kappa; R) = (\omega, \Omega, t; R)$ by
\begin{equation}
    H^3_{\mathrm{cm}}(\chi; \mathbb{K}^\times) \ .
\end{equation}

\begin{proposition}\label{prop:classification chi-crossed on vec_G}
     $\chi$-crossed braided structures on $\mathrm{vec}_G$ 
    are classified 
    up to $\chi$-crossed braided equivalence
    by $H^3_{\mathrm{cm}}(\chi; \mathbb{K}^\times)$.
\end{proposition}

\begin{remark} Note that
    $\mathrm{vec}_G^{(\kappa_1;R_1)} \boxtimes_G \mathrm{vec}_G^{(\kappa_2;R_2)} \cong \mathrm{vec}_G^{(\kappa_1\kappa_2;R_1R_2)}$ as $\chi$-crossed braided categories. 
In fact, one can use the fibre product to twist $\chi$-crossed braided categories by $(\kappa;R)$ analogously to \eqref{eq:twisting-as-fibred-product}, resulting in an action of $H^3_{\mathrm{cm}}(\chi; \mathbb{K}^\times)$ on equivalence classes of $\chi$-crossed braided categories.
\end{remark}

\begin{remark}
In~\cite[Thm.\,4.4]{Na11}, the author classifies $H$-crossed braided monoidal structures on $\mathrm{vec}_G$ by
$\Aut(\chi)$-orbits in $H^3_{\mathrm{cm}}(\chi;\mathbb{K}^\times)$. This
classification is taken with respect to an equivalence relation different from
ours (see~\cite[Def.\,2.6]{Na11}).
In particular, the functors considered there are not required to be degree-preserving, and the $H$-equivariance data may be twisted by a group automorphism of $H$: the isomorphisms $\omega^h$ in~\eqref{eq: omega^h}
are replaced by isomorphisms of the form ${}^{f(h)}F(-) \Rightarrow F({}^h(-))$ for some $f\in \mathrm{Aut}(H)$.
\end{remark}

If $\chi = \mathrm{id}\colon G \to G$ with the conjugation action, then following~\cite[Lem.\,6.3]{Na11}, there is a group homomorphism
\begin{equation}\label{eq: recover G-crossed case}
    f\colon H^3(G; \mathbb{K}^\times) \longrightarrow H^3_{\mathrm{cm}}(\chi; \mathbb{K}^\times) \ , \qquad [\omega] \mapsto [\omega, \Omega_\omega, t_\omega; 1] \ ,
\end{equation}
where $\Omega_\omega, \, t_\omega\colon G^3 \to \mathbb{K}^\times$ are given by, for all $g,k,l \in G$,
\begin{equation}
    \Omega_\omega(g,k,l) = \frac{\omega(gkg^{-1},g,l)}{\omega(gkg^{-1},glg^{-1},g)\,\omega(g,k,l)} 
    ~,~~
    t_\omega(g,k,l) = \frac{\omega(g,k,l)\,\omega(gklk^{-1}g^{-1},g,k)}{\omega(g,klk^{-1},k)} \ .
\end{equation}
One can check that $f$ is well-defined, i.e.\ independent of the choice of representative of $[\omega]$.

\begin{proposition}\label{H_qa for id chi}
For $\chi = \mathrm{id}_G$ with conjugation action, we have $H^3_{\mathrm{cm}}(\chi; \mathbb{K}^\times) \cong H^3(G; \mathbb{K}^\times)$ via $f$ in~\eqref{eq: recover G-crossed case}.
\end{proposition}
\begin{proof}
It is clear from \eqref{eq:vecG-functor-monoidal-cond} and \eqref{d2tot-via-alpha-beta} that $f$ is injective. For surjectivity, any valid tuple $(\omega, \Omega, t; R)$ is equivalent to $(\omega, \Omega_\omega, t_\omega; 1)$ via $(\alpha, \beta) = (1, R)$.
\end{proof}
\begin{remark}
As noted in Remarks~\ref{rem:usual-ribbon-G-usual-Gcrossed}
and~\ref{rmk: crossed functor when chi is id},
when $\chi = \mathrm{id}:G \to G$, a $\chi$-crossed braided structure is precisely a $G$-crossed braided structure, and a $\chi$-crossed braided functor specialises to a $G$-crossed braided functor.
The proposition above corresponds well to the well-known fact that $G$-crossed braided structures
on $\mathrm{vec}_G$ are classified up to $G$-crossed braided equivalence by $H^3(G;\mathbb{K}^\times)$~\cite[Thm.\,7.12]{ENOM09}.
\end{remark}

It is remarkable that a $G$-crossed braided structure on
$\mathrm{vec}_G$ is (up to the corresponding equivalence relation) fully determined by its monoidal structure. In contrast, for an abelian group~$K$, braided monoidal structures on $\mathrm{vec}_K$ are classified by the \textit{abelian third cohomology group} $H^3_{\mathrm{ab}}(K;\mathbb{K}^\times)$; see~\cite[Ch.\,8.4]{EGNO15}.

\begin{proposition}\label{compare to ab cohomology}
    For $\chi \colon K \to \{*\}$ we have $H^3_\mathrm{ab}(K; \mathbb{K}^\times) \cong H^3_{\mathrm{cm}}(\chi; \mathbb{K}^\times)$.
\end{proposition}

\begin{proof}
    We briefly recall abelian cohomology; see e.g.~\cite{EM53,EM54} and~\cite[Ch.\,8.4]{EGNO15}. Denote by $C_\mathrm{ab}^3(K;\mathbb{K}^\times)$ the collection of pairs $(\omega, R)$, where
\begin{equation}
    \omega\colon K^3 \longrightarrow \mathbb{K}^\times \ , \qquad R\colon K^2 \longrightarrow \mathbb{K}^\times
\end{equation}
are functions. $C_\mathrm{ab}^3(K;\mathbb{K}^\times)$ is an abelian group under pointwise multiplication. Denote by $Z_\mathrm{ab}^3(K;\mathbb{K}^\times)$ the subgroup of $C_\mathrm{ab}^3(K;\mathbb{K}^\times)$ consisting of pairs $(\omega,R)$ satisfying the following conditions:
\begin{itemize}
    \item $\omega$ is a normalised 3-cocycle,

    \item $R$ satisfies, for all $g,k,l \in K$,
    \begin{subequations}\label{eq:qa-cocycle-system'}
        \begin{align}
            &R(0,g) = R(g,0) = 1 \ , \label{eq:qa-cocycle-system'-a}\\
            &\omega(g,k,l)\,R(g+k,l)\,\omega(l,g,k)
= R(k,l)\,\omega(g,l,k)\,R(g,l)\ ,
\label{eq:qa-cocycle-system'-b}\\
&R(g,k+l)\,\omega(k,g,l)
= \omega(g,k,l)\,R(g,k)\,R(g,l)\,\omega(k,l,g)\ .
\label{eq:qa-cocycle-system'-c}
        \end{align}
    \end{subequations}
\end{itemize}
Two such pairs $(\omega,R), (\omega',R') \in Z_\mathrm{ab}^3(K;\mathbb{K}^\times)$ are equivalent if there exists an $\alpha\colon K^2 \to \mathbb{K}^\times$ such that, for all $g,k,l \in G$,
\begin{subequations}\label{eq:ab coh bound}
    \begin{align}
        &\alpha(0,g) = \alpha(g,0) = 1 \ ,\\
        &\frac{\alpha(k,l)}{\alpha(g+k,l)}\,\frac{\alpha(g,k+l)}{\alpha(g,k)} = \frac{\omega'(g,k,l)}{\omega(g,k,l)} \ , \\
        &\frac{\alpha(g,k)}{\alpha(k,g)}
        = \frac{R'(g,k)}{R(g,k)} \ .
    \end{align}
\end{subequations}
$H^3_{\mathrm{ab}}(K;\mathbb{K}^\times)$ is defined by
\begin{equation}
    H^3_{\mathrm{ab}}(K;\mathbb{K}^\times)
\;\coloneqq\;
Z^3_{\mathrm{ab}}(K;\mathbb{K}^\times)\big/ \sim \ .
\end{equation}

\medskip

We claim that the map
\begin{equation}
    f\colon H^3_\mathrm{ab}(K; \mathbb{K}^\times) \longrightarrow H^3_{\mathrm{cm}}(\chi; \mathbb{K}^\times) \ , \qquad [\omega, R] \mapsto [\omega, 1, 1; R]
\end{equation}
is a group isomorphism. Indeed, for such $\chi$, both $\Omega$ and $t$ must be 1, hence~\eqref{eq:qa-cocycle-system} reduces to~\eqref{eq:qa-cocycle-system'}. For the coboundary condition, normalisation of $\beta$ forces $\beta \equiv 1$, thus~\eqref{eq:vecG-functor-monoidal-cond} and~\eqref{braiding coboundary condition} reduce to~\eqref{eq:ab coh bound}.
\end{proof}

\subsection{Example: $G = \mathbb{Z}/4\mathbb{Z}$}\label{different equiv relation}

We continue with our examples in Section~\ref{Z/4 example}. For each $H_j$-action on $G$, we associate a map $\chi_j\colon G \to H_j$, giving rise to the following four crossed modules:
\begin{enumerate}
    \item $\chi_1\colon\Z/4\Z \to \{*\}$ with trivial action;
    \item $\chi_2\colon\Z/4\Z \to \Z/2\Z$ the quotient map with trivial action;
    \item $\chi_3\colon\Z/4\Z \to \Z/4\Z$ the identity map with trivial action;
    \item $\chi_4\colon\Z/4\Z \to \Z/4\Z$ defined by $\chi_4(1)=2$, with action fixed by $1 \rhd k = -k$.
\end{enumerate}
As in \eqref{eq:example-i}, fix an algebraically closed field $\mathbb{K}$ with $\mathrm{char}(\mathbb{K}) \neq 2$ and a primitive 4th root of unity $i \in \mathbb{K}^\times$. We compute $H^3_{\mathrm{cm}}(\chi_j; \mathbb{K}^\times)$ for $j = 1,2,3,4$ in the following.

\medskip

\noindent
$\bullet$ \textit{Case (\romannumeral 1):} By Proposition~\ref{compare to ab cohomology}, $H^3_{\mathrm{cm}}(\chi_1; \mathbb{K}^\times) \cong H^3_{\mathrm{ab}}(\Z/4\Z; \mathbb{K}^\times) \cong \Z/8\Z$. 
Fix an 8th root of unity $\zeta \in \mathbb{K}^\times$ such that
\begin{equation}
    \zeta^2 = i \ .
\end{equation}
Define $\omega\colon(\Z/4\Z)^3 \to \mathbb{K}^\times$ as in~\eqref{generator Z/4Z-cocycle}, 
and
\begin{equation}\label{eq:R_zeta}
    R_\zeta\colon(\Z/4\Z)^2 \longrightarrow \mathbb{K}^\times \ , \qquad R_\zeta(a,b) \;\coloneqq\; \zeta^{ab} \ .
\end{equation}
Following~\cite[Ex.\,2.1.3]{Ga25}, a generator of $H^3_{\mathrm{cm}}(\chi_1; \mathbb{K}^\times)$ is given by $[\omega^2, 1 , 1; R_\zeta]$.

\medskip

\noindent
$\bullet$ \textit{Case (\romannumeral 2):} $H^3_{\mathrm{cm}}(\chi_2; \mathbb{K}^\times) \cong \Z/8\Z$. Here we just give the generator and omit the proof.
Fix a primitive 16th-root of unity $\mu \in \mathbb{K}^\times$ such that
\begin{equation}
     i = \mu^4 \ .
\end{equation}
Define $\omega\colon(\Z/4\Z)^3 \to \mathbb{K}^\times$ as in~\eqref{generator Z/4Z-cocycle}, 
$t \colon(\Z/2\Z)^2 \times \Z/4\Z \to \mathbb{K}^\times$ as in~\eqref{eq:ti}, and 
\begin{equation}
    R_\mu\colon(\Z/4\Z)^2 \longrightarrow \mathbb{K}^\times,
\qquad R_\mu(0,a) = 1 \ , \; R_\mu(1,a) = \mu^{-a} \ , \; R_\mu(2,a) = \mu^{2a} \ , \; R_\mu(3,a) = \mu^{a} \ .
\end{equation}
$[\omega, 1, t; R_\mu]$ generates $H^3_{\mathrm{cm}}(\chi_2; \mathbb{K}^\times)$ and it has order 8.

\medskip

\noindent
$\bullet$ \textit{Case (\romannumeral 3):} Following Proposition~\ref{H_qa for id chi}, $H^3_{\mathrm{cm}}(\chi_3; \mathbb{K}^\times) \cong H^3(\Z/4\Z; \mathbb{K}^\times) \cong \Z/4\Z$.

\medskip

\noindent
$\bullet$ \textit{Case (\romannumeral 4):} We compute all possible $[\kappa;R] \in H^3_{\mathrm{cm}}(\chi_4; \mathbb{K}^\times)$ by hand. 
From case (iv) in Section~\ref{Z/4 example} we know that
$\kappa$ must be of the form
$(\omega^a, \Omega^a, t_\pm)$ for some $a \in \{0,1,2,3\}$, where $\omega, \Omega, t_-$ are given as in~\eqref{generator Z/4Z-cocycle},~\eqref{eq: Omega_i},~\eqref{eq:t_-}, and $t_+ \equiv 1$.
 Then,~\eqref{eq:qa-cocycle-system} simplifies to, for all $h \in \Z/4\Z$, $g,k,l \in \Z/4\Z$,
\begin{subequations}\label{simp-eq:qa-cocycle}
    \begin{align}
        &\omega^a(g,k,l)\,R(g+k,l) = R(k,l)\,R(g,l)\,t_\pm(\chi_g,\chi_k,l) \ ;\label{simp-eq:qa-cocycle-a}\\
        &R(g,k+l)=R(g,k)\,R(g,l)\,\omega^a(k,l,g) \ ;\label{simp-eq:qa-cocycle-b} \\
        &\Omega^a(h,k,g)\,R(h \rhd g, h \rhd k) = \Omega^a(h,g,k) \, R(g,k) \label{simp-eq:qa-cocycle-c}\ .
    \end{align}
\end{subequations}
Evaluating $(g,k,l) = (1,2,2),\,(1,1,2)$ at~\eqref{simp-eq:qa-cocycle-a} gives:
\begin{equation}
    R(3,2) = R(1,2)\,R(2,2) \ , \qquad R(2,2) = R(1,2)^2 \ .
\end{equation}
Evaluating $(g,k,l) = (1,2,2)$ at~\eqref{simp-eq:qa-cocycle-b} gives:
\begin{equation}
     R(1,2)^2 = i^{-a} \ .
\end{equation}
Evaluating $(h,g,k) = (1,1,2)$ at~\eqref{simp-eq:qa-cocycle-c} gives:
\begin{equation}
     R(3,2) = i^{a}R(1,2) \ .
\end{equation}
These imply
\begin{equation}
    i^{2a} = 1 \ ,
\end{equation}
i.e. $a$ must be even.

We now assume $a \equiv 0 \,\mathrm{mod} \,2$. Evaluating $(g,k,l) = (1,1,3)$ at~\eqref{simp-eq:qa-cocycle-b} gives:
\begin{equation}
     1 = R(1,1)\,R(1,3)i^a \ .
\end{equation}
Evaluating $(h,g,k) = (1,1,3)$ at~\eqref{simp-eq:qa-cocycle-c} gives:
\begin{equation}
     R(3,1) = R(1,3) \ .
\end{equation}
Evaluating $(g,k,l) = (1,3,1)$ at~\eqref{simp-eq:qa-cocycle-a} for $t_-$ gives:
\begin{equation}
     i^a = -R(3,1)\,R(1,1) \ .
\end{equation}
These imply $i^{2a} = -1$, which contradicts to the assumption that $a$ is even. Therefore, for every~$a$, the triple $(\omega^a, \Omega^a, t_-)$ does not admit any $R\colon(\Z/4\Z)^2 \to \mathbb{K}^\times$ satisfying~\eqref{eq:qa-cocycle-system}.

Now the only possible $\kappa$ are $(\omega^a, \Omega^a, 1)$ for $a = 0,2$. It is not hard to check that for any choice of $a$, an admissible $R$ is determined by its value $R(1,1)$ which should satisfy
\begin{equation}
    R(1,1)^4 = i^a \ .
\end{equation}
Define $R_\zeta$ as in~\eqref{eq:R_zeta}. One can further check that $[\omega^2, \Omega^2, 1; R_\zeta]$ is a generator for $H^3_{\mathrm{cm}}(\chi_4; \mathbb{K}^\times)$, and it has order 8, hence We have
\begin{equation}
    H^3_{\mathrm{cm}}(\chi_4; \mathbb{K}^\times) \;\cong\; \Z/8\Z \ .
\end{equation}

\begin{remark}\label{rmk: different equiv for chi}
We now clarify Remark~\ref{equivalence remark}. There are four bicharacters of
$\mathbb Z/4\mathbb Z$, namely
\begin{equation}
    R_k\colon (\mathbb Z/4\mathbb Z)^2 \longrightarrow \mathbb K^\times \ ,
    \qquad
    R_k(a,b)= i^{kab} \ ,
\end{equation}
for $k\in\{0,1,2,3\}$. For each $j=1,2,3,4$ and each $k$, the tuple
$(1,1,1;R_k)$ defines a class in
\begin{equation}
    H^3_{\mathrm{cm}}(\chi_j;\mathbb K^\times) \ .
\end{equation}
However, for a fixed $\chi_j$, the four bicharacters $R_k$ need not determine
four distinct cohomology classes. Our computations show that they determine
four distinct classes for $\chi_1$, two distinct classes for $\chi_2$, one
class for $\chi_3$, and four distinct classes for $\chi_4$.

For the first three cases $j=1,2,3$, the representatives $(1,1,1;R_k)$
encode trivial $H_j$-actions on $\mathrm{vec}_G$ with untwisted monoidal structure. Thus the $H_3$-crossed action factors through both $\chi_1$ and $\chi_2$ in the sense of~\eqref{eq: action factors}. Hence the underlying $G$-crossed braided structures, equivalently
the $\chi_3$-crossed braided structures, represented by $(1,1,1;R_k)$ may
also be regarded as $\chi_1$- or $\chi_2$-crossed braided structures. However, as we have seen, the equivalence relation depends on the crossed module $\chi_j$. A larger acting group $H_j$ gives a coarser equivalence relation, while a smaller acting group gives a finer one. This can be seen from~\eqref{crossed braided functor diagram}, where a larger acting group gives more freedom to relate two crossed braidings via the term $\omega^{\chi_a}_V\otimes_{\mathcal C'}\id$.
\end{remark}

\subsection{Equivariantisation}\label{sec:equivar}
We recall the notion of equivariantisation, 
see e.g.~\cite{DGNO10, GLM24b, HPRW25}.
Here we treat the case of $\chi$-crossed ribbon categories
from Definition~\ref{crossed balanced},
and explain the induced grading on the equivariantised categories.

Let $H$ act on $G$ via group automorphisms and let $\mathcal{C} = \bigoplus_{g \in G}\mathcal{C}_g$ be an $H$-crossed $G$-graded monoidal category.
Here we assume that the $H$-action is strictly unital as in Remark~\ref{free choice up to monoidal natural equiv}.
We consider the \textit{$H$-equivariantisation} of $\mathcal{C}$, denoted by $\mathcal{C}^H$:
\begin{itemize}
    \item objects in $\mathcal{C}^H$ are pairs $\big(U, \{\gamma_U^h\colon{}^hU \xrightarrow{\sim} U\}_{h \in H}\big)$, where $U \in \mathcal{C}$, and $\{\gamma^h_U\}$ is a family of isomorphisms making the following diagram commutes for any $h_1, h_2 \in H$:
\begin{equation}\label{eq:equivar condition}
\begin{tikzcd}
	{^{h_1}(^{h_2}U)} && {^{h_1h_2}U} \\
	{^{h_1}U} && U
	\arrow["{T_U^{h_1,h_2}}", from=1-1, to=1-3]
	\arrow["{^{h_1}\gamma^{h_2}_U}", from=1-1, to=2-1]
	\arrow["{\gamma^{h_1h_2}_U}"', from=1-3, to=2-3]
	\arrow["{\gamma^{h_1}_U}", from=2-1, to=2-3]
\end{tikzcd}
\end{equation}

\item a morphism $f\colon(U, \gamma_U) \to (V, \gamma_V)$ is a morphism $f\colon U \to V$ in $\mathcal{C}$ such that the following diagram commutes for any $h \in H$:
\begin{equation}
\begin{tikzcd}[column sep = 4em, row sep = 1.5em]
	{^{h}U} && U \\
	{^hV} && V
	\arrow["{\gamma_U^h}", from=1-1, to=1-3]
	\arrow["{^hf}", from=1-1, to=2-1]
	\arrow["f"', from=1-3, to=2-3]
	\arrow["{\gamma_V^h}", from=2-1, to=2-3]
\end{tikzcd}
\end{equation}
\end{itemize}

Let $p\colon H \times G \to G$ denote the map $(h,g) \mapsto g$, and let $\rho\colon H \times G \to G$ denote the $H$-action on $G$. Define the coinvariant quotient $G_H$ to be the coequaliser:
\begin{equation}\label{eq: coninvarinat group}
\begin{tikzcd}[column sep = 4em]
H \times G \ar[r,shift left=.75ex,"\rho"]
  \ar[r,shift right=.75ex,swap,"p"]
&
G \ar[r,"\pi"]
&
G_H \ .
\end{tikzcd}
\end{equation}
Explicitly we have that $G_H = G/N$, where $N \subset G$ 
is the normal subgroup generated by the elements $\{ (h \rhd g)g^{-1}\}_{(h,g)\in H\times G}$. 

\begin{remark}
If $G$ acts on itself by conjugation, then $G_G \cong G^{\mathrm{ab}}$ is the abelianisation of $G$. If $\chi \colon G\to H$ is a crossed module, then
\begin{equation}
    G_H \cong (G_{\mathrm{Im}(\chi)})_{H/\mathrm{Im}(\chi)} \cong (G^{\mathrm{ab}})_{H/\mathrm{Im}(\chi)} \ ,
\end{equation}
because $\mathrm{Im}(\chi)\subset H$ is normal
(recall Lemma~\ref{lem: K central I normal}), and acts on $G$
via conjugation by~\eqref{crossed module}. In particular, if $\chi$ is surjective, then $G_H = G^{\mathrm{ab}}$; if $\chi$ is an embedding of a normal subgroup $G \subset H$, then $G_H = G/[G,H]$.
\end{remark}

Let $(U, \gamma_U)$ be an
object in $\mathcal C^H$. 
Since $\mathcal{C}=\bigoplus_{g\in G}\mathcal{C}_g$ is $G$-graded additive,
we can write
\begin{equation}
    U = \bigoplus_{[\alpha]\in G_H} \Big(\bigoplus_{g \in \pi^{-1}[\alpha]}U_g\Big)
\qquad\text{with } U_g\in \mathcal{C}_g \ ,
\end{equation}
where $\pi$ is as in~\eqref{eq: coninvarinat group}.
For each $[\alpha] \in G_H$, denote $U_{[\alpha]} \coloneqq \bigoplus_{g \in \pi^{-1}[\alpha]}U_g$. One can see that $\bigl(U_{[\alpha]},\{\gamma^h_{|U_{[\alpha]}}\}_{h \in H}\bigr)$ is also an object of $\mathcal C^H$,
where $\gamma^h_{|U_{[\alpha]}}$ denotes the restriction of $\gamma^h_U$ on $^hU_{[\alpha]}$.
This relies on the fact that $\pi(h \rhd g) = \pi(g)$, ensuring that
the $H$-action on $\mathcal C$
does not mix components from different $G_H$-classes.
Thus $\mathcal{C}^H$ admits a $G_H$-grading.

Moreover, the (rigid) monoidal structure on $\mathcal{C}$ descends to $\mathcal{C}^H$ with
\begin{itemize}
    \item $(U, \gamma_U)\otimes (V, \gamma_V) \coloneqq (U \otimes V, \gamma_{U \otimes V})$, where $\gamma_{U \otimes V}$ is defined by
\begin{equation}
    ^h(U \otimes V) \xrightarrow{h^{-1}_{U,V}} {}^hU \otimes {}^hV \xrightarrow{\gamma_U^h \otimes \gamma_V^h} U \otimes V;
\end{equation}
\item $(U, \gamma_U)^* \coloneqq (U^*, \gamma_{U^*})$, where $ \gamma_{U^*}$ is defined by (recall Remark~\ref{rem:canoncial-iso-duals})
\begin{equation}
    ^h(U^*) \xrightarrow{\nu_{h,U}} ({}^hU)^* 
    \xrightarrow{((\gamma_U^h)^*)^{-1}}U^* \ .
\end{equation}
\end{itemize}
Since $\pi$ is a group homomorphism, the $G_H$-grading is then compatible with the monoidal structure on $\mathcal{C}^H$. We conclude that $\mathcal{C}^H$ carries a $G_H$-graded monoidal structure.

\begin{proposition}\label{prop:equiv}
Let $\chi\colon G \to H$ be a crossed module, and $\mathcal{C}$ be a $\chi$-crossed ribbon category. Then the $H$-equivariantisation $\mathcal{C}^H$ is
ribbon with braiding
and balancing isomorphisms
given by
\begin{subequations}
    \begin{gather}
        U \otimes V = \bigoplus_{g \in G}U_g \otimes V \xrightarrow{\bigoplus_{g \in G}
        C_{U_g,V}} \bigoplus_{g \in G}{}^{\chi_g}V \otimes U_g \xrightarrow{\bigoplus_{g \in G}\gamma_V^{\chi_g}\otimes \mathrm{id}_{U_g}}\bigoplus_{g \in G}V \otimes U_g = V \otimes U \ , \\
        U = \bigoplus_{g \in G} U_g \xrightarrow{\bigoplus_{g \in G}\theta_{U_g}}\bigoplus_{g \in G} {}^{\chi_g}U_g\xrightarrow{\bigoplus_{g \in G}\gamma^{\chi_g}_{|U_g}}\bigoplus_{g \in G}U_g = U \ ,
\end{gather}
\end{subequations}
where
$U = \bigoplus_{g \in G} U_g$ is a decomposition of $U$ into its homogeneous components,
$C_{U_g, V}$ is the $\chi$-crossed braiding in $\mathcal{C}$, and $\theta_{U_g}$ is the $\chi$-crossed balancing isomorphism in $\mathcal{C}$.
\end{proposition}

\begin{proof}
    The proof follows the arguments for the $G$-equivariantisation of a $G$-crossed ribbon category; see, for example,~\cite[Ch.\,8.24]{EGNO15} and~\cite[Sec.\,3.4]{GLM24b}.
\end{proof}

\begin{example}
Let $[\kappa; R] = [1,1,1;1]$ be the trivial class in $H^3_\mathrm{cm}(\chi; \mathbb{K}^\times)$. Unpacking the data of the $H$-equivariantisation $(\mathrm{vec}_G^{\kappa;R})^H$,
an object is a $G$-graded vector space $V = \bigoplus_{g \in G}V_g$ together with a family of isomorphisms $\{\gamma^h_V \colon {}^hV \to V\}_{h \in H}$, where $(^hV)_g \coloneqq V_{h^{-1}\rhd g}$ following our convention in~\eqref{eq: degree shifting notation}. Under this identification, one can see $V$ as a $G$-graded vector space equipped with an $H$-action. It is not hard to see that $(\mathrm{vec}_G^{\kappa;R})^H$ coincides precisely with the category $\mathcal{M}(\chi)$ of finite-dimensional representations of the crossed module $\chi$, as introduced 
in~\cite[Sec.\,3]{Ba05}, see also~\cite[Def.\,1.2.3]{Au25}.
It is further shown in~\cite{Au25} that $\mathcal{M}(\chi)$ is equivalent to the category of finite-dimensional representations of a quasitriangular Hopf algebra $D(G,H)$, whose underlying algebra is the smash product $\mathbb{K}^G \rtimes \mathbb{K}[H]$.
Another way to understand $(\mathrm{vec}_G^{\kappa;R})^H$ is via the functor category 
\begin{equation}
    \mathrm{Fun}(G \sslash H, \mathrm{vec}) \ ,
\end{equation}
whose objects are functors from the action groupoid $G \sslash H$ to $\mathrm{vec}$ and whose morphisms are natural transformations~\cite[Prop.\,1.2.4]{Au25}. Furthermore, the functor category $\mathrm{Fun}(G \sslash H, \mathrm{vec})$ naturally admits a braided monoidal structure. The tensor product is given on objects by
\begin{equation}
    (F_1 \otimes F_2)(g) = \bigoplus_{xy=g} F_1(x) \otimes F_2(y)
\end{equation}
and on morphisms via the diagonal $H$-action. The braiding $\beta_{F_1, F_2} \colon F_1 \otimes F_2 \xRightarrow{\sim} F_2 \otimes F_1$ is defined on the homogeneous components $F_1(x) \otimes F_2(y)$ by
\begin{equation}
    \beta(v \otimes w) = F_2\big(y \xrightarrow{\chi(x)} \chi(x) \rhd y\big)(w) \otimes v \ ,
\end{equation}
for $v \in F_1(x)$ and $w \in F_2(y)$. This braiding is well-defined because $\chi$ is a crossed module: the Peiffer identity ensures degree compatibility, and the other crossed module condition guarantees that $\beta_{F_1,F_2}$ is a valid natural transformation. It is a straightforward exercise to verify that with this structure, the identification $\mathrm{Fun}(G \sslash H, \mathrm{vec}) \simeq (\mathrm{vec}_G^{\kappa;R})^H$ is an equivalence of braided monoidal categories.
\end{example}

\section{Twisted-local modules}\label{section twisted-local mod}

In this section, we introduce doubly central algebras in a monoidal category and twisted-local modules over them (Definition~\ref{twisted local modules}).
We establish the conditions under which the category of twisted-local modules forms a $\chi$-crossed braided, balanced, or ribbon category (Theorem~\ref{twisted local module crossed braiding}, Proposition~\ref{twisted local module crossed balanced} and Theorem~\ref{twisted local module crossed ribbon}).
In Theorem~\ref{thm:inner-direct-sum-for-algebra-modules}
we give conditions on a doubly central algebra which ensure that every module can be uniquely decomposed into a direct sum of twisted local modules. As a class of examples, in Section~\ref{twisted local module examples} we discuss doubly central structures and twisted local modules for group algebras in graded vector spaces.

Doubly central algebras and their twisted local modules generalise the well-known construction of (twisted) local modules over a commutative algebra object in a braided tensor category~\cite{P95, Ki01a, Ki01b, KO01,Mc19}. See also~\cite{BJLP18, GLM24a} where modules over a $G$-equivariant algebra in the neutral component of a $G$-crossed braided category are studied.

\medskip

Throughout this section, we fix an abelian monoidal category $\mathcal{C}$ with right exact tensor product. We do not assume $\mathcal{C}$ to be linear, but note that $\End(\mathbbm{1})$ is automatically a commutative ring.
Without loss of generality we assume the monoidal structure of $\mathcal{C}$ to be strict. 

\subsection{Algebras and modules}\label{algebras and modules}

An object $A \in \mathcal{C}$ is called a \textit{unital associative algebra} (or simply an algebra, for brevity) if there exist a multiplication $\mu\colon A \otimes A \to A$ and a unit $\eta\colon \mathbbm{1} \to A$,
such that
\begin{equation}\label{eq: algbera def}
    \mu\circ(\mu\otimes \mathrm{id}_A) = \mu\circ(\mathrm{id}_A\otimes \mu) \ ,
\qquad 
    \mu\circ (\eta \otimes \mathrm{id}_A)=\mathrm{id}_A = \mu\circ (\mathrm{id}_A\otimes \eta) \ .
\end{equation}
If the ambient category $\mathcal{C}$ is braided, then an algebra $A$ is called \textit{commutative} if
\begin{equation}\label{eq: comm alg}
    \mu = \mu \circ C_{A,A} \ ,
\end{equation}
where $C_{A,A}$ is the braiding of $A$ with itself.

For a given algebra $A \in \mathcal{C}$, a \textit{left A-module} is an object $M\in \mathcal{C}$ with a morphism $\rho_M: A \otimes M \to M$ satisfying
\begin{equation}\label{eq: module def}
    \rho_M\circ(\mathrm{id}_A \otimes \rho_M)=\rho_M\circ(\mu\otimes\mathrm{id}_M) \ , \qquad
    \mathrm{id}_M = \rho_M\circ(\eta\otimes\mathrm{id}_M) \ .
\end{equation}

In the following, we introduce a diagrammatic presentation for algebras and their modules. 
The basic morphisms appearing above are depicted as follows:
\begin{equation}
\begin{tikzpicture}[baseline=0pt,font=\scriptsize,
  strand/.style={draw, line width=0.7pt}
]
\node[eq, scale=0.8] at (-1,0) {$\mu\,=$};
\begin{scope}[xshift=0cm, yshift=0cm, scale=0.7]
  \node at (-0.6,-1.0) {$A$};
  \node at ( 0.6,-1.0) {$A$};
  \draw[strand] (-0.6,-0.8) to[out=90,in=180] (0,0.2);
  \draw[strand] ( 0.6,-0.8) to[out=90,in=0] (0,0.2);
  \draw[strand] (0,0.2) -- (0,1.1);
  \node at (0,1.3) {$A$};
\end{scope}
\node[eq, scale=0.8] at (1.75,0) {$\eta\,=$};
\begin{scope}[xshift=2.5cm, yshift=0cm, scale=0.7]
  \node[open dot, inner sep=1.2pt] at (-0.2,-0.4) {};
  \draw[strand] (-0.2,-0.33) to (-0.2,0.6);
  \node at (-0.2,0.8) {$A$};
\end{scope}
\node[eq, scale=0.8] at (3.6,0) {$\rho_M\,=$};
\begin{scope}[xshift=4.4cm, yshift=0cm, scale=0.7]
  \node at (-0.4,-1.0) {$A$};
  \node at ( 0.6,-1.0) {$M$};
  \draw[strand, line width=0.7pt, -{Triangle[length=5pt,width=3pt]}]
  (-0.4,-0.8) to[out=90,in=190] (0.6,0.2);
  \draw[strand] ( 0.6,-0.8) to (0.6,1.1);
   \node at (0.6,1.3) {$M$};
\end{scope}
\node[eq, scale=0.8] at (6.5,0) {$C_{U,V}=$};
\begin{scope}[xshift=7.7cm, yshift=0cm, scale=0.7]
  \node at (-0.6,-1.0) {$U$};
  \node at ( 0.6,-1.0) {$V$};
  \draw[strand] (-0.6,-0.8) to (0.6,1.1);
  \draw[strand] ( 0.6,-0.8) to (0.1,0);
  \draw[strand] (-0.1,0.3) -- (-0.6,1.1);
  \node at (0.6,1.3) {$U$};
  \node at (-0.6,1.3) {$V$};
\end{scope}
\node[eq, scale=0.8] at (9.7,0) {$C^{-1}_{U,V}=$};
\begin{scope}[xshift=10.9cm, yshift=0cm, scale=0.7]
  \node at (-0.6,-1.0) {$V$};
  \node at ( 0.6,-1.0) {$U$};
  \draw[strand] (0.6,-0.8) to (-0.6,1.1);
  \draw[strand] (-0.6,-0.8) to (-0.1,0);
  \draw[strand] (0.1,0.3) -- (0.6,1.1);
  \node at (0.6,1.3) {$V$};
  \node at (-0.6,1.3) {$U$};
\end{scope}
\end{tikzpicture}
\end{equation}
The defining conditions~\eqref{eq: algbera def} for algebras read:
\begin{equation}\label{eq: algbera def graph}
\begin{tikzpicture}[baseline=0pt,font=\scriptsize,
  strand/.style={draw, line width=0.7pt}]
\begin{scope}[xshift=0cm, yshift=0cm, scale=0.7]
  \node at (-0.8,-1.0) {$A$};
  \node at ( 0.0,-1.0) {$A$};
  \node at ( 0.8,-1.0) {$A$};
  \draw[strand] (-0.8,-0.8) to[out=90,in=180] (-0.4,-0.1);
  \draw[strand] (0.0,-0.8) to[out=90,in=0] (-0.4,-0.1);
  \draw[strand] (-0.4,-0.1) to[out=90,in=180] (0.2,0.6);
  \draw[strand] (0.8,-0.8) to[out=90,in=0] (0.2,0.6);
  \draw[strand] (0.2,0.6) -- (0.2,1.1);
  \node at (0.2,1.3) {$A$};
\end{scope}
\node[eq] at (1.25,0) {$=$};
\begin{scope}[xshift=2.5cm, yshift=0cm, scale=0.7]
  \node at (-0.8,-1.0) {$A$};
  \node at (0.0,-1.0) {$A$};
  \node at (0.8,-1.0) {$A$};
  \draw[strand] (0.8,-0.8) to[out=90,in=0] (0.4,-0.1);
  \draw[strand] (0.0,-0.8) to[out=90,in=180] (0.4,-0.1);
  \draw[strand] (0.4,-0.1) to[out=90,in=0] (-0.2,0.6);
  \draw[strand] (-0.8,-0.8) to[out=90,in=180] (-0.2,0.6);
  \draw[strand] (-0.2,0.6) -- (-0.2,1.1);
  \node at (-0.2,1.3) {$A$};
\end{scope}
\begin{scope}[xshift=6.5cm, yshift=0cm, scale=0.7]
\node[open dot, inner sep=1.2pt] at (-0.6,-0.5) {};
  \node at ( 0.6,-1.0) {$A$};
  \draw[strand] (-0.6,-0.43) to[out=90,in=180] (0,0.2);
  \draw[strand] ( 0.6,-0.8) to[out=90,in=0] (0,0.2);
  \draw[strand] (0,0.2) -- (0,1.1);
  \node at (0,1.3) {$A$};
\end{scope}
\node[eq] at (7.5,0) {$=$};
\begin{scope}[xshift=8.1cm, yshift=0cm, scale=0.7]
  \node at (0,-1.0) {$A$};
  \draw[strand] (0,-0.8) -- (0,1.1);
  \node at (0,1.3) {$A$};
\end{scope}
\node[eq] at (8.7,0) {$=$};
\begin{scope}[xshift=9.7cm, yshift=0cm, scale=0.7]
  \node[open dot, inner sep=1.2pt] at (0.6,-0.5) {};
  \node at (-0.6,-1.0) {$A$};
  \draw[strand] (0.6,-0.43) to[out=90,in=0] (0,0.2);
  \draw[strand] (-0.6,-0.8) to[out=90,in=180] (0,0.2);
  \draw[strand] (0,0.2) -- (0,1.1);
  \node at (0,1.3) {$A$};
\end{scope}
\end{tikzpicture}
\end{equation}
The conditions~\eqref{eq: module def} for modules are presented as:
\begin{equation}
\begin{tikzpicture}[baseline=0pt,font=\scriptsize,
  strand/.style={draw, line width=0.7pt}]
\begin{scope}[xshift=0cm, yshift=0cm, scale=0.7]
  \node at (-1.4,-1.0) {$A$};
  \node at (-1,-1.0) {$A$};
  \node at ( 0,-1.0) {$M$};
  \draw[strand,line width=0.7pt,-{Triangle[length=5pt,width=3pt]}] (-1,-0.8) to[out=90,in=190] (0,0.2);
  \draw[strand,line width=0.7pt,-{Triangle[length=5pt,width=3pt]}] (-1.4,-0.8) to[out=90,in=190] (0,0.6);
  \draw[strand] ( 0,-0.8) to (0,1.1);
   \node at (0,1.3) {$M$};
\end{scope}
\node at (0.7,0) {$=$};
\begin{scope}[xshift=2.5cm, yshift=0cm, scale=0.7]
  \node at (-1.4,-1.0) {$A$};
  \node at (-0.6,-1.0) {$A$};
  \node at (0,-1.0) {$M$};
  \draw[strand] (-1.4,-0.8) to[out=90,in=180] (-1,0);
  \draw[strand] (-0.6,-0.8) to[out=90,in=0] (-1,0);  
  \draw[strand,line width=0.7pt,-{Triangle[length=5pt,width=3pt]}] (-1,0) to[out=90,in=190] (0,0.6);
  \draw[strand] ( 0,-0.8) to (0,1.1);
   \node at (0,1.3) {$M$};
\end{scope}
\begin{scope}[xshift=5cm, yshift=0cm, scale=0.7]
  \node at (0.6,-1.0) {$M$};
  \draw[strand] ( 0.6,-0.8) to (0.6,1.1);
  \node at (0.6,1.3) {$M$};
\end{scope}
\node at (6.3,0) {$=$};
\begin{scope}[xshift=7.3cm, yshift=0cm, scale=0.7]
  \node[open dot, inner sep=1.2pt] at (-0.4,-0.5) {};
  \node at (0.6,-1.0) {$M$};
  \draw[strand,line width=0.7pt,-{Triangle[length=5pt,width=3pt]}] (-0.4,-0.43) to[out=90,in=190] (0.6,0.2);
  \draw[strand] ( 0.6,-0.8) to (0.6,1.1);
   \node at (0.6,1.3) {$M$};
\end{scope}
\end{tikzpicture}
\end{equation}
The commutativity condition~\eqref{eq: comm alg} is expressed as:
\begin{equation}
    \begin{tikzpicture}[baseline=0pt,font=\scriptsize,
  strand/.style={draw, line width=0.7pt}]
        \begin{scope}[xshift=10.3cm, yshift=0cm, scale=0.7]
  \node at (-0.6,-1.0) {$A$};
  \node at ( 0.6,-1.0) {$A$};
  \draw[strand] (-0.6,-0.8) to[out=90,in=180] (0,0.2);
  \draw[strand] ( 0.6,-0.8) to[out=90,in=0] (0,0.2);
  \draw[strand] (0,0.2) -- (0,1.1);
  \node at (0,1.3) {$A$};
\end{scope}
\node[eq] at (11.3,0) {$=$};
\begin{scope}[xshift=12.3cm, yshift=0cm, scale=0.7]
  \node at (-0.6,-1.0) {$A$};
  \node at ( 0.6,-1.0) {$A$};
  \draw[strand] (-0.6,-0.8) to[out=90,in=225] (0,-0.2);
  \draw[strand] (0.6,-0.8) to[out=90,in=-45] (0.1,-0.3);
  \draw[strand] (0,-0.2) to[out=45,in=0] (0,0.7);
  \draw[strand] (-0.1,-0.1) to[out=135,in=180] (0,0.7);
  \draw[strand] (0,0.7) -- (-0,1.1);
  \node at (0,1.3) {$A$};
\end{scope}
\end{tikzpicture}
\end{equation}

Let $A$ be an algebra in $\mathcal{C}$. Suppose there exists a commutative algebra $\overline{A}$ in the Drinfeld centre $Z(\mathcal{C})$ of $\mathcal{C}$ whose underlying algebra structure forgets to $A$.
For any $U \in \mathcal{C}$, we denote by $\overline{\sigma}_U\colon A \otimes U \to U \otimes A$ the half-braiding associated with $\overline{A}$.
Such a pair $(A, \overline \sigma)$ is also called a \textit{central algebra} in $\mathcal C$~\cite[Def.\,5.1]{LW22}.
By definition we have
\begin{subequations}\label{eq: algebra compatible half braiding}
    \begin{gather}
         \overline{\sigma}_U \circ (\eta \otimes \mathrm{id}_U) = \mathrm{id}_U \otimes \eta \ , \label{eq: algebra compatible half braiding a}\\
         \overline{\sigma}_U \circ (\mu \otimes \mathrm{id}_U) = (\mathrm{id}_U \otimes \mu) \circ (\overline{\sigma}_U \otimes \mathrm{id}_A) \circ (\mathrm{id}_A \otimes \overline{\sigma}_U) \ . \label{eq: algebra compatible half braiding b}
    \end{gather}
\end{subequations}
We depict $\overline{\sigma}_U$ and its inverses as follows:
\begin{equation}
\begin{tikzpicture}[baseline=0pt,font=\scriptsize,
  strand/.style={draw, line width=0.7pt}]
\node[scale=1.3] at (-1.2,0) {$\overline{\sigma}_U \ = $};
\begin{scope}[xshift=0cm, yshift=0cm, scale=0.7]
  \node at (-0.6,-1.0) {$A$};
  \node at ( 0.6,-1.0) {$U$};
  \draw[strand] (-0.6,-0.8) to (0.6,1.1);
  \draw[strand] ( 0.6,-0.8) to (0.15,-0.08);
  \draw[strand] (-0.15,0.38) -- (-0.6,1.1);
  \node[circle,fill=black,inner sep=1.4pt] at (0,0.15) {};
  \node at (0.6,1.3) {$A$};
  \node at (-0.6,1.3) {$U$};
\end{scope}
\node[scale=1.3] at (2.3,0) {$\overline{\sigma}^{-1}_U \ = $};
\begin{scope}[xshift=3.5cm, yshift=0cm, scale=0.7]
  \node at (-0.6,-1.0) {$U$};
  \node at ( 0.6,-1.0) {$A$};
  \draw[strand] (0.6,-0.8) to (-0.6,1.1);
  \draw[strand] (-0.6,-0.8) to (-0.15,-0.08);
  \draw[strand] (0.15,0.38) -- (0.6,1.1);
  \node[circle,fill=black,inner sep=1.4pt] at (0,0.15){};
  \node at (0.6,1.3) {$U$};
  \node at (-0.6,1.3) {$A$};
\end{scope}
\end{tikzpicture}
\end{equation}
Equation~\eqref{eq: algebra compatible half braiding} then reads
\begin{equation}\label{eq: algebra compatible half braiding pic}
\begin{tikzpicture}[baseline=0pt,font=\scriptsize,
  strand/.style={draw, line width=0.7pt},
  eq/.style={font=\normalsize},
  open dot/.style={circle,draw,fill=white,inner sep=1.2pt}
]
\begin{scope}[xshift=0cm, yshift=0cm, scale=0.75]
  \node[open dot] at (-0.6,-0.55) {};
  \node at (0.6,-1.0) {$U$};
  \draw[strand] (-0.57,-0.48) -- (0.6,1.1);
  \draw[strand] (0.6,-0.8) -- (0.085,0.015);
  \draw[strand] (-0.155,0.395) -- (-0.6,1.1);
  \node[circle,fill=black,inner sep=1.4pt] at (-0.047,0.197) {};
  \node at (-0.6,1.3) {$U$};
  \node at (0.6,1.3) {$A$};
\end{scope}

\node[eq] at (1.2,0) {$=$};

\begin{scope}[xshift=2.6cm, yshift=0cm, scale=0.75]
  \node at (-0.6,-1.0) {$U$};
  \draw[strand] (-0.6,-0.8) -- (-0.6,1.1);
  \node at (-0.6,1.3) {$U$};

  \node[open dot] at (0.6,-0.55) {};
  \draw[strand] (0.6,-0.48) -- (0.6,1.1);
  \node at (0.6,1.3) {$A$};
\end{scope}

\begin{scope}[xshift=6cm, yshift=0cm, scale=0.75]
  \node at (-0.8,-1.0) {$A$};
  \node at (0.0,-1.0) {$A$};
  \node at (0.8,-1.0) {$U$};
  \draw[strand] (-0.8,-0.8) to[out=90,in=180] (-0.4,-0.15);
  \draw[strand] (0.0,-0.8) to[out=90,in=0] (-0.4,-0.15);
  \draw[strand] (-0.4,-0.15) -- (0.6,1.1);
  \draw[strand] (0.8,-0.8) -- (0.113,0.132);
  \draw[strand] (-0.153,0.493) -- (-0.6,1.1);
  \node[circle,fill=black,inner sep=1.4pt] at (-0.05,0.29) {};
  \node at (-0.6,1.3) {$U$};
  \node at (0.6,1.3) {$A$};
\end{scope}

\node[eq] at (7.4,0) {$=$};

\begin{scope}[xshift=9cm, yshift=0cm, scale=0.75]
  \node at (-0.9,-1.0) {$A$};
  \node at ( 0.0,-1.0) {$A$};
  \node at ( 0.9,-1.0) {$U$};

  \coordinate (c1) at ( 0.42,-0.192); 
  \coordinate (c2) at (-0.18, 0.568); 
  \coordinate (m)  at ( 0.5, 0.8);  

  \draw[strand] ( 0.9,-0.8) -- ( 0.54,-0.344);
  \draw[strand] ( 0.30,-0.040) -- (-0.06, 0.416);
  \draw[strand] (-0.30, 0.720) -- (-0.6, 1.1);

  \draw[strand]
    (0.0,-0.8)
    to[out=90,in=-150] (c1)
    to[out=30,in=0] (m);

  \draw[strand]
    (-0.9,-0.8)
    to[out=90,in=-150] (c2)
    to[out=30,in=180] (m);

  \draw[strand]
    (0.65,0.75) to[out=60,in=-90] (0.8,1.1);

  \node[circle,fill=black,inner sep=1.4pt] at (c1) {};
  \node[circle,fill=black,inner sep=1.4pt] at (c2) {};

  \node at (-0.6,1.3) {$U$};
  \node at (0.8,1.3) {$A$};
\end{scope}
\end{tikzpicture}
\end{equation}

Denote by $_{A}\mathcal{C}$ the category of left $A$-modules and $A$-intertwiners in $\mathcal C$. We define a bifunctor
\begin{gather}\label{monoidal str on _AC}
        \otimes_A\colon {}_A\mathcal{C} \times {}_A\mathcal{C} \to {}_A\mathcal{C} \ , \qquad
        (M, N) \mapsto M \otimes_A N \ , \qquad
        (f, g) \mapsto f \otimes_A g \ ,
\end{gather}
where $(M \otimes_A N, \pi_{M,N})$ is the coequaliser
\begin{equation}\label{eq: tensor over A}
\begin{tikzcd}[column sep = 9em]
A \otimes M \otimes N \ar[r,shift left=.75ex,"\rho_M \otimes \mathrm{id}_N"]
  \ar[r,shift right=.75ex,swap,"(\mathrm{id}_M \otimes \rho_N) \circ (\overline{\sigma}_{M} \otimes \mathrm{id}_N)"]
&
M \otimes N \ar[r,"\pi_{M,N}"]
&
M \otimes_A N
\end{tikzcd},
\end{equation}
or equivalently, since $\mathcal C$ is abelian,
\begin{equation}\label{A-module monoidal structure}
    M \otimes_A N = \mathrm{Coker}\big(\rho_M \otimes \mathrm{id}_N-(\mathrm{id}_M \otimes \rho_N) \circ (\overline{\sigma}_{M} \otimes \mathrm{id}_N)\big) \ ,
\end{equation}
and $f \otimes_A g$ is the unique morphism such that the following diagram commutes:
\begin{equation}\label{morphism tensor product}
\begin{tikzcd}[column sep = 4em, row sep = 1.5em]
M\otimes N \ar[r, "\pi_{M,N}"]\ar[d, "f \otimes g"]
&
M \otimes_A N \ar[d, swap, dashed, "\exists ! \, f \otimes_A g"]\\
M'\otimes N' \ar[r, "\pi_{M',N'}"]
&
M'\otimes_A N'
\end{tikzcd}
\end{equation}
The $A$-module structure on $M \otimes_A N$ is given by the unique morphism $\rho_{M \otimes_A N}$ making the following diagram commute:
\begin{equation}
\begin{tikzcd}
	{A \otimes M \otimes N} && {A \otimes M\otimes_A N} \\
	{M \otimes N} && {M \otimes_A N}
	\arrow["{\mathrm{id}_A \otimes \pi_{M,N}}", from=1-1, to=1-3]
	\arrow["{\rho_M \otimes \mathrm{id}_N}", from=1-1, to=2-1]
	\arrow["{\exists ! \, \rho_{M \otimes_A N}}"', dashed, from=1-3, to=2-3]
	\arrow["{\pi_{M,N}}", from=2-1, to=2-3]
\end{tikzcd}
\end{equation}
Indeed, since $\mathcal C$ has right exact tensor product, the functor $A\otimes -:\mathcal C\to\mathcal C$ is right exact, hence $\mathrm{id}_A \otimes \pi_{M,N}$ is again a coequaliser.
Therefore, the existence of $\rho_{M\otimes_A N}$ follows from the fact that $\pi_{M,N}\circ(\rho_M\otimes \mathrm{id}_N)$ coequalises \begin{equation}
\begin{tikzcd}[column sep = 11em]
A \otimes A \otimes M \otimes N \ar[r,shift left=.75ex,"\id_A \otimes \rho_M \otimes \mathrm{id}_N"]
  \ar[r,shift right=.75ex,swap,"\id_A \otimes \bigl((\mathrm{id}_M \otimes \rho_N) \circ (\overline{\sigma}_{M} \otimes \mathrm{id}_N)\bigr)"]
&
A \otimes M \otimes N
\end{tikzcd},
\end{equation}
which is a consequence of~\eqref{eq: algebra compatible half braiding b}.
One can further check that this is indeed a module action. This makes $\pi_{M,N}$ an $A$-intertwiner, and hence $f \otimes_A g$ defined by~\eqref{morphism tensor product} is also $A$-linear.
For further details on the category of $A$-modules, see, for example,~\cite{FS01, O01}.

Although the following result is well known~\cite{Sch01},
we include a sketch of the proof for the reader’s convenience.
\begin{proposition}\label{monoidal structure on A-mod}
     Let $(A, \overline{\sigma})$ be a
     central algebra in $\mathcal C$. Then $_{A}\mathcal{C}$
     is an abelian monoidal category with tensor product $\otimes_A$ and tensor unit the regular module $A$.
\end{proposition}

\begin{proof}
${}_A\mathcal{C}$ is canonically equivalent to the Eilenberg--Moore category of the monad
\begin{equation}
    T=\bigl(A\otimes -,\, \mu\otimes \mathrm{id},\, \eta\otimes \mathrm{id}\bigr)
\end{equation}
on $\mathcal{C}$. Since $A\otimes -$ is right exact, $_A\mathcal{C}$ is abelian~\cite[Prop.\,5.3]{EM65}.

We next check the monoidal structure. Every left $A$-module $M$ admits a right
$A$-module structure with the right action
$\overline{\lambda}_M \colon M\otimes A \to M$ defined by
\begin{equation}\label{right module}
    \overline{\lambda}_M \;\coloneqq\; \rho_M \circ \overline{\sigma}_M^{-1} \ .
\end{equation}
One can check that $\overline{\lambda}_M$ satisfies the axioms of a right $A$-action; in particular, the associativity of $\overline{\lambda}_M$ follows from that of $\rho_M$, combined with the compatibility of the half-braiding $\overline{\sigma}$ with the multiplication of $A$~\eqref{eq: algebra compatible half braiding b}. This yields
a full embedding of ${}_A\mathcal{C}$ into the category of $A$-$A$ bimodules in $\mathcal{C}$:
\begin{equation}
    \iota \colon {}_A\mathcal{C} \hookrightarrow {}_A\mathcal{C}_A \ .
\end{equation}

It is well known that ${}_A\mathcal{C}_A$ is monoidal under the relative tensor product
$\otimes'_A$,
with tensor unit the regular $A$-$A$-bimodule, whose left and
right actions are both given by the multiplication of~$A$; see, for
instance,~\cite[Sec.\,2.2]{Sch01}.
Since $(A, \overline \sigma)$ is commutative in $Z(\mathcal C)$, $\iota$ sends the regular left $A$-module to the regular $A$-$A$-bimodule, and hence the tensor unit of ${}_A\mathcal C_A$. Moreover, for any $M,N\in{}_A\mathcal{C}$, the tensor product
$\iota(M)\otimes'_A \iota(N)$ in ${}_A\mathcal{C}_A$ is given by the coequaliser
\begin{equation}
\mathrm{Coker}\bigl(\overline{\lambda}_M\otimes\mathrm{id}_N-\mathrm{id}_M\otimes\rho_N\bigr) \ ,
\end{equation}
which is canonically isomorphic to $M\otimes_A N$ as defined in~\eqref{A-module monoidal structure},
via the unique isomorphism induced by $\mathrm{id}\colon M\otimes N \to M\otimes N$.
Finally, because ${}_A\mathcal{C}$ is a full subcategory of ${}_A\mathcal{C}_A$, the unitors and associators in $_A\mathcal{C}$ are simply given by those in ${}_A\mathcal{C}_A$. 
We can therefore identify ${}_A\mathcal{C}$ with a monoidal subcategory of ${}_A\mathcal{C}_A$~\cite[Lem.\,4.1]{Sch01}.
\end{proof}

For later use, we also give the unitors and the associator:
\begin{itemize}
\item The left unitor $l_M\colon A\otimes_A M \to M$ is the unique morphism making the following diagram commute:
\begin{equation}\label{left unitor}
\begin{tikzcd}[column sep = 4em, row sep = 1.5em]
A\otimes M \ar[r,"{\pi_{A,M}}"] \ar[d,swap,"{\rho_M}"] &
A\otimes_A M \ar[dl,dashed,"{\exists !\,l_M}"] \\
M &
\end{tikzcd}
\end{equation}

\item The right unitor $r_M\colon M\otimes_A A \to M$ is the unique morphism making the following diagram commute:
\begin{equation}\label{right unitor}
\begin{tikzcd}[column sep = 4em, row sep = 1.5em]
M\otimes A \ar[r,"{\pi_{M,A}}"] \ar[d,"{\overline{\sigma}_M^{-1}}"] &
M\otimes_A A \ar[d, swap, dashed,"{\exists !\, r_M}"] \\
A\otimes M \ar[r,"{\rho_M}"] &
M
\end{tikzcd}
\end{equation}

\item the associator $\alpha_{M,N,L}\colon M \otimes_A (N \otimes_A L) \to (M \otimes_A N) \otimes_A L$ is the unique morphism making the following diagram commute:
\begin{equation}\label{associator Amod}
\begin{tikzcd}
	{M \otimes N \otimes L} &&&& {M \otimes_A (N \otimes_A L)} \\
	{M \otimes N \otimes L} &&&& {(M \otimes_A N) \otimes_A L}
	\arrow["{\pi_{M,N\otimes_A L} \circ (\mathrm{id}_M \otimes \pi_{N,L})}", from=1-1, to=1-5]
	\arrow[equals, from=1-1, to=2-1]
	\arrow["{\exists ! \, \alpha_{M,N,L}}"', dashed, from=1-5, to=2-5]
	\arrow["{\pi_{M \otimes_A N, L} \circ (\pi_{M,N} \otimes \mathrm{id}_L)}", from=2-1, to=2-5]
\end{tikzcd}
\end{equation}
\end{itemize}

\begin{example}\label{ex:full-centre}
    Let $A$ be an algebra in $\mathcal{C}$. Its full centre $Z(A)\in Z(\mathcal{C})$, introduced in~\cite{Sch04,FFRS06,Da09}, is
    a commutative algebra in $Z(\mathcal{C})$ whenever it exists. In the case where $\mathcal{C}$ is a finite (multi)tensor category over any field $\mathbb{K}$, the existence of $Z(A)$
    for any algebra $A \in \mathcal C$
    is proven in~\cite[Thm.\,3.25]{FGSW26}, where an explicit
    formula of $Z(A)$ is also given.
    For example, when $A=\mathbbm{1}$ is the tensor unit, its full centre $Z(\mathbbm 1)$ is given by the end
    \begin{equation}
        \int_{X\in\mathcal C} X\otimes X^* \ ,
    \end{equation}
    equipped with its canonical half-braiding. Via the monoidal forgetful functor
    \begin{equation}
        U\colon Z(\mathcal{C})\longrightarrow \mathcal{C} \ ,
    \end{equation}
    the object $UZ(A)$ is an algebra in $\mathcal C$. Together with the half-braiding inherited from $Z(A)$, this gives a central algebra in $\mathcal C$. However, the underlying algebra $UZ(A)$ need not be commutative as an algebra in $\mathcal C$, even when $\mathcal C$ is braided.

Assume now that $\mathcal C$ is a finite multitensor category with enough
projective objects, that $A$ as regular right $A$-module is indecomposable, and that the category $\mathcal C_A$ of right $A$-modules in $\mathcal C$ is an exact left $\mathcal C$-module category in  the sense of~\cite[Def.\,7.5.1]{EGNO15}. Denote by $\mathrm{Rex}(\mathcal C_A)$ the category of right exact endofunctors of $\mathcal C_A$. There is an equivalence of
categories
\begin{equation}
        {}_{UZ(A)}\mathcal C
        \simeq
        \mathrm{Rex}(\mathcal C_A) \ .
\end{equation}
Indeed, this follows from the monadicity of the right adjoint to the action functor
\begin{equation}
        \mathcal C \longrightarrow \mathrm{Rex}(\mathcal C_A) \ ,
        \qquad
        X\longmapsto X\otimes - \ ,
\end{equation}
together with the identification of the corresponding adjoint algebra
with the full centre $Z(A)$, see \cite[Thm.\,3.4]{Sh18} and \cite[Thm.\,3.25]{FGSW26}. It would be interesting to compare the monoidal structure on ${}_{UZ(A)}\mathcal C$ induced by the central algebra structure of $UZ(A)$, as in Proposition~\ref{monoidal structure on A-mod}, with the natural monoidal structure on $\mathrm{Rex}(\mathcal C_A)$ given by composition of endofunctors.
\end{example}

\subsection{Twisted-local modules}\label{twisted-local modules}

We next define twisted-local modules by comparing the right $A$-module structures induced on a left $A$-module by two half-braidings of $A$, with one of these structures twisted by an automorphism of $A$.

\begin{definition}
    A \textit{doubly central
    algebra} in $\mathcal{C}$ is a triple $(A, \overline{\sigma}, \widetilde{\sigma})$, where $A \in \mathcal C$ is an algebra,
    and $\overline{\sigma}, \widetilde{\sigma}$ are half-braidings of $A$ such that $\overline{A} = (A, \overline{\sigma})$ and $\widetilde{A} = (A, \widetilde{\sigma})$ are commutative algebras in $Z(\mathcal{C})$ whose underlying algebra structures forget to $A$.
\end{definition}

For any $U \in \mathcal C$, we further depict
$\widetilde{\sigma}_U$, together with its inverses, as follows:
\begin{equation}
\begin{tikzpicture}[baseline=0pt,font=\scriptsize,
  strand/.style={draw, line width=0.7pt}]
\node[scale=1.3] at (5.8,0) {$\widetilde{\sigma}_U \ = $};
\begin{scope}[xshift=7cm, yshift=0cm, scale=0.7]
  \node at (-0.6,-1.0) {$A$};
  \node at ( 0.6,-1.0) {$U$};
  \draw[strand] (-0.6,-0.8) to (0.6,1.1);
  \draw[strand] ( 0.6,-0.8) to (0.15,-0.08);
  \node[open dot, inner sep=1.5pt] at (0,0.15) {};
  \draw[strand] (-0.15,0.38) -- (-0.6,1.1);
  \node at (0.6,1.3) {$A$};
  \node at (-0.6,1.3) {$U$};
\end{scope}
\node[scale=1.3] at (9.3,0) {$\widetilde{\sigma}^{-1}_U \ = $};
\begin{scope}[xshift=10.5cm, yshift=0cm, scale=0.7]
  \node at (-0.6,-1.0) {$U$};
  \node at ( 0.6,-1.0) {$A$};
  \draw[strand] (0.6,-0.8) to (-0.6,1.1);
  \draw[strand] (-0.6,-0.8) to (-0.15,-0.05);
  \node[open dot, inner sep=1.5pt] at (0,0.15) {};
  \draw[strand] (0.15,0.38) -- (0.6,1.1);
  \node at (0.6,1.3) {$U$};
  \node at (-0.6,1.3) {$A$};
\end{scope}
\end{tikzpicture}
\end{equation}
Similarly to~\eqref{right module}, for any left $A$-module $M \in {}_A\mathcal{C}$, $\widetilde{A}$ induces another right $A$-module structure on $M$ with the right $A$-action given by
\begin{equation}
    \widetilde{\lambda}_M = \rho_M \circ \widetilde{\sigma}_M^{-1} \ .
\end{equation}

\begin{example}\label{ex:sweedler}
Let $\mathcal C=H\text{-}\mathrm{mod}$, where $H$ is the Sweedler Hopf
algebra. In~\cite[Sec.\,5]{FGSW26}, a family of
exact $\mathcal C$-module categories $\{\mathcal M_\xi\}_{\xi\in\mathbb C}$ is considered. As linear categories, they are all equivalent to 
$\mathrm{vec}_{\mathbb C}$,
but they carry inequivalent
$\mathcal C$-module structures. For each $\xi$, there is an algebra $B_\xi\in\mathcal C$ such that
\begin{equation}
    \mathcal M_\xi
    \simeq
    \mathcal C_{B_\xi}
\end{equation}
as $\mathcal C$-module categories. It is shown in~\cite[Sec.\,5.3]{FGSW26} that their full centres,
\begin{equation}
    Z(B_\xi) \in Z(\mathcal C)
\end{equation}
are non-isomorphic, even as objects of $Z(\mathcal C)$. On the other hand, their images under the forgetful functor $U\colon Z(\mathcal C)\longrightarrow\mathcal C$ are all isomorphic as algebras in $\mathcal C$.
Comparing to Example \ref{ex:full-centre}, we see that
$A \coloneqq UZ(B_\xi)\in\mathcal C$ admits
infinitely many central algebra structures, and so in particular $A$ is a doubly central algebra in infinitely many ways.
\end{example}

\begin{definition}\label{def:act-on-doubly}
Let $G$ be a group and denote by $\Aut_{\mathcal C}(A)$ the group of automorphisms of $A$ as an object of $\mathcal C$. A \textit{$G$-action on a doubly central algebra} $(A, \overline{\sigma}, \widetilde{\sigma})$ is a group homomorphism
\begin{equation}
    \phi\colon G \to \mathrm{Aut}_\mathcal{C}(A)
\end{equation}
such that for each $g \in G$, 
\begin{equation}
    \phi_g \coloneqq \phi(g)
\end{equation}
lifts to an automorphism of objects in $Z(\mathcal{C})$ for both $\overline{A}$ and $\widetilde{A}$,
i.e.\ for any $U \in \mathcal C$, we have
\begin{equation}\label{eq: naturality of action}
    \overline{\sigma}_U \circ (\phi_g \otimes \id_U) = (\id_U \otimes \phi_g) \circ \overline{\sigma}_U\ , \qquad
    \widetilde{\sigma}_U \circ (\phi_g \otimes \id_U) = (\id_U \otimes \phi_g) \circ \widetilde{\sigma}_U\ .
\end{equation}
\end{definition}

We can now give the main definition of this section.

\begin{definition}\label{twisted local modules}
Let $(A, \overline{\sigma}, \widetilde{\sigma})$ be a doubly central algebra in $\mathcal{C}$ and let $\phi$ be a $G$-action on it.
For $g \in G$, an $A$-module $M$ is called \textit{$g$-local (w.r.t. $\overline{\sigma}$, $\widetilde{\sigma}$, and $\phi$}) 
if
\begin{equation}\label{eq:twist-local-condition}
    \overline{\lambda}_M = \widetilde{\lambda}_M \circ (\mathrm{id}_M \otimes \phi_g) \ ,
\end{equation}
or equivalently, using the left module convention and applying~\eqref{eq: naturality of action},
\begin{equation}\label{eq:g-local-2}
    \rho_M \circ (\phi_g^{-1} \otimes \mathrm{id}_M) = \rho_M \circ \widetilde{\sigma}^{-1}_M \circ \overline{\sigma}_M \, ,
    \qquad 
    \rho_M \circ (\phi_g \otimes \mathrm{id}_M) = \rho_M \circ \overline{\sigma}^{-1}_M \circ \widetilde{\sigma}_M \,
\end{equation}
Graphically,~\eqref{eq:twist-local-condition} reads:
\begin{equation}\label{eq: twisted local condition}
\begin{tikzpicture}[baseline=0pt,font=\scriptsize,
  strand/.style={draw, line width=0.7pt}]
\begin{scope}[xshift=0cm, yshift=0cm, scale=0.7]
  \node at (0,-1.0) {$M$};
  \node at (0.6,-1.0) {$A$};
  \draw[strand] (0.6,-0.8) to[out=90,in=90] (0.6,-0.22);
  \draw[strand] (0.6,-0.22) to[out=90,in=0] (0,0.1);
  \draw[strand] (0,0.1) to[out=180,in=-90] (-0.3,0.4);
  \draw[strand,line width=0.7pt,-{Triangle[length=5pt,width=3pt]}] (-0.3,0.4) to[out=90,in=200] (0,0.7);
  \draw[strand] (0,-0.8) -- (0,-0.1);
  \draw[strand] (0,0.3) -- (0,1.1);
  \node[circle,fill=black,inner sep=1.4pt] at (0,0.1) {};
\end{scope}

\node at (1.6,0) {$=$};

\begin{scope}[xshift=3cm, yshift=0cm, scale=0.7]
  \node at (0,-1.0) {$M$};
  \node at (0.6,-1.0) {$A$};
  \node[coupon] (f) at (0.6,-0.4) {$\phi_g$};
  \draw[strand] (0.6,-0.8) to[out=90,in=90] (0.6,-0.6);
  \draw[strand] (0.6,-0.21) to[out=90,in=0] (0,0.1);
  \draw[strand] (0,0.1) to[out=180,in=-90] (-0.3,0.4);
  \draw[strand,line width=0.7pt,-{Triangle[length=5pt,width=3pt]}] (-0.3,0.4) to[out=90,in=200] (0,0.7);
  \draw[strand] (0,-0.8) -- (0,-0.1);
  \draw[strand] (0,0.3) -- (0,1.1);
  \node[open dot, inner sep=1.4pt] at (0,0.1) {};
\end{scope}
\end{tikzpicture}
\end{equation}
For each $g \in G$, we denote by $_{A}\mathcal{C}_g$ the full subcategory of $_{A}\mathcal{C}$ consisting of $g$-local $A$-modules.
\end{definition}

\begin{remark}\label{rem:twisted-local}
We clarify several aspects of the definition above.
\begin{itemize}
    
    \item For every $g \in G$, we only require $\phi_g$ to be an automorphism of $A$ as an object, 
    not necessarily as an algebra,
    see Example~\ref{ex: non-algebra automorphism} below.
    However, the following graphical equalities still hold as consequences of \eqref{eq:twist-local-condition} and both $\overline{\lambda}_M, \, \widetilde\lambda_M$ being right actions:
\begin{equation}
\begin{tikzpicture}[baseline=0pt,font=\scriptsize,
  strand/.style={draw, line width=0.7pt}]
\begin{scope}[xshift=0cm, yshift=0cm, scale=0.7]
  \node at (0,-1.3) {$M$};
  \node at (1.5,-1.3) {$A$};
  \node at (0.5,-1.3) {$A$};
  \node[coupon] (f) at (0.5,-0.75) {$\phi_g$};
  \node[coupon] (f) at (1.5,-0.75) {$\phi_g$};
  \draw[strand] (0.5,-1.1) -- (0.5,-0.92);
  \draw[strand] (1.5,-1.1) -- (1.5,-0.92);
  \draw[strand] (0.5,-0.58) to[out=90,in=180] (1,-0.2);
  \draw[strand] (1.5,-0.58) to[out=90,in=0] (1,-0.2);
  \draw[strand] (1,-0.2) to[out=90,in=0] (0,0.3);
  \draw[strand] (0,0.3) to[out=180,in=-90] (-0.3,0.6);
  \draw[strand,line width=0.7pt,-{Triangle[length=5pt,width=3pt]}] (-0.3,0.6) to[out=90,in=200] (0,0.9);
  \draw[strand] (0,-1.1) -- (0,0.1);
  \draw[strand] (0,0.5) -- (0,1.2);
  \node[open dot, inner sep=1.4pt] at (0,0.3) {};
\end{scope}

\node at (2.2,0) {$=$};

\begin{scope}[xshift=3.5cm, yshift=0cm, scale=0.7]
  \node at (0,-1.3) {$M$};
  \node at (1.5,-1.3) {$A$};
  \node at (0.5,-1.3) {$A$};
  \draw[strand] (0.5,-1.1) to[out=90,in=180] (1,-0.6);
  \draw[strand] (1.5,-1.1) to[out=90,in=0] (1,-0.6);
  \node[coupon] (f) at (1,-0.2) {$\phi_g$};
  \draw[strand] (1,-0.6) to[out=90,in=90] (1,-0.41);
  \draw[strand] (1,-0.02) to[out=90,in=0] (0,0.3);
  \draw[strand] (0,0.3) to[out=180,in=-90] (-0.3,0.6);
  \draw[strand,line width=0.7pt,-{Triangle[length=5pt,width=3pt]}] (-0.3,0.6) to[out=90,in=200] (0,0.9);
  \draw[strand] (0,-1.1) -- (0,0.1);
  \draw[strand] (0,0.5) -- (0,1.2);
  \node[open dot, inner sep=1.4pt] at (0,0.3) {};
\end{scope}

\begin{scope}[xshift=7.5cm, yshift=0cm, scale=0.7]
  \node at (0,-1.3) {$M$};
   \node[open dot, inner sep=1.1pt] at (0.6,-0.85) {};
  \draw[strand] (0.6,-0.77) to[out=90,in=90] (0.6,-0.22);
  \draw[strand] (0.6,-0.22) to[out=90,in=0] (0,0.1);
  \draw[strand] (0,0.1) to[out=180,in=-90] (-0.3,0.4);
  \draw[strand,line width=0.7pt,-{Triangle[length=5pt,width=3pt]}] (-0.3,0.4) to[out=90,in=200] (0,0.7);
  \draw[strand] (0,-1.1) -- (0,-0.1);
  \draw[strand] (0,0.3) -- (0,1.2);
  \node[open dot,inner sep=1.4pt] at (0,0.1) {};
\end{scope}

\node at (9.1,0) {$=$};

\begin{scope}[xshift=10.5cm, yshift=0cm, scale=0.7]
  \node at (0,-1.3) {$M$};
   \node[open dot, inner sep=1.1pt] at (0.6,-0.85) {};
  \node[coupon] (f) at (0.6,-0.4) {$\phi_g$};
  \draw[strand] (0.6,-0.77) to[out=90,in=90] (0.6,-0.62);
  \draw[strand] (0.6,-0.2) to[out=90,in=0] (0,0.1);
  \draw[strand] (0,0.1) to[out=180,in=-90] (-0.3,0.4);
  \draw[strand,line width=0.7pt,-{Triangle[length=5pt,width=3pt]}] (-0.3,0.4) to[out=90,in=200] (0,0.7);
  \draw[strand] (0,-1.1) -- (0,-0.1);
  \draw[strand] (0,0.3) -- (0,1.2);
  \node[open dot, inner sep=1.4pt] at (0,0.1) {};
\end{scope}
\end{tikzpicture}
\end{equation}

\item By definition of the half-braiding, naturality with respect to morphisms on the $A$-strand is not automatic. The reason to require that each $\phi_g$ lifts to automorphisms of both $\overline{A}$ and $\widetilde{A}$ is so that we are able to move it past the corresponding half-braidings.

\item Since both $\overline{A}$ and $\widetilde{A}$ are commutative in $Z(\mathcal{C})$,
the tensor unit $A$ as the regular left $A$-module is $1$-local, i.e.
\begin{equation}
    A \in {}_{A}\mathcal{C}_{1} \ .
\end{equation}

\item It may occur that an \(A\)-module \(M\) is simultaneously \(g\)-local and \(k\)-local for distinct \(g,k\in G\).
We will see examples of this in Example~\ref{ex: non-algebra automorphism} and in Section~\ref{superalgebra example}. Also note that $_A\mathcal C_g$ may be 0 for some $g \in G$.

\item If $\mathcal{C}$ is braided, and if we choose $\overline{\sigma}_M = C_{A,M}$, $\widetilde{\sigma}_M = C^{-1}_{M,A}$, then the notion of a $1$-local module coincides with that of a dyslexic module in the sense of~\cite{P95}.
See also \cite{KO01}, where the category of such modules is denoted $\mathcal{C}_A^0$, or \cite{FRS02}, where the term ``local module'' is used.

\item One can define another monoidal structure $\widetilde{\otimes}_A$ on $_A\mathcal{C}$ via
\begin{equation}
    M \widetilde{\otimes}_A N = \mathrm{Coker}(\widetilde{\lambda}_M \otimes \mathrm{id}_N-\mathrm{id}_M \otimes \rho_N) \ .
\end{equation}
$\widetilde{\otimes}_A$ coincides with $\otimes_A$ on $_{A}\mathcal{C}_{1}$.
\end{itemize}
\end{remark}

\begin{example}\label{ex: non-algebra automorphism}
Let $\mathbb{K}$ be a field, 
$\mathcal C = \mathrm{vec}_\mathbb K$, and $A = \mathbb{K}[x]/(x^2)$. 
Recalling that $Z(\mathrm{vec}_\mathbb{K}) \cong \mathrm{vec}_\mathbb{K}$,
we choose both half-braidings
$\overline{\sigma}$ and $\widetilde{\sigma}$ to be the flip.
 Let $G=\mathbb{K}$ and define
\begin{equation} \phi\colon \mathbb K \longrightarrow \mathrm{Aut}(A) \ , \qquad \alpha \longmapsto \begin{cases} 1 \longmapsto 1+ \alpha x\ ,\\ x \longmapsto x \ . \end{cases} \end{equation}
For each $\alpha\neq 0$, $\phi_\alpha$ is only a linear map, but not an algebra map. In this setting, every $A$-module is $0$-local. Moreover, for every $\alpha\neq 0$, the
category ${}_A\mathcal C_\alpha$ is
semisimple with the (up to isomorphism) unique simple object $\mathbb{K}^0$ given by the one-dimensional
representation on which $x$ acts by $0$. Note that $\mathbb{K}^0$ admits no dual. By contradiction, suppose that $\mathbb{K}^0$ has a dual $X$. Then there exists an evaluation morphism
$\mathrm{ev}\colon X \otimes_A \mathbb K^0 \to A$. Since $x$ acts trivially on $\mathbb K^0$, the $A$-linearity of $\mathrm{ev}$ implies that $\operatorname{Im}(\mathrm{ev}) \subseteq \langle x\rangle$, contradicting the zig-zag identity. Therefore, the category ${}_A\mathcal C
$ is not rigid. In fact,
the dualisable objects in this example are precisely the finite direct sums of the tensor unit.
\end{example}

\begin{lemma}
    For each $g \in G$, the full subcategory $_A\mathcal{C}_g$ is abelian.
\end{lemma}

\begin{proof}
Let $M$ be a $g$-local $A$-module, and let $q\colon M \twoheadrightarrow Q$ be an epimorphism in ${}_A\mathcal C$. 
By naturality of the half-braidings and the fact that $q$ is $A$-linear, we have
\begin{equation}
    \overline{\lambda}_Q \circ (q \otimes \mathrm{id}_A)
    \;=\;
    \widetilde{\lambda}_Q \circ (\mathrm{id}_Q \otimes \phi_g) \circ (q \otimes \mathrm{id}_A)\ .
\end{equation}
Since the tensor product in $\mathcal C$ is right exact, $q\otimes \mathrm{id}_A$ is an epimorphism. Hence we may cancel it and obtain
\begin{equation}
    \overline{\lambda}_Q
    \;=\;
    \widetilde{\lambda}_Q \circ (\mathrm{id}_Q \otimes \phi_g)\ ,
\end{equation}
so $Q$ is $g$-local. By similar arguments, every subobject of $M$ is also $g$-local even without assuming that the tensor product in $\mathcal C$ is right exact.
Therefore ${}_A\mathcal C_g$ is closed under taking subquotients, and is in particular an abelian full subcategory of ${}_A\mathcal C$.
\end{proof}

\begin{lemma}\label{twisted-local modules G-graded monoidal}
    If $M \in {}_{A}\mathcal{C}_g$, $N \in {}_{A}\mathcal{C}_k$, then $M \otimes_A N \in {}_{A}\mathcal{C}_{gk}$.
\end{lemma}

\begin{proof}
We repeatedly use the definitions of twisted-local modules and of the relative tensor product $\otimes_A$:
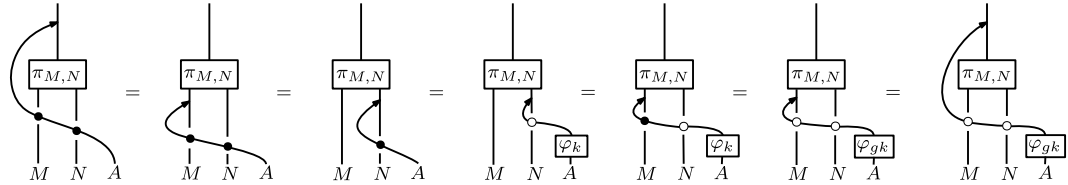
\begin{equation*}
    \raisebox{-35pt}[75pt][0pt]{\includesvg[width=0.95\textwidth]{picture/refined_tensor_product_degree}}
\end{equation*}
By naturality and tensor-compatibility of the half-braidings 
we further obtain
\begin{equation}
  \overline{\lambda}_{M \otimes_A N} \circ (\pi_{M,N} \otimes \id_A)
  =
  \widetilde{\lambda}_{M \otimes_A N}\circ(\id_{M \otimes_A N}\otimes \phi_{gk})\circ(\pi_{M,N}\otimes \id_A)\ .
\end{equation}
Since $\pi_{M,N}$ is an epimorphism and the tensor product in $\mathcal C$ is right exact, the morphism
$\pi_{M,N}\otimes \id_A$ is again an epimorphism. Cancelling it from the right yields
\begin{equation}
  \overline{\lambda}_{M \otimes_A N}
  =
  \widetilde{\lambda}_{M \otimes_A N}\circ(\id_{M \otimes_A N}\otimes \phi_{gk})\ ,
\end{equation}
which is exactly the $gk$-local condition for $M \otimes_A N$.
\end{proof}

For fixed $\big((A, \overline{\sigma}, \widetilde{\sigma}), \phi\big)$, we then obtain an abelian $G$-graded monoidal category
\begin{equation}\label{eq: _AC graded}
    {}_{A}\mathcal{C}^\mathrm{gr} \; \coloneqq \; \bigoplus_{g \in G} {}_{A}\mathcal{C}_g
\end{equation}
by taking the external direct sum. So by construction there exist no morphisms between distinct homogeneous components, even if such morphisms exist in $_A\mathcal C$.
In other words, there is a forgetful functor
\begin{equation}
    U\colon {}_{A}\mathcal{C}^\mathrm{gr} \longrightarrow {}_A\mathcal{C} \ , \qquad M \longmapsto M \ , \qquad f \longmapsto f \ ,
\end{equation}
which is faithful and exact, but not necessarily full.

Let $H$ be a group acting both on $(A, \overline{\sigma}, \widetilde{\sigma})$ (in the sense of Definition~\ref{def:act-on-doubly})
and on the group $G$ by
\begin{gather}\label{eq:H-action-conditions-for-crossed-action}
    \upsilon\colon H \to \mathrm{Aut}_\mathcal{C}(A) \ , \qquad
    \rhd\colon H \times G \to G \ ,
\end{gather}
such that the following conditions hold:
\begin{itemize}
    \item $\upsilon$ and $\rhd$ are compatible in the sense that the diagram

\begin{equation}\label{action compatibility}
\begin{tikzcd}[column sep = 4em, row sep = 1.5em]
H\times G \ar[r, "\rhd"]\ar[d, "\upsilon \times \phi"]
&
G \ar[d, swap, "\phi"]\\
\mathrm{Aut}_\mathcal{C}(A) \times \mathrm{Aut}_\mathcal{C}(A) \ar[r, "\lambda"]
&
\mathrm{Aut}_\mathcal{C}(A)
\end{tikzcd}
\end{equation}
commutes, where $\lambda(f,f')=ff'f^{-1}$ for any 
$(f,f')\in \mathrm{Aut}_\mathcal{C}(A)$;

\item for each $h \in H$, $\upsilon_h\coloneqq\upsilon(h)$ is an algebra morphism, i.e.
    \begin{equation}\label{upsilon alegbra iso}
        \upsilon_h \circ \mu = \mu \circ (\upsilon_h \otimes \upsilon_h) \ , \qquad \eta = \upsilon_h \circ \eta \ .
    \end{equation}
\end{itemize}

\begin{proposition}\label{structure morphism for h-action in AMod}
Let $G$ and $H$ be groups, let $(A,\overline{\sigma},\widetilde{\sigma})$ be a doubly central algebra in~$\mathcal C$, let $\rhd$ be an $H$-action on $G$, and let $\phi$ and $\upsilon$ be respectively $G$- and $H$-actions on $(A,\overline{\sigma},\widetilde{\sigma})$ satisfying~\eqref{action compatibility}--\eqref{upsilon alegbra iso}. Then ${}_A\mathcal C^{\mathrm{gr}}$ carries an $H$-crossed $G$-graded monoidal structure. More precisely:
\begin{enumerate}
    \item There is a strictly monoidal functor
\begin{gather}\label{eq:strict psi}
    \psi\colon \underline{H} \to \mathrm{Aut}_\otimes({_{A}\mathcal{C}}) \ , \qquad
    h \mapsto {^h(-)}
\end{gather}
sending each group element $h$ to a monoidal autoequivalence ${^h(-)}$.

\item  For any $h \in H$ and any homogeneous object $M \in {_{A}\mathcal{C}^\mathrm{gr}}$, it holds that $^hM \in {}_{A}\mathcal{C}_{h \rhd |M|}$.
\end{enumerate}
\end{proposition}

For any $h \in H$, the monoidal functor $^h(-) = \psi(h)$ is defined as follows: for $M, N\in {_{A}\mathcal{C}}$ and a morphism $f\colon M \to N$, 
\begin{itemize}
    \item $^hM \coloneqq M$ as the same object in $\mathcal{C}$, but with the left $A$-action twisted by $\upsilon_h^{-1}$:
    \begin{equation}\label{eq:alg-twisted-action-on-module}
        \rho_{{}^h\!M} = \rho_M \circ (\upsilon_h^{-1} \otimes \mathrm{id}_M) \ ,
    \end{equation}
    or graphically
    \begin{equation}
        \begin{tikzpicture}[baseline=0pt,font=\scriptsize,
  strand/.style={draw, line width=0.7pt}]
            \begin{scope}[xshift=0cm, yshift=0cm, scale=0.7]
            \node at (-0.4,-1.0) {$A$};
            \node at ( 0.6,-1.0) {$^hM$};
            \draw[strand,-{Triangle[length=5pt,width=3pt]}] (-0.4,-0.8) to[out=90,in=190] (0.6,0.2);
            \draw[strand] ( 0.6,-0.8) to (0.6,0.8);
            \end{scope}

             \node at (1.5,0) {$=$};

            \begin{scope}[xshift=3cm, yshift=0cm, scale=0.7]
            \node at (-0.4,-1.0) {$A$};
            \node at (0.6,-1.0) {$M$};
            \node[coupon] (f) at (-0.4,-0.35) {$\upsilon_h^{-1}$};
            \draw[strand] (-0.4,-0.8) to (-0.4,-0.62);
            \draw[strand,-{Triangle[length=5pt,width=3pt]}] (-0.4,-0.07) to[out=90,in=190] (0.6,0.4);
            \draw[strand] ( 0.6,-0.8) to (0.6,0.8);
            \end{scope}
        \end{tikzpicture}
    \end{equation}

    \item $^hf \coloneqq f$,
\end{itemize}
and the monoidal structure morphisms of $^h(-)$ are given by:
\begin{itemize}
    \item the isomorphism $h_{M,N}\colon {^hM \otimes_A {^hN}} \to {^h(M \otimes_A N)}$ and its inverse are the unique isomorphisms making the following diagram commute:
\begin{equation}
\begin{tikzcd}[column sep = 3em, row sep = 1.5em]\label{eq:h-action monoidal}
A \otimes {^hM} \otimes {^hN}\ar[d,swap,"\upsilon_h^{- 1}\otimes \mathrm{id}_{M \otimes N}"]\ar[r,shift left=.75ex,]
  \ar[r,shift right=.75ex,swap,]
&
{^hM} \otimes {^hN}\ar[d,swap,"\mathrm{id}_{M \otimes N}"]\ar[r,"\pi_{{^hM}, {^hN}}"]
&
{^hM} \otimes_A {^hN}\ar[d,dashed, "\exists \, !\ h_{M,N}"]\\
^h(A \otimes M \otimes N) \ar[r,shift left=.75ex,]
  \ar[r,shift right=.75ex,swap,]\ar[d,swap,"\upsilon_h\otimes \mathrm{id}_{M \otimes N}"]
&
^h(M \otimes N)\ar[d,swap,"\mathrm{id}_{M \otimes N}"]\ar[r,"^h\pi_{{M}, {N}}"]
&
^h(M \otimes_A N)\ar[d,dashed, "\exists \, !\ h^{-1}_{M,N}"]\\
A \otimes {^hM} \otimes {^hN} \ar[r,shift left=.75ex,]
  \ar[r,shift right=.75ex,swap,]
&
{^hM} \otimes {^hN}\ar[r,"\pi_{{^hM}, {^hN}}"]
&
{^hM} \otimes_A {^hN}
\end{tikzcd}
\end{equation}
where the missing labels are the defining morphisms of $\otimes_A$;

\item the isomorphism $h_A\colon A \to {^hA}$ is given by
\begin{equation}\label{eq: h-action unitor mod}
    \upsilon_h^{-1}\colon A \to A
\end{equation}
\end{itemize}

Note that below we will sometimes omit boxes around morphisms in the graphical notation to make the diagrams less cluttered, in particular for endomorphisms of $A$.

\begin{proof}[Proof of Proposition~\ref{structure morphism for h-action in AMod}]
(\romannumeral 1)
We check that $^h(-)$ is well-defined and is an autoequivalence.
Since $\upsilon_h$ is an algebra automorphism in $\mathcal{C}$,
it is straightforward to check that $^hM$ is indeed a left $A$-module. Moreover, for any $A$-intertwiner $f$, the morphism ${}^hf$ is also $A$-linear. Hence ${}^h(-)$ is a well-defined endofunctor, and ${}^{h^{-1}}(-)$ is its inverse, thus ${}^h(-)$ is an autoequivalence. We further have $^{1_H}(-) = \mathrm{id}$ and $^{h_1h_2}(-) = {}^{h_1}(^{h_2}(-))$, so $\psi$ is indeed strictly monoidal.
The coherence conditions in~\eqref{action C'} are routine to verify. As an example, we check the lower half of~\eqref{action C3}, leaving the rest to the reader.
This follows from the commutativity of the following diagram in $\mathcal C$:
\begin{equation}
    \begin{tikzcd}[column sep = 2em, row sep = 1.5em]
	{A \otimes {}^hM} &&&& {^hM} \\
	& {A \otimes_A {}^hM} && {^hM} \\
	& {^hA \otimes_A {}^hM} && {^h(A \otimes_A M)} \\
	{^hA \otimes {}^hM} &&&& {^h(A \otimes M)}
	\arrow["{\rho_{^hM}}", from=1-1, to=1-5]
	\arrow["{\pi_{A,{}^hM}}"', from=1-1, to=2-2]
	\arrow["{\upsilon^{-1}_h \otimes \id}"', shift right=3, from=1-1, to=4-1]
	\arrow["{l_{^hM}}", from=2-2, to=2-4]
	\arrow["{(h_A) \otimes_A \id}"', from=2-2, to=3-2]
	\arrow[equal, from=2-4, to=1-5]
	\arrow["{h_{A,M}}"', from=3-2, to=3-4]
	\arrow["{^h(l_M)}"', from=3-4, to=2-4]
	\arrow["{\pi_{{}^hA,{}^hM}}", from=4-1, to=3-2]
	\arrow["\id"', from=4-1, to=4-5]
	\arrow["{^h\rho_M}"', from=4-5, to=1-5]
	\arrow["{^h(\pi_{A,M})}"', from=4-5, to=3-4]
\end{tikzcd}
\end{equation}
where the lower half of~\eqref{action C3} is the inner square.
Here $l$ denotes the left unitor in $_A\mathcal C$ as in~\eqref{left unitor}, and $\rho$ denotes the corresponding action morphism.

\medskip

(\romannumeral 2) Denote $g = |M|$.
Applying~\eqref{eq: naturality of action} and~\eqref{eq:alg-twisted-action-on-module} repeatedly,
we have
\begin{equation}
  \vcenter{\hbox{%
    \makebox[0pt][c]{%
      \hspace*{7em}%
      \clipbox{0pt 30pt 0pt 0pt}{%
        \includesvg[width=0.85\textwidth]{picture/refined_h-action_degree}%
      }%
    }%
  }}
\end{equation}
Note that $\upsilon_h\phi_g\upsilon_h^{-1} = \phi_{h \rhd g}$ by~\eqref{action compatibility}, hence we conclude that $^hM$ is $(h \rhd g)$-local.
\end{proof}

\subsection{$\chi$-Crossed braided structures on twisted local modules}

We now assume in addition that $\mathcal{C}$ is braided, and denote the braiding for any $U,V \in \mathcal{C}$ by
\begin{equation}
    C_{U,V}\colon U \otimes V \longrightarrow V \otimes U \ .
\end{equation}
We want to endow \({}_A\mathcal C^{\mathrm{gr}}\) with the structure of a \(\chi\)-crossed braided category.

\begin{theorem}\label{twisted local module crossed braiding}
Let $(G,H,\chi,\rhd)$ be a crossed module. Assume, in addition to Proposition~\ref{structure morphism for h-action in AMod}, that $\mathcal{C}$ is braided, that $\overline{\sigma}_U = C_{A,U}$ for all $U \in \mathcal C$, and that for every $g \in G$ and $M \in {}_A\mathcal{C}_g$, one has
\begin{equation}\label{crossing change'}
    \rho_M \circ C_{M,A} \circ C_{A,M}
    = \rho_M \circ (\upsilon_{\chi(g)}^{-1} \otimes \mathrm{id}_M) \ .
\end{equation}
Then ${}_A\mathcal{C}^{\mathrm{gr}}$ admits a $\chi$-crossed braiding.
\end{theorem}

Graphically, equation~\eqref{crossing change'} reads:
\begin{equation}
  \raisebox{-0.5\height}{
  \includesvg[width=0.3\textwidth]{picture/chi-crossed_braiding_condition}
}
\end{equation}

\begin{proof}
Let $M,N \in {}_{A}\mathcal{C}^\mathrm{gr}$ and suppose that $M$ is homogeneous of degree $g$. We then have the following diagrammatic equation:
\begin{equation*}
  \raisebox{-105pt}[90pt][0pt]{
  \includesvg[width=0.95\textwidth]{picture/crossed_braiding}
}
\end{equation*}
Here, the first equality follows from~\eqref{crossing change'}, the second from the naturality of the braiding, and the third from the definition of the relative tensor product $\otimes_A$ combined with the assumption $\overline{\sigma}_N = C_{A,N}$.
Therefore, by the universal property of a coequalizer, there exists a unique morphism 
    \begin{equation}
        \widetilde{C}_{M,N}\colon M \otimes_A N \to {^{\chi(g)}N} \otimes_A M \ 
    \end{equation}
    such that the following diagram commutes:
    \begin{equation}\label{eq: crossed braiding existence}
    \begin{tikzcd}[column sep = 4em, row sep = 1.5em]
M\otimes N \ar[r, "\pi_{M,N}"]\ar[d, swap, "C_{M,N}"]
&
M \otimes_A N \ar[d, dashed, "\exists ! \, \widetilde{C}_{M,N}"]\\
N\otimes M \ar[r, "\pi_{{}^{\chi_g}N, M}"]
&
{}^{\chi_g}N \otimes_A M
\end{tikzcd}
    \end{equation}
The following equation shows that $\pi_{{}^{\chi(g)}N,M} \circ C_{M,N}$ 
is an $A$-intertwiner:
\begin{equation}
  \vcenter{\hbox{%
    \includesvg[width=0.75\textwidth]{picture/crossed_braiding_A-linear}%
  }}
\end{equation}
where in the second equality we 
apply the definition of the relative tensor product $\otimes_A$ together with the assumption that $\overline{\sigma}_N = C_{A,N}$, and in the last equality we use $A$-linearity of $\pi_{{}^{\chi(g)}N,M}$.
Therefore, $\widetilde{C}_{M,N}$ is also $A$-linear. 
By similar arguments, one can show that $\widetilde{C}_{M,N}$ is invertible.

It is left to verify that the crossed braiding constructed above satisfies conditions~\eqref{B1}--\eqref{B4} in Definition~\ref{crossed braiding def}. Note that in our case we have a strict $H$-action $\psi$ as defined in~\eqref{eq:strict psi}, hence all the $T$ terms in~\eqref{B1}--\eqref{B4} become trivial. Condition~\eqref{B1} follows directly by the naturality of braiding in $\mathcal{C}$. Condition~\eqref{B4} follows by straightforward computation using~\eqref{eq:h-action monoidal} and~\eqref{eq: crossed braiding existence}. Here we only check~\eqref{B2};~\eqref{B3} holds by similar arguments. Let $L,M,N \in {}_{A}\mathcal{C}^\mathrm{gr}$ be homogeneous objects with $|L| = g, |M| = k$. The condition~\eqref{B2} is equivalent to the commutativity of the following diagram (recall that $\chi_g = \chi(g)$):
\begin{equation}
\begin{tikzcd}[column sep = 4em, row sep = 1.5em]
L \otimes_A (M \otimes_A N)\ar[r, "\mathrm{id}_L \otimes_A \widetilde{C}_{M,N}"]\ar[d, "\alpha_{L,M,N}"]
&
L \otimes_A ({}^{\chi_k}N \otimes_A M) \ar[d, swap, "\alpha_{L,{}^{\chi_k}N,M}"]\\
(L \otimes_A M) \otimes_A N\ar[d,"\widetilde{C}_{L \otimes_A M,N}"]
&
(L \otimes_A {}^{\chi_k}N) \otimes_A M\ar[d, swap, "\widetilde{C}_{L,{}^{\chi_k}N}\otimes_A \mathrm{id}_M"]\\
{}^{\chi_{gk}}N \otimes_A (L \otimes_A M)\ar[r,swap,"\alpha_{{}^{\chi_{gk}}N,L,M}"]
&
({}^{\chi_{gk}}N \otimes_A L) \otimes_A M
\end{tikzcd}
\end{equation}
which is induced by the commutativity of the following diagram:
\begin{equation*}
\small
\begin{tikzcd}[column sep = 3.8em, row sep = 1.8em]
L \otimes M \otimes N\ar[rrr, "\mathrm{id}_L\otimes C_{M,N}"]\ar[dddd, "C_{L\otimes M,N}"]\ar[rd,"\scalebox{1.4}{$\pi$}_{L,M \otimes_A N}\circ (\mathrm{id}_L \otimes \scalebox{1.4}{$\pi$}_{M,N})"]
&
{}
&
{}
&
L \otimes {}^{\chi_k}N \otimes M\ar[dddd,swap,"C_{L,N}\otimes \mathrm{id}_M"]\ar[ld,swap,"\scalebox{1.4}{$\pi$}_{L,{^{\chi_k}N} \otimes_A M}\circ (\mathrm{id}_L \otimes \scalebox{1.4}{$\pi$}_{{^{\chi_k}N},M})"]\\
{}
&
L \otimes_A (M \otimes_A N)\ar[r, "\mathrm{id}_L \otimes_A \widetilde{C}_{M,N}"]\ar[d,swap, "\alpha_{L,M,N}"]
&
L \otimes_A ({}^{\chi_k}N \otimes_A M) \ar[d, "\alpha_{L,{}^{\chi_k}N,M}"]\\
{}
&
(L \otimes_A M) \otimes_A N\ar[d,swap,"\widetilde{C}_{L \otimes_A M,N}"]
&
(L \otimes_A {}^{\chi_k}N) \otimes_A M\ar[d, "\widetilde{C}_{L,{}^{\chi_k}N}\otimes_A \mathrm{id}_M"]\\
{}
&
{}^{\chi_{gk}}N \otimes_A (L \otimes_A M)\ar[r,swap,"\alpha_{{}^{\chi_{gk}}N,L,M}"]
&
({}^{\chi_{gk}}N \otimes_A L) \otimes_A M\\
{}^{\chi_{gk}}N \otimes L \otimes M\ar[rrr,swap, "="]\ar[ru,swap,"\scalebox{1.4}{$\pi$}_{{^{\chi_{gk}}N},L \otimes_A M}\circ (\mathrm{id}_{^{\chi_{gk}}N} \otimes \scalebox{1.4}{$\pi$}_{L,M})"]
&
{}
&
{}
&
{}^{\chi_{gk}}N \otimes L \otimes M\ar[lu,"\scalebox{1.4}{$\pi$}_{{^{\chi_{gk}}N}\otimes_A L, M}\circ (\scalebox{1.4}{$\pi$}_{{^{\chi_{gk}}N},L}\otimes \mathrm{id}_M)"]
\end{tikzcd}
\end{equation*}
due to the fact that $\pi_{L,M \otimes_A N}\circ (\mathrm{id}_L \otimes \pi_{M,N})$ is an epimorphism. To see that the second diagram indeed commutes, we need to apply~\eqref{associator Amod}.
\end{proof}

\begin{proposition}\label{twisted local module crossed balanced}   
In addition to the conditions in Theorem~\ref{twisted local module crossed braiding}, suppose that the category $\mathcal{C}$ is balanced with the balancing isomorphism denoted by $\theta$, and that $\theta_A = \mathrm{id}_A$.
Then $_{A}\mathcal{C}^\mathrm{gr}$ is $\chi$-crossed balanced.
\end{proposition}

\begin{proof}
Since $\mathcal{C}$ is balanced, by the naturality of $\theta$ we have the following commutative diagram in $\mathcal{C}$:
\begin{equation}
\begin{tikzcd}[column sep = 4em, row sep = 1.5em]
A\otimes M \ar[d, "\rho_M"]\ar[r, "\theta_{A\otimes M}"]
&
A \otimes M \ar[d, swap, "\rho_M"]\\
M \ar[r, "\theta_M"]
&
M
\end{tikzcd}
\end{equation}
Suppose now that $M\in {}_A\mathcal{C}_g$.
Since $\theta_{A\otimes M} = C_{M,A} \circ C_{A,M} \circ (\theta_{A} \otimes \theta_M)$, and using our assumption~\eqref{crossing change'}, we can redraw the diagram above as
\begin{equation}\label{theta_M A-intertwiner}
\begin{tikzcd}[column sep = 4em, row sep = 1.5em]
A\otimes M \ar[d, "\rho_M"]\ar[r, "\mathrm{id}_A \otimes \theta_{M}"]
&
A \otimes {^{\chi_g}M} \ar[d, swap, "\rho_{^{\chi_g}M}"]
\\
M \ar[r, "\theta_M"]
&
^{\chi_g}M
\end{tikzcd}
\end{equation}
We define an isomorphism
\begin{equation}\label{eq: crossed twist}
    \widetilde{\theta}_M:M \longrightarrow 
    {^{\chi_g}M} 
    \qquad \textup{in }{}_A\mathcal C
\end{equation}
whose underlying morphism in $\mathcal C$ is just $\theta_M$.
Note that due to the cross-module condition~\eqref{crossed module} we have $^{\chi_g}M \in {}_A\mathcal C_g$, and~\eqref{theta_M A-intertwiner} shows that $\widetilde{\theta}_M$ is indeed an $A$-intertwiner. 

We now check that $\widetilde{\theta}$ satisfies all the conditions in Definition~\ref{crossed balanced}. Conditions~\eqref{BL1} and~\eqref{BL3} are immediate because the $H$-action functor $\psi$ from~\eqref{eq:strict psi} is strict.
We now check~\eqref{BL2}. For any $g,k \in G$ and any $M \in {}_A\mathcal{C}_g$, $N \in {}_A\mathcal{C}_k$, by~\eqref{morphism tensor product} we have the unique morphism $\widetilde{\vartheta}_M \otimes_A \widetilde{\vartheta}_N$ such that the following diagram commutes:
\begin{equation}\label{P415.1}
\begin{tikzcd}[column sep = 4em, row sep = 1.5em]
M\otimes N \ar[r, "\pi_{M,N}"]\ar[d, swap, "\vartheta_M \otimes \vartheta_N"]
&
M \otimes_A N \ar[d, dashed, "\widetilde{\vartheta}_M \otimes_A \widetilde{\vartheta}_N"]\\
M\otimes N \ar[r, "\pi_{{}^{\chi_g}M, {}^{\chi_k}N}"]
&
{}^{\chi_g}M \otimes_A {}^{\chi_k}N
\end{tikzcd}
\end{equation}
and by~\eqref{eq: crossed braiding existence} we have the unique morphisms $\widetilde{C}_{^{\chi_{gk}}N, {^{\chi_g}M}}$ and $\widetilde{C}_{^{\chi_g}M,{^{\chi_k}N}}$ such that the following digram commutes:
\begin{equation}\label{P415.2}
\begin{tikzcd}[column sep = 4em, row sep = 1.5em]
M \otimes N\ar[r, "\pi_{{}^{\chi_g}M, {}^{\chi_k}N}"]\ar[d,swap, "C_{M,N}"]
&
{}^{\chi_g}M \otimes_A {}^{\chi_k}N \ar[d, dashed, "\widetilde{C}_{^{\chi_g}M,{^{\chi_k}N}}"]\\
N \otimes M\ar[r,"\pi_{{}^{\chi_{gk}}N,{}^{\chi_g}M}"]\ar[d,swap,"C_{N,M}"]
&
{}^{\chi_{gk}}N \otimes_A {}^{\chi_g}M\ar[d, dashed, "\widetilde{C}_{^{\chi_{gk}}N, {^{\chi_g}M}}"]\\
M \otimes N\ar[r,"\pi_{{}^{\chi_{gk}}M,{}^{\chi_{gk}}N}"]
&
{}^{\chi_{gk}}M \otimes_A {}^{\chi_{gk}}N
\end{tikzcd}
\end{equation}
Here we use the fact that the functor $\psi$ as defined in~\eqref{eq:strict psi} is strictly monoidal, and that $^{\chi_{gk}}N \in {}_A\mathcal C_{gkg^{-1}}$.
Composing~\eqref{P415.1} and~\eqref{P415.2} vertically, we obtain on the left hand side $\theta_{M\otimes N}$ and on the right hand side
\begin{equation}
    \widetilde{C}_{^{\chi_{gk}}N, {^{\chi_g}M}} \circ \widetilde{C}_{^{\chi_g}M,{^{\chi_k}N}} \circ (\widetilde{\vartheta}_M \otimes_A \widetilde{\vartheta}_N) \ .
\end{equation}
We further have the following commutative diagram:
\begin{equation}
\begin{tikzcd}[column sep = 4em, row sep = 1.5em]
M\otimes N \ar[r, "\pi_{M,N}"]\ar[d, "\vartheta_{M \otimes N}"']
&
M \otimes_A N \ar[d, swap, dashed, "\widetilde{\vartheta}_{M \otimes_A N}"']\\
^{\chi_{gk}}(M\otimes N) \ar[r, "^{\chi_{gk}}\pi_{M,N}"]\ar[d, "\mathrm{id}"']
&
{}^{\chi_{gk}}(M \otimes_A N)\ar[d, dashed, "{\chi_{gk}}^{-1}_{M,N}"]\\
M\otimes N \ar[r, "\pi_{{}^{\chi_{gk}}M, {}^{\chi_{gk}}N}"]
&
{}^{\chi_{gk}}M \otimes_A {}^{\chi_{gk}}N
\end{tikzcd}
\end{equation}
where the upper rectangle commutes following~\eqref{BL1}; recall that $\tilde{\theta}_{M \otimes_A N}=\theta_{M \otimes_A N}$ as a morphism in $\mathcal{C}$,
and the lower rectangle commutes by~\eqref{eq:h-action monoidal}.
Note that in the left column we wrote $^{\chi_{gk}}(M\otimes N)$, but as objects in $\mathcal{C}$ we actually have $^{\chi_{gk}}(M\otimes N) = M\otimes N$. Writing it this way makes clear why the horizontal arrow makes sense.
Consequently, we have
\begin{equation}
    (\chi_{gk})_{M,N} \circ \widetilde{C}_{^{\chi_{gk}}N, {^{\chi_g}M}} \circ \widetilde{C}_{^{\chi_g}M,{^{\chi_k}N}} \circ (\widetilde{\vartheta}_M \otimes_A \widetilde{\vartheta}_N) = \widetilde{\vartheta}_{M \otimes_A N} \ ,
\end{equation}
which is exactly~\eqref{BL2} when the structure isomorphisms $T^{g,k}_M$ of $\psi$ are trivial.
\end{proof}

\subsection{Examples from the superalgebra $\mathbb{K}[x]/(x^2)$}\label{superalgebra example}

Let $\mathcal C$ be the category of finite-dimensional
super vector spaces over an algebraically closed field $\mathbb K$ of characteristic 0. An object $V \in \mathcal C$ can be written as $V = V_0 \oplus V_1$, where $V_0$ and $V_1$ are the degree-$0$ and degree-$1$ components, respectively. The symmetric braiding on $\mathcal C$ is given by, for all $U, V \in \mathcal C$,
    \begin{equation}
        C_{U,V}\colon U \otimes V \longrightarrow V \otimes U \ , \qquad
        u_i \otimes v_j \longmapsto (-1)^{ij}\, v_j \otimes u_i \ ,       
    \end{equation}
where $i, j \in \{0,1\}$ denote the degrees.    
Fix the algebra object $A = \mathbb{K}[x]/(x^2)$
with the unit $\mathbf 1$ in degree 0 and $x$ in degree 1. Fix the half-braiding $\overline{\sigma}_U = C_{A,U}$ for all $U \in \mathcal C$. Let $\chi = \mathrm{id}\colon \mathbb{K}^\times \to \mathbb{K}^\times$ with trivial action. For every $n \in \Z$, define
\begin{equation}
    \phi_n\colon \mathbb K^\times \longrightarrow \mathrm{Aut}(A) \ , \qquad
    \alpha \longmapsto
    \begin{cases}
        \mathbf{1} \longmapsto \mathbf{1} \ ,\\
        x \longmapsto \alpha^n\,x \ .
    \end{cases}
\end{equation}
We see in the following what happens when we make different choices for $\widetilde{\sigma}$.

\medskip

\noindent
$\bullet$ \textit{Case (\romannumeral 1):} Choose $\widetilde{\sigma}_U = (C_{U,A})^{-1} = C_{A,U}$ for all $U \in \mathcal C$. For a fixed $n \in \Z$,
condition~\eqref{eq:twist-local-condition} in this case becomes: an $A$-module $M$ is $\alpha$-local if and only if
\begin{equation}
    \rho_M \circ \big((\mathrm{id}_A-\phi_n(\alpha)) \otimes \mathrm{id}_M\big) = 0 \ .
\end{equation}
We then obtain twisted localness regarding $\big((A, \overline{\sigma}, \widetilde{\sigma}), \phi_n\big)$:
\begin{itemize}
    \item if $\alpha \in \mathbb{K}^\times$ such that $\alpha ^n = 1$, then all $A$-modules are $\alpha$-local;

    \item for other $\alpha \in \mathbb K^\times$,
the category ${}_A\mathcal C_\alpha$ is semisimple with simple objects given by the two simple $A$-modules in ${}_A\mathcal C$, namely the
one-dimensional modules of degrees $0$ and $1$ on which
$x$ acts by $0$.
\end{itemize}

For any $m \in \Z$ with $n \mid m$, we can then define
\begin{equation}\label{eq: upsilon example}
    \upsilon\colon \mathbb K^\times \longrightarrow \mathrm{Aut}(A) \ , \qquad
    \alpha \longmapsto
    \begin{cases}
        \mathbf{1} \longmapsto \mathbf{1} \ ,\\
        x \longmapsto \alpha^m\,x \ .
    \end{cases}
\end{equation}
For any $\alpha$-local module $M$, condition~\eqref{crossing change'} becomes
\begin{equation}
    \rho_M \circ \big((\mathrm{id}_A-\upsilon^{-1}(\alpha)) \otimes \mathrm{id}_M\big) = 0 \ .
\end{equation}
This condition indeed holds for all $\alpha \in \mathbb{K}^\times$: for $\alpha^n = 1$, the assumption $n \mid m$ forces $\alpha^{-m} = 1$; for $\alpha^n \neq 1$, the module $M$ itself is annihilated by $x$. Therefore, the category ${}_A\mathcal C^{\mathrm{gr}}=\bigoplus_{\alpha\in \mathbb{K}^{\times}}{}_A\mathcal C_\alpha$ forms a $\mathbb K^\times$-crossed braided category with exactly $n$ non-semisimple components, while all remaining components are semisimple.

\medskip

\noindent
$\bullet$ \textit{Case (\romannumeral 2):}
Choose $\widetilde{\sigma}$ to be the flip map $\tau$. For a fixed $n \in \Z$,
a module $M$ is $\alpha$-local if and only if
\begin{equation}
    \rho_M = \rho_M \circ \tau \circ C_{A,M} \circ \big(\phi_n(\alpha) \otimes \mathrm{id}_M\big) \ .
\end{equation}
We then conclude:
\begin{itemize}
    \item for $\alpha \in \mathbb{K}^\times$ such that $\alpha ^n = 1$, $M = M_0 \oplus M_1$ is $\alpha$-local if and only if $x$ acts on $M_1$ by~0;

    \item for $\alpha \in \mathbb{K}^\times$ such that $\alpha ^n = -1$, $M = M_0 \oplus M_1$ is $\alpha$-local if and only if $x$ acts on $M_0$ by 0;

    \item for other $\alpha \in \mathbb K^\times$,
the only $\alpha$-local modules are direct sums of the two simple $A$-modules.
\end{itemize}

For any $m \in \Z$ with $2n \mid m$, we define $\upsilon\colon \mathbb K^\times \to \mathrm{Aut}(A)$ as in~\eqref{eq: upsilon example}.
Similarly, one can check that condition~\eqref{crossing change'} is satisfied and hence ${}_A\mathcal C^{\mathrm{gr}}$ is a
$\mathbb K^\times$-crossed braided category.

\medskip

In both cases, the explicit formula for the $\mathbb K^\times$-crossed braiding for any $M\in {}_A\mathcal C_\alpha$ and $N\in {}_A\mathcal C$ is given as follows:
\begin{align}
\begin{split}
    \widetilde{C}_{M,N}\colon M \otimes_A N \longrightarrow {}^\alpha N \otimes_A M \ , \qquad
    v_i \otimes_A w_j \longmapsto (-1)^{ij}\, w_j \otimes_A v_i \ .
\end{split}
\end{align}
Although the crossed braiding does not depend on (allowed choices of) $m$,
the data defining the $\mathbb K^\times$-crossed $\mathbb K^\times$-graded monoidal structure (specifically, the $\mathbb K^\times$-action)
does depend on $m$.

\begin{remark}
    From these two examples, we see that the freedom in choosing $\widetilde{\sigma}$ can lead to interesting phenomena. Both cases produce $\mathbb K^\times$-graded categories, but with different non-semisimple fibres.
\end{remark}

\subsection{$\chi$-Crossed ribbon structures on twisted local modules}\label{sec: crossed ribbon twisted mod}

Recall from Example~\ref{ex: non-algebra automorphism} that even when $\mathcal{C}$ is rigid, the category $_A\mathcal{C}$ for a general central algebra $A \in \mathcal{C}$ is not necessarily rigid. However, if $\mathcal C$ is rigid and $A$ is separable, we show in Proposition~\ref{prop: G-graded rigid twisted mod} that $_A\mathcal{C}^{\mathrm{gr}}$ is a rigid $G$-graded monoidal category. Furthermore, if in addition to the conditions of Proposition~\ref{twisted local module crossed balanced}, $\mathcal{C}$ is ribbon and $A$ is $\Delta$-separable Frobenius, Theorem~\ref{twisted local module crossed ribbon} shows that $_A\mathcal{C}^{\mathrm{gr}}$ is $\chi$-crossed ribbon.

\medskip

An algebra $A \in \mathcal{C}$ is called \textit{separable} if it admits an $A$-bimodule
morphism $S \colon A\to A\otimes A$, which is a section of the multiplication. Explicitly, this means
\begin{subequations}\label{eq: separable algebra}
    \begin{gather}
        (\mu\otimes \id_A)\circ(\id_A\otimes S)
=
    S\circ \mu
=
(\id_A\otimes \mu)\circ(S \otimes \id_A) \label{eq:frobenius condition}\ , \\
\mu\circ S=\id_A \label{eq: separability}\ .
    \end{gather}
\end{subequations}
Graphically, we depict $S\colon A \to A \otimes A$ by:
\begin{equation}
    \begin{tikzpicture}[baseline=0pt,font=\scriptsize,
  strand/.style={draw, line width=0.7pt}
]
\node[eq, scale=0.8] at (-1.0,0) {$S \, \, \, = $};
\begin{scope}[xshift=0.3cm, scale=0.8]
  \node at (0,-1.0) {$A$};
  \node at (-0.6,0.8) {$A$};
  \node at (0.6,0.8) {$A$};
  \draw[strand] (0,-0.85) -- (0,-0.2);
  \draw[strand] (0,-0.2) to[out=0,in=-90] (0.6,0.65);
  \draw[strand] (0,-0.2) to[out=180,in=-90]   (-0.6,0.65);
\end{scope}
\end{tikzpicture}
\end{equation}
Then~\eqref{eq:frobenius condition} and~\eqref{eq: separability} become:
\begin{subequations}\label{eq:diagram-relations}
\begin{gather}
\begin{tikzpicture}[baseline=0pt,font=\scriptsize,
  strand/.style={draw, line width=0.7pt},
  eq/.style={inner sep=0pt}
]
\begin{scope}[xshift=0cm, scale=0.5]
 \node at (-0.8,-2.4) {$A$};
 \node at (1.6,-2.4) {$A$};
 \node at (0,0.8) {$A$};
 \node at (2.4,0.8) {$A$};
 \draw[strand] (-0.8,-0.8) to[out=90,in=180] (0,0);
 \draw[strand] (0,0) to[out=0,in=90] (0.8,-0.8);
 \draw[strand] (0.8,-0.8) to[out=-90,in=180] (1.6,-1.6);
 \draw[strand] (1.6,-1.6) to[out=0,in=-90] (2.4,-0.8);
 \draw[strand] (-0.8,-0.8) -- (-0.8,-2.15);
 \draw[strand] (2.4,-0.8) -- (2.4,0.55);
 \draw[strand] (1.6,-1.6) -- (1.6,-2.15);
 \draw[strand] (0,0) -- (0,0.55);
\end{scope}
\node[eq] at (2.3,-0.5) {$=$};
\begin{scope}[xshift=3cm, scale=0.5]
 \node at (0,-2.4) {$A$};
 \node at (2,-2.4) {$A$};
 \node at (0,0.8) {$A$};
 \node at (2,0.8) {$A$};
 \draw[strand] (1,-0.4) -- (1,-1.2);
 \draw[strand] (1,-0.4) to[out=0,in=-90] (2,0.55);
 \draw[strand] (1,-0.4) to[out=180,in=-90] (0,0.55);
 \draw[strand] (1,-1.2) to[out=0,in=90] (2,-2.15);
 \draw[strand] (1,-1.2) to[out=180,in=90] (0,-2.15);
\end{scope}
\node[eq] at (4.8,-0.5) {$=$};
\begin{scope}[xshift=6.1cm, scale=0.5]
 \node at (-0.8,0.8) {$A$};
 \node at (0,-2.4) {$A$};
 \node at (1.6,0.8) {$A$};
 \node at (2.4,-2.4) {$A$};
 \draw[strand] (-0.8,0.55) -- (-0.8,-0.8);
 \draw[strand] (-0.8,-0.8) to[out=-90,in=180] (0,-1.6);
 \draw[strand] (0,-1.6) to[out=0,in=-90] (0.8,-0.8);
 \draw[strand] (0.8,-0.8) to[out=90,in=180] (1.6,0);
 \draw[strand] (1.6,0) to[out=0,in=90] (2.4,-0.8);
 \draw[strand] (1.6,0) -- (1.6,0.55);
 \draw[strand] (0,-1.6) -- (0,-2.15);
 \draw[strand] (2.4,-0.8) -- (2.4,-2.15);
\end{scope}
\end{tikzpicture}
\label{eq:diagram-relations-a}
\\
\begin{tikzpicture}[baseline=0pt,font=\scriptsize,
  strand/.style={draw, line width=0.7pt}
]
\begin{scope}[xshift=0cm, scale=0.6]
  \node at (0,-1.0) {$A$};
  \node at (0,1.8) {$A$};
  \draw[strand] (0,-0.75) -- (0,-0.2);
  \draw[strand] (0,-0.2) to[out=0,in=-90] (0.6,0.4);
  \draw[strand] (0,-0.2) to[out=180,in=-90] (-0.6,0.4);
  \draw[strand] (-0.6,0.4) to[out=90,in=180] (0,1);
  \draw[strand] (0.6,0.4) to[out=90,in=0] (0,1);
  \draw[strand] (0,1) -- (0,1.55);
\end{scope}
\node at (1.0,0.3) {$=$};
\begin{scope}[xshift=1.8cm, scale=0.6]
  \node at (0,-1.0) {$A$};
  \node at (0,1.8) {$A$};
  \draw[strand] (0,-0.75) -- (0,1.55);
\end{scope}
\end{tikzpicture}
\label{eq:diagram-relations-b}
\end{gather}
\end{subequations}

\medskip

Suppose that $\mathcal C$ is rigid, and let $(A, \overline{\sigma},\widetilde{\sigma})$ be a doubly central algebra in $\mathcal C$. Consider the induction functor
\begin{equation}
    A \otimes - \colon \mathcal C \longrightarrow {}_A\mathcal C \ , \qquad X \mapsto A \otimes X \ , \qquad f \mapsto \id_A \otimes f \ ,
\end{equation}
where the $A$-action on the free module $A \otimes M$ is given by algebra multiplication. Recall that the monoidal structure on $_A\mathcal C$ is defined via~\eqref{eq: tensor over A}. The induction functor is monoidal \cite[Prop.\,5.11]{FS01}.
For any $X, Y \in \mathcal C$ the monoidal structure morphism
\begin{equation}
    \mathrm{Ind}_{X,Y} \colon (A \otimes X)\otimes_A(A \otimes Y) \to A \otimes X \otimes Y
\end{equation}
is induced by:
\begin{equation}
    \begin{tikzcd}
	{(A \otimes X)\otimes(A \otimes Y)} && {(A \otimes X)\otimes_A(A \otimes Y)} \\
	&& {A \otimes X \otimes Y}
	\arrow["{\pi_{A \otimes X, A \otimes Y}}", from=1-1, to=1-3]
	\arrow["{(\mu \otimes \id_{X \otimes Y})(\id_A \otimes \overline{\sigma}_X^{-1}\otimes \id_Y)}"', from=1-1, to=2-3]
	\arrow["{\mathrm{Ind}_{X,Y}}", dashed, from=1-3, to=2-3]
\end{tikzcd}
\end{equation}
with the inverse given by
\begin{equation}
    A \otimes X \otimes Y \xrightarrow{\id_{A \otimes X} \otimes \eta \otimes \id_Y} A\otimes X \otimes A \otimes Y \xrightarrow{\pi_{A \otimes X, A \otimes Y}} (A\otimes X) \otimes_A (A \otimes Y) \ .
\end{equation}

Since $\mathcal{C}$ is rigid and monoidal functors preserve duality, every free module $A \otimes X$ is dualisable in ${}_A\mathcal C$
with dual object $A \otimes X^*$, where $X^*$ is the dual of $X$ in $\mathcal{C}$. 
The evaluation and coevaluation morphisms are given by
\begin{subequations}
\begin{align}
\mathrm{ev}^A_{A \otimes X} = (\id_A \otimes \mathrm{ev}_X) \circ \mathrm{Ind}_{X^*,X}
&\colon
(A \otimes X^*) \otimes_A (A \otimes X) \longrightarrow A \ , \\
\mathrm{coev}^A_{A \otimes X} = \mathrm{Ind}^{-1}_{X,X^*}\circ(\id_A \otimes \mathrm{coev}_X)
&\colon
A \longrightarrow (A \otimes X) \otimes_A (A \otimes X^*) \ .
\end{align}
\end{subequations}
Here, we write $\mathrm{ev}^A$ and $\mathrm{coev}^A$ for the $A$-linear evaluation and coevaluation morphisms in $_A\mathcal{C}$, and $\mathrm{ev}_X$ and $\mathrm{coev}_X$ for the corresponding morphisms in $\mathcal{C}$.

Now suppose that $A$ is separable.
Every $A$-module $M$ is then a retract of the free module $A \otimes M$: the composition
\begin{equation}\label{eq: free retract}
    M \xrightarrow{\big(\id_A \otimes \rho_M\big) \circ \big((S \circ \eta) \otimes \id_M\big)} A \otimes M \xrightarrow{\rho_M} M
\end{equation}
is identity on $M$,
where $\rho_M$ is the action morphism, $\eta$ is the unit of $A$, and we use~\eqref{eq:diagram-relations-b}.
It follows that $M$ admits a dual in $_A\mathcal{C}$, which we denote by $M^\vee$ to distinguish it from the dual $M^*$ of its underlying object in $\mathcal{C}$. Indeed, in a general monoidal category, if an object $Y$ is a retract of a dualisable object $Z$ via morphisms $Y \xrightarrow{i} Z \xrightarrow{p} Y$,
then the morphism $(i \circ p)^* \colon Z^* \to Z^*$ is an idempotent. If the category is idempotent complete, then $(i \circ p)^*$ splits, factoring through an object $\widetilde Y$ which serves as the dual of $Y$. By Proposition~\ref{monoidal structure on A-mod}, we know that $_A\mathcal{C}$ is abelian and hence idempotent complete. Therefore, we conclude that $_A\mathcal{C}$ is rigid.

\begin{remark}\label{rem:sep-alg}
    If $\mathcal{C}$ is semisimple and $A$ is separable, then ${}_A\mathcal C$ is semisimple~\cite[Prop.\,7.8.30]{EGNO15}.
\end{remark}

\begin{proposition}\label{prop: G-graded rigid twisted mod}
Let $\mathcal{C}$ be an abelian rigid monoidal category with a right exact tensor product, and let $(A, \overline{\sigma}, \widetilde{\sigma})$ be a doubly central algebra in $\mathcal{C}$ whose underlying algebra is separable. Then the category $_A\mathcal{C}^{\mathrm{gr}}$ defined in~\eqref{eq: _AC graded} is a rigid $G$-graded monoidal category.
\end{proposition}

\begin{proof}
We already know that $_A\mathcal{C}$ is rigid. It remains to check that if $M \in {}_A\mathcal{C}$ is $g$-local, then its dual $M^\vee$ is $g^{-1}$-local. 
From the duality morphisms we get the isomorphism (here the $A$-action on $M^\vee \otimes A$ is given by the action on the first tensor factor)
\begin{equation}
    \Hom_{_A\mathcal C}(M^\vee \otimes A, M^\vee) \cong \Hom_{_A\mathcal C}((M^\vee \otimes A) \otimes_A M , A) \ ,
\end{equation}
and it thus suffices to establish the following identities (we draw $\mathrm{ev}^A_M$ as a bead in the following diagrams):
\begin{equation}
  \vcenter{\hbox{%
    \makebox[0pt][c]{%
      \hspace*{10em}%
      \includesvg[width=0.95\textwidth]{picture/dual_degree}%
    }%
  }}
\end{equation}
Here the first step applies~\eqref{eq: tensor over A}, the second uses $g$-localness of $M$ in~\eqref{eq:g-local-2},
the third applies~\eqref{eq: tensor over A} again, the fourth uses the $A$-linearity of $\pi_{M^\vee,M}$ and $\mathrm{ev}^A_M$ as well as naturality of the half-braiding $\overline{\sigma}$, the fifth applies the central structure of $(A, \overline \sigma)$, the sixth is the second Reidemeister move, the seventh uses naturality of the half-braiding $\widetilde{\sigma}$, and the eighth applies the central structure of $(A, \widetilde \sigma)$ together with the $A$-linearity of $\pi_{M^\vee,M}$ and $\mathrm{ev}^A_M$.
\end{proof}

A \textit{Frobenius algebra} in $\mathcal{C}$ is a tuple $(A, \mu, \eta, \Delta, \epsilon)$ consisting of an algebra $(A, \mu, \eta)$ and a coalgebra $(A, \Delta, \epsilon)$ with coproduct and counit
\begin{equation}\Delta \colon A \to A \otimes A \ , \qquad \epsilon \colon A \to \mathbbm{1} \ ,
\end{equation}
subject to the usual coassociativity and counitality conditions, such that $\Delta$ satisfies condition~\eqref{eq:frobenius condition} (with $S$ replaced by $\Delta$). If \eqref{eq: separability} also holds for the coproduct, then $A$ is called \textit{$\Delta$-separable}.

\begin{remark}
    The relation between $\Delta$-separability and separability for Frobenius algebras in general can be found in \cite[Sec.\,3.2]{Mul22}. However, let us stress that in general the section of a separable algebra need not be coassociative or counital.
\end{remark}

\begin{theorem}\label{twisted local module crossed ribbon}   
Let $\mathcal{C}$ be an abelian ribbon category with braiding $C$ and balancing isomorphism~$\theta$, and let $(G,H,\chi,\rhd)$ be a crossed module. Suppose furthermore:
\begin{itemize}
    \item $(A,\overline{\sigma},\widetilde{\sigma})$ is a doubly central algebra in $\mathcal{C}$ with $\theta_A = \mathrm{id}_A$, and such that $A$ is $\Delta$-separable Frobenius. For all $U \in \mathcal C$, we have $\overline{\sigma}_U = C_{A,U}$.
    \item 
    $\phi$ and $\upsilon$ are $G$- and $H$-actions on $(A,\overline{\sigma},\widetilde{\sigma})$, respectively, satisfying~\eqref{action compatibility} and~\eqref{upsilon alegbra iso}. For every $g \in G$ and $M \in {}_A\mathcal{C}_g$, we have $\rho_M \circ C_{M,A} \circ C_{A,M}
    = \rho_M \circ (\upsilon_{\chi(g)}^{-1} \otimes \mathrm{id}_M)$.
\end{itemize}
Then the category $_A\mathcal C^{\mathrm{gr}}$ is $\chi$-crossed ribbon.
\end{theorem}

\begin{proof}
    We just need to verify condition~\eqref{RB1}, which in our case simplifies to, for any $M \in {}_A\mathcal C_g$:
    \begin{equation}
    \begin{tikzcd}[column sep = 2em, row sep = 1.5em]
	{^{\chi_g}(M^\vee)} && {M^\vee} \\
	{(^{\chi_g}M)^\vee}
	\arrow["{^{\chi_g}\widetilde{\theta}_{M^\vee}}", from=1-1, to=1-3]
	\arrow["{\nu_{\chi_g,M}}"', from=1-1, to=2-1]
	\arrow["{(\widetilde{\theta}_M)^\vee}"', from=2-1, to=1-3]
\end{tikzcd}
\end{equation}
where $\nu_{\chi_g,M}$ is as defined in Remark~\ref{rem:canoncial-iso-duals}.
It suffices to check the commutativity of the above diagram in the underlying category $\mathcal C$. By our definition of the $\chi_g$-action below Proposition~\ref{structure morphism for h-action in AMod} and the definition of $\widetilde \theta$ in~\eqref{eq: crossed twist}, the underlying morphisms $^{\chi_g}\widetilde{\theta}_{M^\vee}$ and $(\widetilde{\theta}_M)^\vee$ are identical in~$\mathcal C$ by the usual ribbon condition in $\mathcal C$. It is left to check that $\nu_{\chi_g,M}$ 
is identity in $\mathcal C$, which we explain in the following.

Following~\cite[Prop.\,5.15]{FS01},
when $A$ is $\Delta$-separable Frobenius and commutative,
the underlying object of $M^\vee$ is just given by $M^*$, and the $A$-action on $M^*$ is defined by
\begin{equation}
\begin{split}
    A \otimes M^* &\xrightarrow{\id_{A \otimes M^*} \otimes \mathrm{coev}_M} A \otimes M^* \otimes M \otimes M^* \xrightarrow{C_{A,M^*} \otimes \id_{M \otimes M^*}}  M^* \otimes A \otimes M \otimes M^* \\
    &\xrightarrow{\id_{M^*} \otimes \rho_M \otimes \id_{M^*}} M^* \otimes M \otimes M^* \xrightarrow{\mathrm{ev}_M \otimes \id_{M^*}} M^*\ ,
\end{split}
\end{equation}
where $C_{A,M}$ is the braiding in $\mathcal C$.
Define $\widehat{\mathrm{ev}}_M \colon M^* \otimes M \to A$ by
\begin{equation}
\begin{split}
M^* \otimes M &\xrightarrow{(\Delta\circ\eta)\otimes\id_{M^* \otimes M}}A \otimes A \otimes M^* \otimes M \xrightarrow{\id_A \otimes C_{A,M^*} \otimes \id_M} A \otimes M^* \otimes A \otimes M \\
&\xrightarrow{\id_{A \otimes M^*}\otimes\rho_M} A \otimes M^* \otimes M \xrightarrow{\id_A \otimes \mathrm{ev}_M} A
\end{split}
\end{equation}
and $\widehat{\mathrm{coev}}_M \colon A \to M \otimes M^*$ by
\begin{equation}
        A = A \otimes \mathbbm1 \xrightarrow{\id_A \otimes \mathrm{coev}_M} A \otimes M \otimes M^* \xrightarrow{\rho_M \otimes \id_{M^*}} M \otimes M^* \ .
\end{equation}
The evaluation $\mathrm{ev}^A_M \colon M^* \otimes_A M \to A$ and coevaluation $\mathrm{coev}^A_M \colon A \to M \otimes_A M^*$ are 
defined by
\begin{equation}
    \mathrm{ev}^A_M \circ \pi_{M^*,M} = \widehat{\mathrm{ev}}_M \ , \qquad \mathrm{coev}^A_M = \pi_{M,M^*} \circ \widehat{\mathrm{coev}}_M \ .
\end{equation}
Recall the $H$-action $\upsilon$ on the algebra $A$ as in~\eqref{eq:H-action-conditions-for-crossed-action}, the associator $\alpha$~\eqref{associator Amod} and the left and right unitors $l$~\eqref{left unitor} and $r$~\eqref{right unitor} in $_A\mathcal{C}$, as well as the structure morphisms of the $\chi_g$-action given in~\eqref{eq:h-action monoidal} and~\eqref{eq: h-action unitor mod}.
We transform the computation of $\nu_{\chi_g,M}$ into the unerlying categoy $\mathcal C$ via the following commutative diagram in $\mathcal C$:
\begin{equation}
    \begin{tikzcd}
	{^{\chi_g}(M^*)} && {^{\chi_g}(M^*)\otimes A} \\
	{^{\chi_g}(M^*)\otimes_A A} \\
	{^{\chi_g}(M^*)\otimes_A ({}^{\chi_g}M \otimes_A ({}^{\chi_g}M)^*)} && {^{\chi_g}(M^*)\otimes {}^{\chi_g}M \otimes ({}^{\chi_g}M)^*} \\
	{(^{\chi_g}(M^*)\otimes_A {}^{\chi_g}M) \otimes_A ({}^{\chi_g}M)^*} \\
	{^{\chi_g}(M^*\otimes_A M) \otimes_A ({}^{\chi_g}M)^*} && {^{\chi_g}(M^*\otimes M) \otimes ({}^{\chi_g}M)^*} \\
	{^{\chi_g}A \otimes_A ({}^{\chi_g}M)^*} && {^{\chi_g}A \otimes ({}^{\chi_g}M)^*} \\
	{A \otimes_A ({}^{\chi_g}M)^*} && {A \otimes ({}^{\chi_g}M)^*} \\
	{({}^{\chi_g}M)^*}
	\arrow["{\id \otimes \eta}", from=1-1, to=1-3]
	\arrow["{r^{-1}_{^{\chi_g}(M^*)}}", from=1-1, to=2-1]
	\arrow[from=1-3, to=2-1]
	\arrow["{\id \otimes \widehat{\mathrm{coev}}_{^{\chi_g}M}}"', from=1-3, to=3-3]
	\arrow["{\id \otimes_A \mathrm{coev}^A_{^{\chi_g}M}}", from=2-1, to=3-1]
	\arrow["\alpha", from=3-1, to=4-1]
	\arrow[from=3-3, to=3-1]
	\arrow[equal, from=3-3, to=5-3]
	\arrow["{(\chi_g)_{M^*,M} \otimes_A \id}", from=4-1, to=5-1]
	\arrow["{^{\chi_g}\mathrm{ev}^A_M \otimes_A \id}", from=5-1, to=6-1]
	\arrow[from=5-3, to=5-1]
	\arrow["{^{\chi_g}\widehat{ev}_M \otimes \id}"', from=5-3, to=6-3]
	\arrow["{(\chi_g)_A^{-1}\otimes_A \id}", from=6-1, to=7-1]
	\arrow["{\upsilon_{\chi_g} \otimes \id}"', from=6-3, to=7-3]
	\arrow["{l_{({}^{\chi_g}M)^*}}", from=7-1, to=8-1]
	\arrow[from=7-3, to=7-1]
	\arrow["{\rho_{({}^{\chi_g}M)^*}}", from=7-3, to=8-1]
\end{tikzcd}
\end{equation}
where we omit labels for the projectors $\pi$.
Using $\Delta$-separability it is straightforward to check that the right path is the identity morphism: the $\upsilon_{\chi_g}$ is cancelled against the twisted action $\rho_{({}^{\chi_g}M)^*}$, cf.\ \eqref{eq:alg-twisted-action-on-module}, and to remove the
coproduct and actions of $A$
one uses in particular $\Delta$-separability.
\end{proof}

\subsection{Orthogonal $G$-grading on $_A\mathcal C$}

Fix a group $G$, a doubly central algebra $(A,\overline{\sigma},\widetilde{\sigma})$ in $\mathcal{C}$, and a $G$-action $\varphi$ on $(A,\overline{\sigma},\widetilde{\sigma})$. 
We assume in addition:
\begin{quote}
    $\phi_g$ is an algebra automorphism of $A$ for each $g \in G$.
\end{quote}
Recall from Remark~\ref{rem:twisted-local}
that an $A$-module $M$ may be simultaneously $g$-local and $k$-local for distinct $g,k\in G$. It is then natural to ask for conditions under which one has
\begin{equation}
    {}_A\mathcal C = {}_A\mathcal C^{\mathrm{gr}} \ ,
\end{equation}
or, equivalently, when each $A$-module $M$ has a unique decomposition $M = \bigoplus_{g \in G} M_g$ into $g$-local submodules $M_g$.
Here we give a sufficient condition in the case where $A$ is a separable algebra
and $G$ is a finite group, generalising the arguments in~\cite{Mc19}.

Assume that the underlying algebra $A$ of the
doubly central algebra $(A,\overline{\sigma},\widetilde{\sigma})$ is in addition separable. We require that the section $S$ lifts to a morphism in $Z(\mathcal C)$ for $\overline A$, that is, for all $U \in \mathcal{C}$,
\begin{equation}\label{eq: coprod half braiding compatible}
(\id_U\otimes S)\circ\overline{\sigma}_U
=
(\overline{\sigma}_U\otimes\id_A)\circ(\id_A\otimes\overline{\sigma}_U)\circ(S\otimes\id_U) \ .
\end{equation}
For each $g \in G$ and $M \in {}_A\mathcal C$, we define a morphism $(\pi_g)_M \colon M\to M$ as
\begin{equation}\label{eq: pi_g}
  \vcenter{\hbox{%
    \includesvg[width=0.27\textwidth]{picture/g_projector}%
  }}
\end{equation}

The notation $(\pi_g)_M$ is justified as $\pi_g$ is a natural endomorphism of the identity functor on $_A\mathcal C$ for each $g \in G$. Indeed, from the definition above we see that $f \circ (\pi_g)_M = (\pi_g)_N \circ f$ for all $A$-linear morphisms $f\colon M \to N$, and in Lemma~\ref{lem:pig-intertwiner} we see that each $(\pi_g)_M$ is a morphism in~$_A\mathcal C$.

\begin{theorem}\label{thm:inner-direct-sum-for-algebra-modules}
Let $\mathcal{C}$ be an abelian monoidal category with right exact tensor product,
let $G$ be a finite group, and let $(A, \overline{\sigma}, \widetilde{\sigma})$ be a doubly central algebra with a $G$-action $\varphi$. Suppose
\begin{enumerate}
    \item $\varphi \colon G \to \Aut_{\mathcal{C}}(A)$ is injective and each $\varphi_g$ is an algebra automorphism of $A$,
    \item $A$ is separable with the section $S$ lifting to $\overline A$ as in \eqref{eq: coprod half braiding compatible}, and $A$ is simple as a
    left $A$-module,
\item $\sum_{g\in G}\pi_g = \id$ as natural endomorphisms of the identity functor on ${}_A\mathcal{C}$.
\end{enumerate} 
Then
\begin{equation}
{}_A\mathcal C
=
{}_A\mathcal C^{\mathrm{gr}}
=
\bigoplus_{g\in G}{}_A\mathcal C_g \ .
\end{equation}
\end{theorem}

To prove this theorem, we first establish the following technical lemmas for a doubly central algebra $(A, \overline{\sigma}, \widetilde{\sigma})$.

\begin{lemma}\label{lem:pig-intertwiner}
Suppose that $S \colon A \to A \otimes A$ is a bimodule map which satisfies \eqref{eq: coprod half braiding compatible} (but not necessarily \eqref{eq: separability}). Then
$(\pi_g)_M$ is $A$-linear and its image 
    $\pi_g(M)$
is $g$-local.
\end{lemma}

\begin{proof}
We first check that $(\pi_g)_M$ is an intertwiner.
We need to show
\begin{equation}
  \vcenter{\hbox{%
    \includesvg[width=0.35\textwidth]{picture/g-projector_A-linear}%
  }}
\end{equation}
Indeed,
\begin{equation*}
    \raisebox{0pt}[105pt][0pt]{
  \includesvg[width=0.85\textwidth]{picture/g_proj_A-linear_proof}
}
\end{equation*}
where the first equality follows from the naturality of both half-braidings, the second uses the associativity of the algebra multiplication and module action alongside the central algebra structure of $(A, \widetilde{\sigma})$, the equality $(*)$ applies~\eqref{eq: coprod half braiding compatible}, and the final equality uses the central algebra structure of $(A,\overline{\sigma})$ and the associativity of the algebra multiplication and module action.

Since $_A\mathcal C$ is abelian, the image $\pi_g(M)$ always exists and is given by $\mathrm{Ker}(\mathrm{Coker}((\pi_g)_M))$. By definition, $(\pi_g)_M$ factors through its image, i.e.\ there exist a monomorphism $i_g\colon \pi_g(M) \hookrightarrow M$ and an epimorphism $p_g\colon M \twoheadrightarrow \pi_g(M)$ such that $(\pi_g)_M = i_g \circ p_g$. 
To see that $\pi_g(M)$ is $g$-local, we first establish the following identity
(recall~\eqref{eq:g-local-2}):
\begin{equation}\label{eq: g-proj g-local}
  \vcenter{\hbox{%
    \includesvg[width=0.4\textwidth]{picture/im_g-projector_g-local}%
  }}
\end{equation}
Indeed, we have:
\begin{equation*}
    \raisebox{0pt}[220pt][0pt]{
  \includesvg[width=0.9\textwidth]{picture/g-proj_g-local_proof}
}
\end{equation*}
Here, the first and the third equalities apply~\eqref{eq:diagram-relations}; the second equality uses commutativity of
$(A, \overline{\sigma})$; the fourth equality relies on the assumption that each $\phi_g$ is an algebra map; the fifth applies compatibility~\eqref{eq: algebra compatible half braiding pic} for $\overline{\sigma}$ and its analog for $\widetilde{\sigma}^{-1}$;
the third-to-last equality uses associativity of multiplication;
the second-to-last equality uses commutativity of
$(A, \widetilde{\sigma})$; and the last step follows from naturality of the half-braiding $\overline{\sigma}$.

We can replace $(\pi_g)_M$ in \eqref{eq: g-proj g-local} by $i_g \circ p_g$, and pull $i_g$ to the top of the diagram since it is $A$-linear by construction. Finally, since $i_g$ is monomorphic, $p_g$ is epic, and the tensor product in $\mathcal C$ is right exact, we conclude that $\pi_g(M)$ is $g$-local.
\end{proof}

We need a second technical lemma that only requires the bimodule and lifting properties of $S$ but not  necessarily \eqref{eq: separability}.

\begin{lemma}\label{lemm:other-g-zero}
    Suppose that $S \colon A \to A \otimes A$ is a bimodule map which satisfies \eqref{eq: coprod half braiding compatible}. 
If $A$ as regular module is a simple object in $_A\mathcal C$, then
for every $g \in G$ with $\varphi_g \neq \id_A$, we have
\begin{equation}
\begin{tikzpicture}[baseline=0pt,font=\small,
  strand/.style={draw, line width=0.7pt}
]
\begin{scope}[xshift=0cm]
    \node at (0,0.4) {$f_g \; \; \; \coloneqq$};
\end{scope}
\begin{scope}[xshift=1.5cm]
  \node at (0,-1.0) {$A$};
  \node at (0.56,0.225) {$\phi_g$};
  \node at (0,1.8) {$A$};
  \draw[strand] (0,-0.8) -- (0,-0.2);
  \draw[strand] (0,-0.2) to[out=0,in=-120] (0.519,0.1);
  \draw[strand] (0,-0.2) to[out=180,in=-90] (-0.6,0.4);
  \draw[strand] (-0.6,0.4) to[out=90,in=180] (0,1);
  \draw[strand] (0.6,0.4) to[out=90,in=0] (0,1);
  \draw[strand] (0,1) -- (0,1.6);
\end{scope}
\begin{scope}[xshift=3cm]
    \node at (0,0.4) {$= \; \; \; 0$};
\end{scope}
\end{tikzpicture}
\end{equation}
\end{lemma}

\begin{proof}
The map $f_g$ is an $A$-intertwiner of the left regular module $A$.
Since $A$ is simple and $_A\mathcal{C}$ is abelian, $f_g$ is either an isomorphism or 0.
We have
\begin{equation*}
    \begin{tikzpicture}[baseline=0pt,font=\small,
  strand/.style={draw, line width=0.7pt},
  eq/.style={inner sep=0pt}
]

\begin{scope}[xshift=-1.15cm]
    \node at (0,0) {$f_g \; \; = $};
\end{scope}

\begin{scope}[xshift=0cm, scale=0.4]
 \node[open dot, inner sep=1.2pt] at (0,-2.23) {};
 \node at (2.4,-3.1) {$A$};
 \draw[strand] (-0.8,0.55) -- (-0.8,-0.8);
 \draw[strand] (-0.8,-0.8) to[out=-90,in=180] (0,-1.6);
 \draw[strand] (0,-1.6) to[out=0,in=-90] (0.8,-0.8);
 \draw[strand] (0.8,-0.8) to[out=90,in=180] (1.6,0);
 \draw[strand] (1.6,0) to[out=0,in=90] (2.4,-0.8);
 \draw[strand] (1.6,0) -- (1.6,0.55);
 \node at (1.6,0.75) {$\phi_g$};
 \draw[strand] (1.6,1.1) to[out=90,in=0] (0,1.8);
 \draw[strand] (-0.8,0.55) to[out=90,in=180] (0,1.8);
 \draw[strand] (0,-1.6) -- (0,-2.1);
 \draw[strand] (2.4,-0.8) -- (2.4,-2.7);
 \draw[strand] (0,1.8) -- (0,2.7);
\end{scope}
\begin{scope}[xshift=1.4cm]
    \node at (0,0) {$=$};
\end{scope}
\begin{scope}[xshift=2.3cm, scale=0.4]
 \node[open dot, inner sep=1.2pt] at (0,-2.23) {};
 \node at (2.4,-3.1) {$A$};
 \draw[strand] (-0.8,0.55) -- (-0.8,-0.8);
 \draw[strand] (-0.8,-0.8) to[out=-90,in=180] (0,-1.6);
 \draw[strand] (0,-1.6) to[out=0,in=-115] (0.65,-1.15);
 \draw[strand] (0.9,-0.6) to[out=75,in=180] (1.6,0);
 \draw[strand] (1.6,0) to[out=0,in=100] (2.4,-0.53);
 \node at (2.4,-0.9) {$\phi_g$};
 \node at (0.8,-0.9) {$\phi_g$};
 \draw[strand] (1.6,0) to[out=90,in=0] (0,1.8);
 \draw[strand] (-0.8,0.55) to[out=90,in=180] (0,1.8);
 \draw[strand] (0,-1.6) -- (0,-2.1);
 \draw[strand] (2.4,-1.1) -- (2.4,-2.7);
 \draw[strand] (0,1.8) -- (0,2.7);
\end{scope}
\begin{scope}[xshift=3.8cm]
    \node at (0,0) {$=$};
\end{scope}
\begin{scope}[xshift=4.7cm, scale=0.4]
 \node[open dot, inner sep=1.2pt] at (0,-2.23) {};
 \node at (2.4,-3.1) {$A$};
 \draw[strand] (-0.8,-0.4) -- (-0.8,-0.8);
 \draw[strand] (-0.8,-0.8) to[out=-90,in=180] (0,-1.6);
 \draw[strand] (0,-1.6) to[out=0,in=-120] (0.62,-1.15);
 \draw[strand] (1,-0.62) to[out=60,in=-90] (1.2,0);
 \draw[strand] (1.2,0) to[out=90,in=0] (0,1);
 \draw[strand] (0,1) to[out=90,in=180] (1.2,1.8);
 \draw[strand] (1.2,1.8) to[out=0,in=90] (2.4,-0.53);
 \node at (2.4,-0.9) {$\phi_g$};
 \node at (0.8,-0.9) {$\phi_g$};
 \draw[strand] (-0.8,-0.4) to[out=90,in=180] (0,1);
 \draw[strand] (0,-1.6) -- (0,-2.1);
 \draw[strand] (2.4,-1.1) -- (2.4,-2.7);
 \draw[strand] (1.2,1.8) -- (1.2,2.7);
\end{scope}
\begin{scope}[xshift=6.3cm]
    \node at (0,0) {$=$};
\end{scope}
\begin{scope}[xshift=7.2cm, scale=0.4]
 \node[open dot, inner sep=1.2pt] at (0,-2.23) {};
 \node at (2.4,-3.1) {$A$};
 \draw[strand] (-0.8,-0.4) -- (-0.8,-0.8);
 \draw[strand] (-0.8,-0.8) to[out=-90,in=180] (0,-1.6);
 \draw[strand] (0,-1.6) to[out=0,in=-120] (0.62,-1.15);
 \draw[strand] (1,-0.62) to[out=60,in=-90] (1.2,0);
 \draw[strand] (1.2,0) to[out=90,in=0] (0,0.8);
 \draw[strand] (0,0.8) to[out=90,in=-135] (1,1.4);
 \draw[strand] (1.2,1.6) to[out=-45,in=90] (2.4,-0.53);
 \draw[strand] (1.2,1.6) to[out=135,in=-90] (0.8,2);
 \draw[strand] (0.8,2) to[out=90,in=180] (1.2,2.3);
 \draw[strand] (1.2,2.3) to[out=0,in=90] (1.6,2);
 \draw[strand] (1.6,2) to[out=-90,in=45] (1.4,1.8);
 \node at (2.4,-0.9) {$\phi_g$};
 \node at (0.8,-0.9) {$\phi_g$};
 \draw[strand] (-0.8,-0.4) to[out=90,in=180] (0,0.8);
 \draw[strand] (0,-1.6) -- (0,-2.1);
 \draw[strand] (2.4,-1.1) -- (2.4,-2.7);
 \draw[strand] (1.2,2.3) -- (1.2,2.7);
 \node[circle,fill=black,inner sep=1.4pt] at (1.2,1.6){};
\end{scope}
\begin{scope}[xshift=8.8cm]
    \node at (0,0) {$=$};
\end{scope}
\begin{scope}[xshift=10.45cm, scale=0.4]
 \node[open dot, inner sep=1.2pt] at (0,-2.23) {};
 \node at (-2.4,-3.1) {$A$};
 \draw[strand] (-0.8,-0.4) -- (-0.8,-0.8);
 \draw[strand] (-0.8,-0.8) to[out=-90,in=180] (0,-1.6);
 \draw[strand] (0,-1.6) to[out=0,in=-120] (0.62,-1.15);
 \draw[strand] (1,-0.62) to[out=60,in=-90] (1.2,0);
 \draw[strand] (1.2,0) to[out=90,in=0] (0,1);
 \draw[strand] (0,1) to[out=90,in=0] (-1.2,1.8);
 \draw[strand] (-1.2,1.8) to[out=180,in=90] (-2.4,-0.53);
 \node at (-2.4,-0.9) {$\phi_g$};
 \node at (0.8,-0.9) {$\phi_g$};
 \draw[strand] (-0.8,-0.4) to[out=90,in=180] (0,1);
 \draw[strand] (0,-1.6) -- (0,-2.1);
 \draw[strand] (-2.4,-1.1) -- (-2.4,-2.7);
 \draw[strand] (-1.2,1.8) -- (-1.2,2.7);
\end{scope}
\begin{scope}[xshift=12.1cm]
    \node at (0,0) {$= \; \; \; f_g \circ \phi_g$};
\end{scope}
\end{tikzpicture}
\end{equation*}
where the first equality uses the unitality of $A$ and~\eqref{eq:diagram-relations}; the second uses the assumption that $\phi_g$ is an algebra morphism; the third applies the associativity of algebra multiplication; the fourth relies on the central algebra structure of $(A, \overline{\sigma})$; the second-to-last applies the naturality of $\overline{\sigma}^{-1}$;
and the final equality uses the $A$-linearity of $f_g$.
If $f_g$ were an isomorphism, this would imply $\varphi_g = \id_A$, which we excluded. Hence $f_g=0$.
\end{proof}

\begin{remark}
The map $f \mapsto f \circ \eta$ gives an isomorphism 
$\mathrm{End}_{_A\mathcal C}(A)
\xrightarrow{\sim} \Hom_\mathcal C(\mathbbm{1},A)$. If $\mathcal{C}$ is linear over an algebraically closed field and has finite-dimensional Hom-spaces, then $A$ being simple as an $A$-module is equivalent to $\mathrm{End}_{_A\mathcal C}(A)$
being one-dimensional, which in turn is equivalent to $\Hom_\mathcal C(\mathbbm{1},A)$ being one-dimensional. An algebra satisfying the latter condition is called \textit{haploid}~\cite{FS01}. 
\end{remark}

\begin{lemma}\label{lem:im-pig-direct-summand}
Suppose that $A$ is separable and that the section $S$ satisfies \eqref{eq: coprod half braiding compatible}.
Then for each $g \in G$ and $M \in {}_A\mathcal C$, the morphism $(\pi_g)_M$ is an idempotent, $\pi_g(M)$ is a direct summand of $M$, and it is the maximal $g$-local submodule of $M$.    
\end{lemma}

\begin{proof}
To see that $(\pi_g)_M$ is an idempotent, we do a slightly more general computation which we can reuse below. Namely, for any $g,k \in G$ we have:
\begin{equation}\label{eq:pig-pik-composite}
  \vcenter{\hbox{%
    \includesvg[width=0.85\textwidth]{picture/idempotent_proof}%
  }}
\end{equation}
where the first equality applies the naturality of both half-braidings; the second uses the central algebra structure of $(A, \widetilde{\sigma})$ alongside the associativity of algebra multiplication; the third uses
compatibility~\eqref{eq: algebra compatible half braiding pic} for $\overline{\sigma}$ and its analog for $\widetilde{\sigma}^{-1}$, together with the assumption that $\phi_k$ is an algebra morphism; and the last step applies the unitality of $A$ as well as~\eqref{eq:diagram-relations}. Now if $k=g$, by~\eqref{eq: separability} $(\pi_g)_M$ is an idempotent.

It follows that the image $\pi_g(M)$ is a retract, i.e.\ a direct summand of $M$. To see that $\pi_g(M)$ is maximal, suppose that $e \colon N \to M$ is an embedding of an arbitrary $g$-local module $N$. It is easy to see that 
$e = e \circ (\pi_g)_N = (\pi_g)_M \circ e$,
using naturality of both half-braidings,
thus $N$ is also a submodule of $\pi_g(M)$. 
\end{proof}

\begin{lemma}\label{lem:pigs-are-orthonormal}
Suppose that $A$ is separable with section $S$ satisfying \eqref{eq: coprod half braiding compatible}, and that $A$ is simple as an $A$-module. 
Then for every $g,k \in G$ we have the following identity for natural endomorphisms of the identity functor $\mathrm{id}_{_A\mathcal C}$:
\begin{equation}\label{eq: pi_g orthogonal}
    \pi_g \circ \pi_k =
\delta_{\varphi_g,\varphi_k}\pi_g \ .
\end{equation}
\end{lemma}

\begin{proof}
This is the place where we use the more general computation in \eqref{eq:pig-pik-composite}. By Lemma~\ref{lemm:other-g-zero}, the right hand side is zero when
$\varphi_k \neq \varphi_g$; 
and reduces to $\pi_g$ when 
$\varphi_k = \varphi_g$.
\end{proof}

We are now ready to prove Theorem~\ref{thm:inner-direct-sum-for-algebra-modules}. The proof broadly follows that in \cite[Sec.\,3]{Mc19}, although our categorical setting is more general (e.g.\ in \cite{Mc19} the category is linear and braided with absolutely simple tensor unit, while here we assume none of these) and our assumptions on $A$ are weaker.

\begin{proof}[Proof of Theorem~\ref{thm:inner-direct-sum-for-algebra-modules}]
It suffices to show that each $M \in {}_A\mathcal C$ admits a decomposition
\begin{equation}\label{eq:M-sum-decomp}
    M = \bigoplus_{g \in G} \pi_g(M)
\end{equation}
into $g$-local modules,
and that if an $A$-module $N$ is both $g$-local and $k$-local for $g \neq k$, then $N=0$.

First note that assumption (iii) applied to $M$ gives
\begin{equation}
    \sum_{g \in G} (\pi_g)_M = 
    \id_M \ .
\end{equation}
Furthermore, Lemma~\ref{lem:pigs-are-orthonormal} and assumption (i) on injectivity of $\phi$
gives
\begin{equation}
	\pi_g \circ \pi_k = \delta_{g,k} \pi_g
\end{equation}
as natural endomorphisms, and hence on every module $M$. In particular, a non-zero module cannot be at the same time $g$-local and $k$-local for $g \neq k$.
Altogether we have a complete set of orthonormal idempotents, and by Lemma~\ref{lem:im-pig-direct-summand}, the image of
$\pi_g(M)$ is $g$-local. This shows~\eqref{eq:M-sum-decomp}.
\end{proof}

\begin{remark}\label{rem:direct-sum-decomp}
Let us comment more on assumptions (ii) and (iii)
in Theorem~\ref{thm:inner-direct-sum-for-algebra-modules}:
\begin{enumerate}

\item 
For $\mathcal{C}$ as in Theorem~\ref{thm:inner-direct-sum-for-algebra-modules},
in the assumption (iii)
we could have demanded the seemingly weaker condition that
$P \coloneqq \sum_{g\in G}\pi_g = \gamma \,\id$ for some invertible element $\gamma \in \End_\mathcal{C}(\mathbbm{1})$. However, since $\pi_g \circ \pi_k = \delta_{g,k} \pi_k$ we have $P\circ P = P$, which implies $\gamma=1$.

\item
Assuming conditions (i) and (ii),
we now give a sufficient condition for assumption (iii)
which is formulated in terms of $A$, rather than as a condition for all $A$-modules. Suppose the image of 
$\Phi \coloneqq \sum_{g \in G} \varphi_g \colon A \to A$ 
is transparent in the following sense: for all $U \in \mathcal{C}$ it holds that
\begin{equation}\label{eq:Phi-transparent}
		\widetilde\sigma_U^{-1} \circ \overline\sigma_U \circ (\Phi \otimes \id_U)
        = \Phi \otimes \id_U \ .
	\end{equation}
We then have
\begin{equation}
  \vcenter{\hbox{%
    \includesvg[width=0.7\textwidth]{picture/sum_pi_g_sufficient}%
  }}
\end{equation}
where in the last equality we apply Lemma~\ref{lemm:other-g-zero}. We note that 
the condition~\eqref{eq:Phi-transparent} is weaker than the corresponding condition demanded in \cite[Thm.\,4.2]{Ki04} or \cite[Sec.\,3]{Mc19}, where one requires that the image of $\Phi$ is just the tensor unit. This weaker condition will be useful for our example in the next section.

\item
Let $\mathcal{C}$ be a fusion category\footnote{i.e.\ a $\mathbb{K}$-linear semisimple finite tensor category, which is in particular rigid with a simple tensor unit.}.
Concerning condition (ii), any separable algebra $A$ in $\mathcal{C}$ that is simple as a left $A$-module is Morita equivalent to a separable algebra that is merely indecomposable  (i.e.\ not a direct sum of two non-zero algebras in $\mathcal{C}$). Indeed, the notion of separable algebras is Morita invariant in the fusion case~\cite[Sec.\,2.5]{DSPS13}: for any non-zero $M \in {}_A\mathcal C$, the internal Hom algebra $A_M \coloneqq \underline{\mathrm{Hom}}(M,M)\in \mathcal C$ is again separable and has equivalent category of modules, i.e.\ ${}_A\mathcal C \cong {}_{A_M}\mathcal C$ as right $\mathcal C$-module categories, using~\cite[Thm.\,7.10.1\,\&\,\text{top p.152}]{EGNO15}.
In particular, any simple $A$-module $M$ (e.g.\ $M=A$) gives a Morita-equivalent separable algebra $A_M$ that is simple as a left $A_M$-module~\cite[Lem.\,4.2]{EO03}, using the fact that ${}_A\mathcal C$ is semisimple (Remark~\ref{rem:sep-alg}) and thus exact as a right $\mathcal{C}$-module category. However, since we are interested in ${}_A\mathcal C$ up to equivalence, we are free to choose any (non-zero and) non-simple $M$ yielding necessarily 
an indecomposable algebra $A_M$; compare with~\cite[Ex.\,7.10.6]{EGNO15}.
It would be interesting to understand how the other structures and conditions required by Theorem~\ref{thm:inner-direct-sum-for-algebra-modules} behave under Morita transport to develop an appropriate Morita theory for doubly central algebras.
 
If $\mathcal{C}$ is a finite tensor category (not necessarily semisimple) and $A$ is a separable algebra in $\mathcal{C}$, then the category ${}_A\mathcal C$ remains an exact right $\mathcal C$-module category~\cite[Lem.\,4.4\,\&\,4.9]{GL26}, and the algebra $\underline{\mathrm{Hom}}(M,M)\in \mathcal C$ is still simple as its own left module for every simple module $M\in {}_A\mathcal C$. However, the notion of separable algebras is not Morita invariant if $\mathcal C$ is not semisimple~\cite[Ex.\,4.16]{GL26}. 
Therefore, we can only conjecture that every indecomposable separable algebra in $\mathcal{C}$ is Morita-equivalent to an algebra which is both separable and simple as its left module.
\end{enumerate}
\end{remark}

\subsection{Examples from abelian group algebras}\label{twisted local module examples}

Fix an algebraically closed field $\mathbb{K}$ of characteristic 0. Let $K$ be an abelian group and let $\mathcal{C} \coloneqq \mathrm{vec}_K$ be the category of finite-dimensional $K$-graded vector spaces over $\mathbb{K}$ with untwisted associator.
In this subsection, we produce explicit examples of $\chi$-crossed ribbon categories built from twisted local modules over group algebras in $\mathrm{vec}_K$.
As in Section~\ref{sec:H-crossed-vec_G}, we take $\{\mathbb{K}_k\}_{k\in K}$ as representatives of the isomorphism classes of simple objects, and fix the unit of $\mathbb{K}$ as a preferred basis $\mathbf{1}_k \in \mathbb{K}_k$ in each $\mathbb{K}_k$. 

\subsubsection{Twisted-local modules}
Let $B \subset K$ be a finite
subgroup. We define an algebra object
\begin{equation}
    A \; \coloneqq \; \bigoplus_{b \in B} \mathbb{K}_b
\end{equation}
in $\mathcal C$ with unit $\mathbf{1}_0$ and multiplication given on simple objects by
\begin{equation}
    \mathbf{1}_b \otimes \mathbf{1}_{b'} \mapsto \mathbf{1}_{b+b'}  \qquad b,b' \in B \ .
\end{equation}

\begin{remark}
Note that for any normalised 2-cocycle $\gamma\colon B^2 \to \mathbb{K}^\times$, we can define a new algebra $A_\gamma$ with a twisted multiplication given by $\mathbf{1}_b \otimes \mathbf{1}_{b'} \mapsto \gamma(b,b') \, \mathbf{1}_{b+b'}$.
This would lead to a more general class of crossed graded monoidal structures
in Proposition~\ref{prop: _AC equiv Vec_G} below, but also to more technical computations and we will not pursue this in the present paper.
\end{remark}

\begin{lemma}\label{lem:group-algebra-centre-lift}
    A lift of the algebra $A$ to a commutative algebra $(\overline{A}, \overline{\sigma})$ in the Drinfeld centre $Z(\mathrm{vec}_K)$ is specified by a bicharacter
\begin{equation}
    \overline{\delta}\colon B \times K/B \to \mathbb{K}^\times \ ,
\end{equation}
and the corresponding half-braiding is given by
\begin{gather}
    \overline{\sigma}_{V} \colon A \otimes V \to V \otimes A \ , \qquad \mathbf{1}_b \otimes v_k \mapsto \overline{\delta}(b,[k]) \, v_k \otimes \mathbf{1}_b \ ,
\end{gather}
where $[k]$ denotes the image of $k$ under the quotient map $\pi\colon K \to K/B$.
\end{lemma}
\begin{proof}
    Since $\mathrm{vec}_K$ is semisimple, a half-braiding $\overline{\sigma}$ of $A$ is specified by a function $\delta\colon B \times K \to \mathbb{K}^\times$ such that $\overline{\sigma}_{V_k}(\mathbf{1}_b \otimes v_k) = \delta(b,k) \, v_k \otimes \mathbf{1}_b$. $\delta$ is multiplicative in each variable because we require that, for all $k, l \in K$,
    \begin{subequations}
        \begin{gather}
       \overline{\sigma}_{V_k} \circ (\mu \otimes \mathrm{id}_{V_k}) = (\mathrm{id}_{V_k} \otimes \mu)\circ(\overline{\sigma}_{V_k} \otimes \mathrm{id}_A)\circ(\mathrm{id}_A \otimes \overline{\sigma}_{V_k}) \ ,  \\
       \overline{\sigma}_{V_k\otimes U_l} = (\mathrm{id}_{V_k} \otimes \overline{\sigma}_{U_l})\circ(\overline{\sigma}_{V_k} \otimes \mathrm{id}_{U_l}) \ ,
    \end{gather}
    \end{subequations}
where the first condition is \eqref{eq: algebra compatible half braiding b} and the second condition is the compatibility of half-braidings with the tensor product.
    Moreover, since $(\overline{A}, \overline{\sigma})$ is commutative in $Z(\mathrm{vec}_K)$, we need
    \begin{equation}
        \delta(b,c) \equiv 1 \; \; \textup{ for all $b,c \in B$} \ .
    \end{equation}
    Therefore, $\overline{\delta}(b,[k]) \coloneqq \delta(b,k)$ defines a bicharacter of $B \times K/B$.
\end{proof}

Fix two bicharacters $\overline{\delta}, \widetilde{\delta}\colon B \times K/B \to \mathbb{K}^\times$, and denote by $\overline\sigma$ and $\widetilde\sigma$ their corresponding half-braidings obtained by Lemma~\ref{lem:group-algebra-centre-lift}. We consider the doubly central algebra
\begin{equation}
    (A, \overline{\sigma}, \widetilde{\sigma}) \ .
\end{equation}
Take $G:= K/B$, and define a $B$-character for each $g \in G$,
\begin{equation}\label{eq:beta_g-character}
    \beta_g \colon B \to \mathbb{K}^\times \ , \qquad
    b \mapsto \frac{\widetilde{\delta}(b,g)}{\overline{\delta}(b,g)} \ .
\end{equation}
Define a $G$-action on $(A, \overline{\sigma}, \widetilde{\sigma})$ by
\begin{equation}\label{eq: phi G-action formula}
    \phi\colon G = K/B \to \mathrm{Aut}_\mathcal{C}(A) \ , \qquad g \; \mapsto \; \big(\mathbf{1}_b \mapsto \beta_g(b) \, \mathbf{1}_b \big) \ .
\end{equation}
Since the map $(g,b) \mapsto \beta_g(b)$ is a bicharacter on $G \times B$, $\phi_g$ is actually an algebra automorphism for each $g \in G$. 
Compatibility with the half-braidings as in \eqref{eq: naturality of action} is immediate.
Furthermore, $\phi_g = \id$ if and only if $\beta_g \equiv 1$ is the trivial character.

The regular $A$-module $A$ is simple, and it is $\Delta$-separable Frobenius with the coproduct 
    \begin{equation}\label{eq: group algebra coprod}
        \Delta(\mathbf 1_b) = \frac{1}{|B|}\sum_{i \in B}\mathbf 1_{b-i} \otimes \mathbf 1_i \ ,
        \qquad \forall \, b \in B \ .
    \end{equation}
and the counit $\epsilon(\mathbf 1_b) = |B| \, \delta_{b,0}$. The coproduct $\Delta$
lifts to $\overline A$: for any $u_k \in U$ of degree $k$, the two sides of
\eqref{eq: coprod half braiding compatible} become
\begin{align}
& (\id_U\otimes \Delta)\circ\overline{\sigma}_U(\mathbf 1_b \otimes u_k)
= \frac{1}{|B|}\sum_{i \in B} \overline{\delta}(b,[k]) \,
u_k \otimes \mathbf 1_{b-i} \otimes \mathbf 1_i \ ,
\\
& (\overline{\sigma}_U\otimes\id_A)\circ(\id_A\otimes\overline{\sigma}_U)\circ(\Delta\otimes\id_U)(\mathbf 1_b \otimes u_k)
=
\frac{1}{|B|}\sum_{i \in B} \overline{\delta}(i,[k]) \, \overline{\delta}(b-i,[k]) \,
u_k \otimes \mathbf 1_{b-i} \otimes \mathbf 1_i  \ ,
\nonumber
\end{align}
which agree because $\overline{\delta}$ is a bicharacter.

Let $M$ be an $A$-module and $g \in G$. From Lemma~\ref{lem:im-pig-direct-summand} we see that $(\pi_g)_M$ in \eqref{eq: pi_g} is an idempotent and that its image is the maximal $g$-local submodule of $M$. Explicitly, for any $m_k \in M$ homogeneous of degree $k \in K$, we have
\begin{equation}\label{eq: g-local computation}
    (\pi_g)_M(m_k) 
    = \frac{1}{|B|}\sum_{i \in B}
    \frac{\overline{\delta}(i,[k])}{\widetilde{\delta}(i,[k])} \beta_g(i) \, m_k
    = \frac{1}{|B|}\sum_{i \in B} \beta_{g-[k]}(i)
    \, m_k
    = \delta_{\varphi_{g},\varphi_{[k]}}\,m_k \ .
\end{equation}
In the last step we use the fact from \eqref{eq:beta_g-character} that each $\beta_h$ 
is a character of $B$; thus, its sum over $B$ vanishes unless $\beta_h$ 
is the trivial character. By definition, $\beta_{g-[k]}$ is trivial if and only if $\phi_g = \phi_{[k]}$.\footnote{
We do not assume $\phi$ to be injective, hence the Kronecker delta $\delta_{\phi_g,\phi_{[k]}}$ cannot be replaced by $\delta_{g,[k]}$ in general.}
Consequently, an $A$-module $N$ is $g$-local for some $g \in G$ if and only if $N$ can be decomposed in $\mathcal C$ as $N=\bigoplus_k N_k$ where the sum is over all $k\in K$ with $\phi_{[k]} = \phi_g$
(or just $[k] = g$ if $\phi$ is injective)
and $N_k$ is homogeneous of degree $k$.
Furthermore, if $\phi$ is injective, and thus $K$ is necessarily finite,\footnote{ 
\label{fn:K-finite-phi-injective}
Injectivity of $\varphi \colon G \to \mathrm{Aut}_{\mathcal{C}}(A)$ already implies that $K$ and $G$ are finite. Indeed, $G$ embeds into the subgroup of algebra automorphisms inside $\mathrm{Aut}_{\mathcal{C}}(A)$, which is isomorphic to the character group $\widehat B$ of $B$.
Hence $G$ must be finite and as $G=K/B$ also $K$ is finite. 
}
we easily check condition (iii) of Theorem~\ref{thm:inner-direct-sum-for-algebra-modules} using~\eqref{eq: g-local computation} and thus conclude that in this case $_A\mathcal C^{\mathrm{gr}} = {}_A\mathcal C$.
In particular, the only $0$-local $A$-modules in this case are direct sums of $A$.

\subsubsection{Crossed actions}
Take $H \coloneqq \widehat{B}$. Define an $H$-action on the algebra $A$ by
\begin{equation}\label{eq:alg-crossed-action}
    \upsilon\colon H \to \mathrm{Aut}_{\mathcal{C}}(A) \ , \qquad h \; \mapsto \; \Big(\mathbf{1}_b \mapsto h(b)\mathbf{1}_b \Big) \ .
\end{equation}
Let $H$ act on $G$ trivially. 
Compatibility with the half-braidings as in
\eqref{eq: naturality of action}, as well as the conditions~\eqref{action compatibility} and~\eqref{upsilon alegbra iso} 
are immediate,
and by Proposition~\ref{structure morphism for h-action in AMod}, ${}_A\mathcal{C}^{\mathrm{gr}}$ is a $\widehat{B}$-crossed $K/B$-graded monoidal category.
In what follows, for injective $\varphi$ in \eqref{eq: phi G-action formula}
we describe ${}_A\mathcal{C}^{\mathrm{gr}}$ explicitly in terms of graded vector spaces and the above fixed data.

\medskip

Fix a set-theoretic section $s \colon G\to K$ with $s(0)=0$ for the quotient map $\pi\colon K \to G$. Define $\kappa\colon G^2 \to B$ by
\begin{equation}\label{eq:kappa_def}
    \kappa(g_1,g_2) \;\coloneqq\; s(g_1)+s(g_2)-s(g_1+g_2) \ .
\end{equation}
The image of $\kappa$ is indeed in $B$, because $s$ is a section and hence $\pi(\kappa(g_1,g_2)) \equiv 0$.
It is straightforward to check that $\kappa$ is a normalised $2$-cocycle 
in $Z^2(G,B)$,
where $G$ acts on $B$ trivially.
Define $\omega:G^3\to \mathbb{K}^\times$ and $\Omega\colon H \times G^2 \to \mathbb{K}^\times$ by
\begin{subequations}\label{eq:omega_def}
\begin{align}
\omega(g_1,g_2,g_3)&\coloneqq\overline{\delta}\big(\kappa(g_2,g_3),\,g_1\big) \ , \label{eq:omega_def a}\\
    \Omega(h,g_1,g_2)&\coloneqq h\big(\kappa(g_1,g_2)\big)\ .\label{eq:omega_def b}
\end{align}
\end{subequations}

Recall the cohomology groups $H^n_{\mathrm{tot};1}(G \times H;\mathbb{K}^\times)$ introduced prior to Proposition~\ref{group cohomology classification H-crossed action} (specialised to trivial $H$-action on $G$).

\begin{lemma}\label{lem:omega_Omega_total_cocycle_minimal}
$(\omega,\Omega,1)$ is a normalised element of $Z^3_{\mathrm{tot};1}(G \times H;\mathbb{K}^\times)$, and its class
$[\omega,\Omega,1] \in H^3_{\mathrm{tot};1}(G \times H;\mathbb{K}^\times)$ is independent of the choice of the section $s$.
\end{lemma}

\begin{proof}
Since $s(0)=0$, the fact that $\omega$ and $\Omega$ are normalised follows immediately from $\kappa(0,g)=\kappa(g,0)=0$ and $\overline{\delta}(0,-)=h(0)=1$. By straightforward computation one can check that
\begin{equation}
    \mathrm d_G\omega=1 \ , \qquad \mathrm d_G\Omega=1 \ , \qquad \mathrm d_H\Omega=1 \ ,
\end{equation}
which implies that $(\omega, \Omega, 1) \in Z^3_{\mathrm{tot};1}(G \times H;\mathbb{K}^\times)$.

Let $s'\colon G \to K$ be another section with $s'(0)=0$, and let $(\omega', \Omega',1)$ be the associated normalised total cocycle.
Define
\begin{equation}
    u(g)\coloneqq s'(g)-s(g)\in B \ ,
\end{equation}
and $\alpha\colon G^2 \to \mathbb{K}^\times, \, \beta\colon H \times G \to \mathbb{K}^\times$ by
\begin{equation}
\alpha(g_1,g_2)\coloneqq\overline{\delta}(u(g_2),g_1)\ ,
\qquad
\beta(h,g)\coloneqq h(u(g)) \ .
\end{equation}
A direct computation gives the explicit coboundary relations
\begin{equation}
     \mathrm d_G\alpha = \frac{\omega}{\omega'} \ ,
\qquad \mathrm d_H\alpha = 1 \ , \qquad \mathrm d_G\beta = \frac{\Omega'}{\Omega} \ , \qquad \mathrm d_H\beta = 1 \ ,
\end{equation}
hence \((\omega',\Omega',1)\) 
and \((\omega,\Omega,1)\) define the
same cohomology class in $H^3_{\mathrm{tot};1}(G \times H;\mathbb{K}^\times)$.
\end{proof}

For a normalised $3$-cocycle $(\omega, \Omega, t) \in Z^3_{\mathrm{tot};1}(G \times H; \mathbb{K}^\times)$,
recall the notation $\mathrm{vec}_G^{(\omega,\Omega,t)}$ from~\eqref{eq:vecG-kappa-def}
for the $H$-crossed $G$-graded monoidal category based on $\mathrm{vec}_G$.
We state a proposition in a similar spirit to~\cite[Thm.\,4.1]{CGR17}, whose proof extends ideas from~\cite[Sec.\,6]{CGR17} to our setting of $g$-local modules.

\begin{proposition}\label{prop: _AC equiv Vec_G}
If $\phi$ in~\eqref{eq: phi G-action formula} is injective, then $_A\mathcal C = {}_A\mathcal{C}^\mathrm{gr} \cong \mathrm{vec}_G^{(\omega,\Omega,1)}$ as $\widehat{B}$-crossed $K/B$-graded monoidal categories,
for $\omega$ and $\Omega$ as in~\eqref{eq:omega_def}.
\end{proposition}
\begin{proof}
Recall from Footnote~\ref{fn:K-finite-phi-injective}
that for $\varphi$ injective we have that $K$ is finite and that ${}_A\mathcal C = {}_A\mathcal{C}^\mathrm{gr}$.
We define a $\widehat{B}$-equivariant $K/B$-graded monoidal functor $F \colon \mathrm{vec}_{K/B}^{(\omega,\Omega,1)} \to {}_A\mathcal{C}$ as follows
(recall conventions in Definition~\ref{crossed functor}):
\begin{itemize}
    \item We first define a $G$-graded $\mathbb{K}$-linear functor, which on simple objects is given by
\begin{equation}
    F\colon \mathrm{vec}_{K/B} \to {}_A\mathcal{C} \ , \qquad
    \mathbb{K}_g \mapsto A \otimes \mathbb{K}_{s(g)} \ .
\end{equation}
Here, $\otimes$ denotes the tensor product in $\mathcal{C}$, and the $A$-action on $A \otimes \mathbb{K}_{s(g)}$ is by left multiplication: $\mathbf 1_b \otimes (\mathbf 1_{b'}\otimes \mathbf 1_{s(g)}) \mapsto \mathbf 1_{b+b'}\otimes \mathbf 1_{s(g)}$.

\item We define a monoidal structure on $F$ by
\begin{equation}
    F_\mathbbm 1\colon A \to A \otimes \mathbb{K}_0 \ , \qquad \mathbf{1}_b \mapsto  \mathbf{1}_b \otimes \mathbf{1}_0 \ ,
\end{equation}
and for any $g, g' \in G$,
\begin{align}
    \begin{split}
        F_{g,g'}^{-1}\colon A \otimes \mathbb{K}_{s(g+g')} &\to \big(A \otimes \mathbb{K}_{s(g)}\big) \otimes_A \big(A \otimes \mathbb{K}_{s(g')}\big) \ ,\\
        \mathbf{1}_{b} \otimes \mathbf{1}_{s(g+g')} &\mapsto (\mathbf{1}_{b-\kappa(g,g')} \otimes \mathbf{1}_{s(g)}) \otimes_A (\mathbf{1}_{0} \otimes \mathbf{1}_{s(g')}) \ .
    \end{split}
\end{align}
One can check that these are indeed degree preserving $A$-linear maps.

\item 
For $\widehat B$-equivariance, we define, for any $h \in \widehat B$, 
$g \in G$, 
\begin{align}
    \begin{split}
    \eta^h_g\colon {}^h(A \otimes \mathbb{K}_{s(g)}) &\to A \otimes {}^h\mathbb{K}_{s(g)} \ , \\
        \mathbf{1}_b \otimes \mathbf{1}_{s(g)} &\mapsto h(b) \, \mathbf{1}_b \otimes \mathbf{1}_{s(g)} \ .
    \end{split}
\end{align}
which is $A$-linear following \eqref{eq:alg-twisted-action-on-module} and \eqref{eq:alg-crossed-action}.
\end{itemize}

The coherence conditions for $F$ to be a $\widehat{B}$-equivariant $K/B$-graded monoidal functor are straightforward to verify. For instance, the pentagon identity reads:
\begin{equation*}
    \hspace*{-2.5mm}
    \small
\begin{tikzcd}[column sep = 1em, row sep = 1.8em]
	{A \otimes \mathbb{K}_{s(g_1+g_2+g_3)}} && {A \otimes \mathbb{K}_{s(g_1+g_2+g_3)}} \\
	{\big(A \otimes \mathbb{K}_{s(g_1)}\big) \otimes_A \big(A \otimes \mathbb{K}_{s(g_2+g_3)}\big)} && {\big(A \otimes \mathbb{K}_{s(g_1+g_2)}\big) \otimes_A \big(A \otimes \mathbb{K}_{s(g_3)}\big)} \\
	{\big(A \otimes \mathbb{K}_{s(g_1)}\big) \otimes_A \Big(\big(A \otimes \mathbb{K}_{s(g_2)}\big) \otimes_A \big(A \otimes \mathbb{K}_{s(g_3)}\big)\Big)} && {\Big(\big(A \otimes \mathbb{K}_{s(g_1)}\big) \otimes_A \big(A \otimes \mathbb{K}_{s(g_2)}\big)\Big) \otimes_A \big(A \otimes \mathbb{K}_{s(g_3)}\big)}
	\arrow["{\omega(g_1,g_2,g_3)\mathrm{id}}", from=1-1, to=1-3]
	\arrow["{F^{-1}_{g_1,g_2+g_3}}"', from=1-1, to=2-1]
	\arrow["{F^{-1}_{g_1+g_2,g_3}}", from=1-3, to=2-3]
	\arrow["{\mathrm{id} \otimes_A F^{-1}_{g_2,g_3}}"', from=2-1, to=3-1]
	\arrow["{F^{-1}_{g_1,g_2}\otimes_A \mathrm{id}}", from=2-3, to=3-3]
	\arrow["{\alpha}", from=3-1, to=3-3]
\end{tikzcd}
\end{equation*}
where $\omega$ is defined as in~\eqref{eq:omega_def a}, and the associator $\alpha$ as in~\eqref{associator Amod}.
Starting with a vector $\mathbf{1}_{b} \otimes \mathbf{1}_{s(g_1+g_2+g_3)}$ in the upper left corner and following the paths of the diagram, the source of $\alpha$ is found to be
\begin{equation}\label{eq: source of alpha}
    \big(\mathbf{1}_{b-\kappa(g_1,g_2+g_3)} \otimes \mathbf{1}_{s(g_1)}\big)
    \otimes_A
    \big(
    (\mathbf{1}_{-\kappa(g_2,g_3)} \otimes \mathbf{1}_{s(g_2)})
    \otimes_A
    (\mathbf{1}_{0} \otimes \mathbf{1}_{s(g_3)})
    \big) \ ,
\end{equation}
and the target of $\alpha$ is
\begin{equation}
\omega(g_1,g_2,g_3)
    \big(
    (\mathbf{1}_{b-\kappa(g_1+g_2,g_3)-\kappa(g_1,g_2)} \otimes \mathbf{1}_{s(g_1)})
    \otimes_A
    (\mathbf{1}_{0} \otimes \mathbf{1}_{s(g_2)})
    \big)
    \otimes_A
    \big(\mathbf{1}_{0} \otimes \mathbf{1}_{s(g_3)}\big) \ .
\end{equation}
To move $\mathbf{1}_{-\kappa(g_2,g_3)}$ from the middle tensor factor in~\eqref{eq: source of alpha} to the left, we apply \eqref{eq: tensor over A} to obtain
\begin{equation}
\overline{\delta}\big(\kappa(g_2,g_3),\,g_1\big)
\big(\mathbf{1}_{b-\kappa(g_1,g_2+g_3)-\kappa(g_2,g_3)} \otimes \mathbf{1}_{s(g_1)}\big)
\otimes_A
\big(
(\mathbf{1}_{0} \otimes \mathbf{1}_{s(g_2)})
\otimes_A
(\mathbf{1}_{0} \otimes \mathbf{1}_{s(g_3)})
\big) \ .
\end{equation}
Note that $\kappa(g_1,g_2+g_3)+\kappa(g_2,g_3) = \kappa(g_1+g_2,g_3)+\kappa(g_1,g_2)$.
Moreover, since $\omega(g_1,g_2,g_3) = \overline{\delta}\big(\kappa(g_2,g_3),\,g_1\big)$ as defined in~\eqref{eq:omega_def a}, and since $\alpha$ merely re-parenthesises the tensor factors, we conclude that the diagram commutes.

$\widehat{B}$-equivariance conditions~\eqref{crossed functor definition} are
straightforward to check and are
left to the reader.

\medskip

Next we show that $F$ is a $\mathbb{K}$-linear equivalence. Any $M  \in {}_A\mathcal C$ can be
uniquely decomposed (as a graded vector space) into
\begin{equation}
    M = \bigoplus_{k \in K}M_k = \bigoplus_{b \in B}\Big(\bigoplus_{g \in K/B}M_{s(g)+b}\Big) \ ,
\end{equation}
where each $M_k$ is homogeneous of degree $k$. Denote
\begin{equation}
   \mathrm{vec}_{K/B} \ni M_{(0)} 
\coloneqq  \bigoplus_{g \in K/B}M_{s(g)}\ ,
\end{equation}
where $M_{s(g)}$ is the degree-$g$ component of $M_{(0)}$.
The inverse of $F$ is given by
\begin{equation}
     G\colon {}_A\mathcal C \to \mathrm{vec}_{K/B} \ , \qquad M \mapsto M_{(0)} \ .
\end{equation}
\end{proof}

\begin{example}
Take $K = \Z/16\Z$ and $B = \Z/4\Z$, so that $G = K/B \cong \Z/4\Z$, $H = \widehat B \cong \Z/4\Z$, and $H$ acts trivially on $G$. This corresponds to case (\romannumeral 3) in Section~\ref{Z/4 example},
and we can compare $_A\mathcal C$ to the classification there.

First note that $[1,\Omega,1] = [1,1,t']
\in H^3_{\mathrm{tot};1}(\Z/4\Z \times \Z/4\Z;\mathbb{K}^\times)$, where $\Omega$ is defined as in~\eqref{eq:omega_def b}, and $t'$ as in~\eqref{generator ti Z/4Z-cocycle}.
Indeed, fix a 16th root of unity $\mu \in \mathbb{K}^\times$ such that
\begin{equation}
    \mu^4 = i \ ,
\end{equation}
where $i$ is the fixed 4th root of unity as in~\eqref{eq:example-i}.
Choose $s:\Z/4\Z \to \Z/16\Z$ with $s(g) = g$, and define
\begin{equation}
    \beta\colon \Z/4\Z \times \Z/4\Z \longrightarrow \mathbb{K}^\times \ , \qquad \beta(h,g) \; = \; \mu^{hg} \ ,
\end{equation}
then one can check that $(1,\Omega,1)$ and $(1,1,t')$ are cohomologous via $(1, \beta)$ in the sense of \eqref{d2tot-via-alpha-beta}.

For $a \in \{0,1,2,3\}$, there are four possible bicharacters $\overline{\delta}_a \colon \mathbb{Z}/4\mathbb{Z} \times \mathbb{Z}/4\mathbb{Z} \to \mathbb{K}^\times$, uniquely determined by the value $\overline{\delta}_a(1,1)=i^a$. For each $\overline{\delta}_a$, there are two choices of $\widetilde{\delta}$, namely
$\widetilde{\delta}_{a\pm}
=\overline{\delta}_{a\pm1}$, for which $\phi$ in~\eqref{eq: phi G-action formula} is injective. For every pair $(\overline{\delta}_a, \, \widetilde{\delta}_{a\pm})$, the 3-cocycle defined in~\eqref{eq:omega_def a} is given by $\omega^a$, where $\omega$ is as in~\eqref{generator Z/4Z-cocycle}. Note that $\Omega$ in~\eqref{eq:omega_def b} is independent of $a$.
Therefore, by Proposition~\ref{prop: _AC equiv Vec_G}, the $\mathbb{Z}/4\mathbb{Z}$-crossed category ${}_A\mathcal{C}$ associated to the corresponding doubly central algebra $(A,\overline{\sigma}_a,\widetilde{\sigma}_{a\pm})$ is equivalent to $\mathrm{vec}_{\mathbb{Z}/4\mathbb{Z}}^{(\omega^a,\Omega,1)}$, and hence to $\mathrm{vec}_{\mathbb{Z}/4\mathbb{Z}}^{(\omega^a,1,t')}$.
\end{example}

\subsubsection{Crossed ribbon structures}

We now consider $\chi$-crossed braided and ribbon
structures on $_A\mathcal{C}^{\mathrm{gr}}$ (note that $\chi$ is not specified yet). We first need a braiding on $\mathrm{vec}_K$, which is determined by a bicharacter $\tau\colon K \times K \to \mathbb{K}^\times$ 
via
\begin{equation}\label{eq: tau braiding}
   C_{V,U}\colon V \otimes U \to U \otimes V \ , \qquad v_k \otimes u_l \mapsto \tau(k,l)\, u_l \otimes v_k \ ;
\end{equation}
recall~\eqref{eq:qa-cocycle-system'} for trivial $\omega$ with $R$ replaced by $\tau$.
To apply Theorem~\ref{twisted local module crossed braiding} we need $\overline{\sigma}$ to coincide with the braiding in $\mathrm{vec}_K$. Therefore, $\tau$ and $\overline{\delta}$ must satisfy
\begin{equation}\label{eq: tau bar-tilde compatible}
    \overline{\delta}(b,[k]) = \tau(b,k) \ , \qquad \forall\; (b,k) \in B \times K \ .
\end{equation}
This implies that $\tau$ restricted on $B \times B$ is trivial. Fix such a $\tau$, and define $\overline{\delta}$ as in~\eqref{eq: tau bar-tilde compatible}.

Define a bicharacter $\chi^*\colon B \times K/B \to \mathbb{K}^\times$ by
\begin{equation}
    \chi^*(b,[k]) = \big(\tau(b,k)\tau(k,b)\big)^{-1} \ .
\end{equation}
This induces a canonical group homomorphism:
\begin{equation}\label{eq:chi-B-def}
    \chi\colon K/B \to \widehat{B} \ , \qquad g \mapsto \chi^*(-,g)\ .
\end{equation}
Then $\chi$ with trivial $\widehat{B}$-action is a crossed module. One can further check that condition~\eqref{crossing change'} is satisfied, hence $_A\mathcal{C}^{\mathrm{gr}}$ admits a $\chi$-crossed braiding following Theorem~\ref{twisted local module crossed braiding}.

Recall $G \coloneqq K/B$ and the section $s$ introduced above~\eqref{eq:kappa_def}. Define $R\colon G \times G \to \mathbb{K}^\times$ by
\begin{equation}
    R(g_1,g_2) = \tau\big(s(g_1),s(g_2)\big) \ .
\end{equation}

\begin{proposition}\label{prop:CA-to-vecG}
    If $\phi$ in~\eqref{eq: phi G-action formula} is injective, then $_A\mathcal C = {}_A\mathcal{C}^\mathrm{gr} \cong \mathrm{vec}_G^{
    (\omega,\Omega,1;R)}$ as $\chi$-crossed braided categories, for $\omega$ and $\Omega$ as in~\eqref{eq:omega_def}.
\end{proposition}
\begin{proof}
    By straightforward computation, 
    one can check that the $\widehat{B}$-equivariant $K/B$-graded monoidal functor $(F,\eta)\colon \mathrm{vec}_{K/B} \to {}_A\mathcal C$ defined as in the proof of Proposition~\ref{prop: _AC equiv Vec_G} satisfies~\eqref{crossed braided functor diagram}, i.e.\ $(F,\eta)$ is a $\chi$-crossed braided equivalence.
\end{proof}

Fixing a choice of $\tau$ (and hence of the braiding) as in~\eqref{eq: tau braiding}, we can further equip $\mathrm{vec}_K$ with a ribbon structure as follows. By the isomorphism between classes in the third abelian group cohomology and quadratic forms (see \cite{EM53} and \cite[Thm.\,4.4]{Ga25} -- though for trivial $\omega$ this is also immediate by explicit computation), $q(k) = \tau(k,k)$ is a quadratic form on $K$. It satisfies, for all $k,l \in K$,
\begin{equation}\label{eq:q-ribbon}
\frac{q(k+l)}{q(k)q(l)} = \tau(k,l)\tau(l,k) \ , \qquad q(k) = q(-k) \ .
\end{equation}
These are precisely the ribbon conditions on $\mathrm{vec}_K$, and so 
\begin{equation}
\theta_U \colon U \to U \ , \qquad u_k \mapsto  q(k)\,u_k \ ,
\end{equation}
defines a ribbon twist on $\mathrm{vec}_K$.

Next recall from below \eqref{eq: tau bar-tilde compatible} that $q(b)=\tau(b,b)=1$ for all $b \in B$, and hence $\theta_A = \id_A$. As we have discussed around~\eqref{eq: group algebra coprod}, $A$ is indeed $\Delta$-separable Frobenius. Altogether, we can apply Theorem~\ref{twisted local module crossed ribbon} to get that $_A\mathcal C$ is $\chi$-crossed ribbon
for the crossed module $\chi$ defined in~\eqref{eq:chi-B-def}.
Under the equivalence in Proposition~\ref{prop:CA-to-vecG}, the ribbon twist on $\mathbb{K}_g \in \mathrm{vec}_G^{(\omega,\Omega,1;R)}$ is given by multiplication with $q(s(g))=\tau(s(g),s(g))$, and with the section $s$ fixed as above~\eqref{eq:kappa_def}.

\end{document}